\documentclass{amsart}
\usepackage{amssymb}
\usepackage{mathrsfs}
\usepackage[colorlinks = true]{hyperref}
\usepackage{tikz-cd}

\newcommand{\bk}{\Bbbk}
\newcommand{\Z}{\mathbb{Z}}
\newcommand{\C}{\mathbb{C}}
\newcommand{\K}{\mathbb{K}}
\newcommand{\bE}{\mathbb{E}}

\renewcommand{\O}{\mathbb{O}}
\newcommand{\F}{\mathbb{F}}
\newcommand{\ovK}{{\overline{\K}}}
\newcommand{\scO}{\mathscr{O}}
\newcommand{\Gm}{\mathbb{G}_{\mathrm{m}}}
\newcommand{\sA}{\mathscr{A}}
\newcommand{\St}{\mathrm{St}}

\newcommand{\fz}{\mathfrak{z}}
\newcommand{\fh}{\mathfrak{h}}
\newcommand{\fg}{\mathfrak{g}}
\newcommand{\fb}{\mathfrak{b}}
\newcommand{\fs}{\mathfrak{s}}
\newcommand{\ft}{\mathfrak{t}}
\newcommand{\fn}{\mathfrak{n}}
\newcommand{\cN}{\mathcal{N}}
\newcommand{\tcN}{{\widetilde{\mathcal{N}}}}
\newcommand{\BC}{\mathsf{BC}}
\newcommand{\tG}{\widetilde{G}}
\newcommand{\tB}{\widetilde{B}}
\newcommand{\tT}{\widetilde{T}}
\newcommand{\tfg}{\widetilde{\fg}}

\newcommand{\bX}{\mathbf{X}}
\newcommand{\tbX}{\widetilde{\bX}}

\newcommand{\Bruhat}{{\mathrm{Bru}}}
\newcommand{\dom}{\mathsf{dom}}
\newcommand{\convo}{\mathsf{conv}^0}
\newcommand{\conv}{\mathsf{conv}}
\newcommand{\Wext}{W_{\mathrm{ext}}}
\newcommand{\Bext}{B_{\mathrm{ext}}}

\newcommand{\cF}{\mathcal{F}}
\newcommand{\cG}{\mathcal{G}}
\newcommand{\cL}{\mathcal{L}}
\newcommand{\cO}{\mathcal{O}}
\newcommand{\cV}{\mathcal{V}}

\newcommand{\tDex}{\widetilde{\Delta}^{\mathrm{ex}}}
\newcommand{\tnex}{\widetilde{\nabla}^{\mathrm{ex}}}
\newcommand{\tDred}{\widetilde{\mathcal{W}}^{\mathrm{ex}}}
\newcommand{\tnred}{\widetilde{\mathcal{Y}}^{\mathrm{ex}}}
\newcommand{\tirrex}{\widetilde{\mathcal{L}}}
\newcommand{\cJ}{\mathcal{J}}

\newcommand{\IC}{\mathcal{IC}}
\newcommand{\Dred}{\mathcal{W}^{\mathrm{pc}}}
\newcommand{\nred}{\mathcal{Y}^{\mathrm{pc}}}
\newcommand{\Dpc}{\Delta^{\mathrm{pc}}}
\newcommand{\npc}{\nabla^{\mathrm{pc}}}
\newcommand{\Lpc}{\cL^{\mathrm{pc}}}

\newcommand{\lgmod}{\text{-}\mathsf{gmod}}
\newcommand{\Rep}{\mathsf{Rep}}
\newcommand{\Irr}{\mathsf{Irr}}
\newcommand{\Coh}{\mathsf{Coh}}
\newcommand{\Cohmix}{\mathsf{Coh}^{G \times \mathbb{G_{\mathrm{m}}}}}
\newcommand{\Cohtmix}{\mathsf{Coh}^{\tG \times \mathbb{G_{\mathrm{m}}}}}
\newcommand{\ExCoh}{\mathsf{ExCoh}}
\newcommand{\PCoh}{\mathsf{PCoh}}
\newcommand{\Db}{D^{\mathrm{b}}}

\DeclareMathOperator{\Hom}{Hom}

\DeclareMathOperator{\Tor}{Tor}
\DeclareMathOperator{\End}{End}

\DeclareMathOperator{\supp}{supp}
\DeclareMathOperator{\im}{im}
\DeclareMathOperator{\Spec}{Spec}
\DeclareMathOperator{\codim}{codim}
\newcommand{\sH}{\mathsf{H}}

\newcommand{\ad}{\mathrm{ad}}

\newcommand{\id}{\mathrm{id}}
\newcommand{\simto}{\xrightarrow{\sim}}
\newcommand{\la}{\langle}
\newcommand{\ra}{\rangle}

\makeatletter
\def\lotimes{\@ifnextchar_{\@lotimessub}{\@lotimesnosub}}
\def\@lotimessub_#1{\mathchoice{\mathbin{\mathop{\otimes}^L}_{#1}}%
  {\otimes^L_{#1}}{\otimes^L_{#1}}{\otimes^L_{#1}}}
\def\@lotimesnosub{\mathbin{\mathop{\otimes}^L}}
\makeatother

\numberwithin{equation}{section}
\newtheorem{thm}{Theorem}[section]
\newtheorem{lem}[thm]{Lemma}
\newtheorem{prop}[thm]{Proposition}
\newtheorem{cor}[thm]{Corollary}

\theoremstyle{definition}
\newtheorem{defn}[thm]{Definition}

\theoremstyle{remark}
\newtheorem{rmk}[thm]{Remark}

\title[Integral exotic sheaves and the modular Lusztig--Vogan bijection]{Integral exotic sheaves and the modular Lusztig--Vogan bijection}

 \author{Pramod N. Achar}
 \address{Department of Mathematics\\
   Louisiana State University\\
   Baton Rouge, LA 70803\\
   U.S.A.}
 \email{pramod@math.lsu.edu}
 
  \author{William Hardesty}
   \address{School of Mathematics and Statisitics\\
   University of Sydney\\
   Camperdown NSW 2006\\
   Australia}
  \email{william.hardesty@sydney.edu.au}
 
 \author{Simon Riche}
 \address{Universit\'e Clermont Auvergne, CNRS, LMBP, F-63000 Clermont-Ferrand, France.
 }
 \email{simon.riche@uca.fr}
 
\thanks{P.A. was supported by NSF Grant Nos.~DMS-1500890 and DMS-1802241. This project has received funding from
the European Research Council (ERC) under the European Union's Horizon 2020 research and
innovation programme (grant agreements No.~677147 and~101002592).}

\begin{document}

\begin{abstract}
Let $G$ be a reductive group over an algebraically closed field $\bk$ of pretty good characteristic.  The \emph{Lusztig--Vogan bijection} is a bijection between the set of dominant weights for $G$ and the set of irreducible $G$-equivariant vector bundles on nilpotent orbits, conjectured by Lusztig and Vogan independently, and constructed in full generality by Bezrukavnikov.  In characteristic $0$, this bijection is related to the theory of $2$-sided cells in the affine Weyl group, and plays a key role in the proof of the Humphreys conjecture on support varieties of tilting modules for quantum groups at a root of unity.

In this paper, we prove that the Lusztig--Vogan bijection is (in a way made precise in the body of the paper) independent of the characteristic of $\bk$. This allows us to extend all of its known properties from the characteristic-$0$ setting to the general case.  We also expect 
this result to be a step towards a proof of the Humphreys conjecture on support varieties of tilting modules for reductive groups in positive characteristic.
\end{abstract}

\maketitle

\section{Introduction}

\subsection{The Lusztig--Vogan bijection}
\label{ss:LV-bijection-intro}

Let $\bk$ be an algebraically closed field, and let $G_\bk$ be a connected reductive algebraic group over $\bk$.
Assume that the characteristic of $\bk$ is ``pretty good'' for $G_\bk$ (see Section~\ref{sec:notation} for the definition).  Let $\bX^+$ be the set of dominant weights for $G_\bk$, and let $\cN_\bk$ denote its nilpotent cone.  Let
\begin{align*}
\Omega_\bk &:= 
\left\{
(\scO,\cV) \,\Big|\,
\begin{array}{@{}c@{}}
\text{$\scO \subset \cN_\bk$ is a nilpotent orbit, and} \\
\text{$\cV$ is an irreducible $G_\bk$-equivariant vector bundle on $\scO$}
\end{array}
\right\} \\
&\cong
\{
(x,V) \mid
\text{$x \in \cN_\bk$ and $V \in \Irr(Z_{G_\bk}(x))$}\}/(\text{$G_\bk$-conjugacy}).
\end{align*}
In the second line above, $Z_{G_\bk}(x) \subset G_\bk$ is the stabilizer of $x$ for the adjoint action, and $\Irr(Z_{G_\bk}(x))$ is the set of isomorphism classes of irreducible $Z_{G_\bk}(x)$-representa\-tions. The \emph{Lusztig--Vogan bijection} for $G_\bk$ is a certain natural bijection
\begin{equation}\label{eqn:lvbij}
\bX^+ \overset{\sim}{\leftrightarrow}
\Omega_\bk.
\end{equation}

For $\bk = \C$, the existence of this bijection was conjectured independently by Lusztig~\cite[\S 10.8]{lusztig} and, in a rather different framework, by Vogan~\cite[Lecture~8]{vogan}.  Its existence was proved in~\cite{achar-phd} for $G_\C = \mathrm{GL}_n(\C)$ and in~\cite{bez} for arbitrary $G_\C$.  In positive characteristic, the existence of the bijection is a consequence of~\cite[Theorem~6.2]{achar-pcoh} (cf.~\cite[Corollaries~3 and~4]{bez}).

This bijection is much better understood in the case $\bk=\C$. For instance, in this case the bijection was described explicitly in~\cite{achar-RT} in the case $G_\C = \mathrm{GL}_n(\C)$. Moreover, it was proved by Bezrukavnikov that this bijection is compatible with Lusztig's bijection between two-sided cells in affine Weyl groups and nilpotent orbits~\cite{lusztig}, in the sense that the first component of the pair attached to a dominant weight $\lambda$ is the orbit attached to the cell containing the minimal length representative in the double coset of $\lambda$ for the finite Weyl group.\footnote{This fact is not emphasized very explicitly in Bezrukavnikov's paper. In fact it follows from the agreement of the bijection constructed in~\cite{bez} with another such bijection, constructed using different methods in~\cite{bez-cells}, and for which this property is obvious from construction. This agreement is justified in~\cite[Remark~6]{bez-perv}.} On the other hand, essentially nothing is known about the bijection in the case when $\mathrm{char}(\bk)>0$.

\subsection{Independence of \texorpdfstring{$\bk$}{k}: orbits}

The goal of this paper is to show that the bijection~\eqref{eqn:lvbij} is ``independent of $\bk$.'' To make sense of this statement, we must first explain how to identify the various sets $\Omega_\bk$ as $\bk$ varies. (Of course, the left-hand side of~\eqref{eqn:lvbij} depends only on the root datum.)  This requires working with an integral version of our group.  Let $\O$ be a complete discrete valuation ring with residue field $\F$ (algebraically closed, of characteristic $p > 0$) and fraction field $\K$ (of characteristic $0$).  Let $G$ be a split connected reductive group over $\O$,
and assume that $p$ is pretty good for $G$. We choose an algebraic closure $\ovK$ of $\K$, and denote by $G_\F$, resp.~$G_\ovK$, the base change of $G$ to $\F$, resp.~$\ovK$.

As a first step, the Bala--Carter theorem parametrizes nilpotent orbits using only information from the root datum of $G$, so it gives us a canonical bijection
\begin{equation}\label{eqn:bc-intro}
\{\text{nilpotent orbits for $G_\ovK$}\} \simto
\{\text{nilpotent orbits for $G_\F$}\},
\end{equation}
which will be denoted $\BC$.

As a first compatibility property, one may wonder whether the nilpotent orbits $\scO_\lambda^\ovK$ and $\scO_\lambda^\F$ attached to a dominant weight $\lambda$ by the Lusztig--Vogan bijections for $G_\F$ and $G_\ovK$ match under this bijection. We prove that this fact indeed holds.

\begin{thm}
\label{thm:main-intro-orbits}
For any $\lambda \in \bX^+$, we have $\scO_\lambda^\F = \BC(\scO_\lambda^\ovK)$.
\end{thm}

\subsection{Independence of \texorpdfstring{$\bk$}{k}: representations of stabilizers}

To go further in the comparison of the bijections, we must compare representations of centralizers of nilpotent elements over the two fields.  To do this, we will work with McNinch's notion of \emph{balanced nilpotent sections}, which are certain well-behaved nilpotent elements in the Lie algebra of $G$ (over $\O$). We denote by $x$ such a section, and by $x_\F$, $x_\ovK$ the nilpotent elements in the Lie algebras of $G_\F$ and $G_\ovK$ obtained from $x$. We will also denote by $Z_G(x)$ the centralizer of $x$ in $G$; then the base change $Z_G(x)_\F$ of $Z_G(x)$ to $\F$ is the scheme-theoretic centralizer of $x_\F$ in $G_\F$ and the base change $Z_G(x)_\ovK$ of $Z_G(x)$ to $\ovK$ is the scheme-theoretic centralizer of $x_\ovK$ in $G_\ovK$.

The second author has shown~\cite{har:smcent} that $Z_G(x)$ is a smooth group scheme over $\O$.  In the companion paper~\cite{ahr-disconn}, we study the representation theory of disconnected reductive groups, and we use this study here to establish the following result.

\begin{prop}
\label{prop:intro-decomp}
Let $x$ be a balanced nilpotent section.  There exists a canonical isomorphism of Grothendieck groups
\[
d: \mathsf{K}(Z_{G}(x)_\ovK) \overset{\sim}{\to} \mathsf{K}(Z_{G}(x)_\F).
\]
Moreover, the change-of-basis matrix relating the basis indexed by $\Irr(Z_{G}(x)_\ovK)$ to that indexed by $\Irr(Z_{G}(x)_\F)$ is upper-triangular, so there exists a canonical bijection
\[
\Irr(Z_{G}(x)_\ovK) \overset{\sim}{\leftrightarrow} \Irr(Z_{G}(x)_\F).
\]
\end{prop}

Implicit in this statement is the assertion that $\Irr(Z_{G}(x)_\ovK)$ and $\Irr(Z_{G}(x)_\F)$ are (partially) ordered, so that ``upper-triangular'' makes sense; in fact, the results of~\cite{ahr-disconn} show that the category of representations of the reductive quotient of $Z_G(x)_\F$ admits a natural highest-weight structure. The map $d$ in Proposition~\ref{prop:intro-decomp} is a ``decomposition map'' in the sense of Serre~\cite{serre-Groth}; it sends irreducible $Z_{G}(x)_\ovK$-modules to ``Weyl modules'' (i.e.~standard objects) for $Z_{G}(x)_\F$.

Since balanced nilpotent sections exist for each nilpotent $G_\F$-orbit
by results of McNinch~\cite{mcninch2016}, combining~\eqref{eqn:bc-intro} and Proposition~\ref{prop:intro-decomp}, we obtain a canonical bijection
\begin{equation}\label{eqn:omega-bij}
\Omega_\ovK \overset{\sim}{\leftrightarrow} \Omega_\F.
\end{equation}
The main result of the paper is the following.

\begin{thm}
\label{thm:intro-main}
Let $G$ be a split connected reductive group over $\O$.  The following diagram commutes:
\[
\begin{tikzcd}[column sep=14.4em, row sep=0em]
& \Omega_\ovK \ar[dd, leftrightarrow, "\eqref{eqn:omega-bij}", "\wr"'] \\
\bX^+ \ar[ur, bend left=10, "\text{\normalfont Lusztig--Vogan bijection~\eqref{eqn:lvbij} for $G_\ovK$}"{above, pos=0.6}, "\sim"']
  \ar[dr, bend right=10, "\text{\normalfont Lusztig--Vogan bijection~\eqref{eqn:lvbij} for $G_\F$}"{below, pos=0.6}, "\sim"] \\
& \Omega_\F.
\end{tikzcd}
\]
\end{thm}

Note that this result in particular subsumes Theorem~\ref{thm:main-intro-orbits}.

\subsection{Modular reduction}

To prove Theorem~\ref{thm:intro-main}, we must dig into the construction of the bijection~\eqref{eqn:lvbij}.  This involves the notion of \emph{perverse-coherent sheaves} on the nilpotent cones $\cN_\ovK$ and $\cN_\F$ of $G_\ovK$ and $G_\F$.  To compare the two, we again need some intermediary that lives over $\O$.  It may be possible to work directly with ``perverse-coherent sheaves'' on an $\O$-form of the nilpotent cone, but this presents certain technical challenges, and this is not the approach we take in the present paper.  Instead, we work with the closely related notion of \emph{exotic sheaves} on the Springer resolution $\tcN$, in part because the foundations needed to define and study them over $\O$ are already available~\cite{br}, and also because the structure of the corresponding t-structure is much more rigid than that of the perverse-coherent t-structure (in fact it is defined by an exceptional sequence, and its heart is a highest weight category).

The technique of using sheaves with coefficients in a local ring is quite classical in the setting of constructible sheaves. Here we have to work with \emph{coherent} sheaves, and even though this technique also makes sense in this context, it turns out to be much harder to use. In particular, there is no elementary analogue over $\O$ of the ``stratification by $G$-orbits'' for $\cN_\ovK$ or $\cN_\F$.  One can study the (coherent) pullback of a complex of coherent sheaves to a point in the nilpotent cone, but this operation is not exact, so it is more difficult to use. To overcome this difficulty, we introduce and work with ``integral versions'' of Slodowy slices, which might be of independent interest.

\subsection{Reduced standard objects}

By definition, the closure of the orbit attached to $\lambda$ by the Lusztig--Vogan bijection is the support of a certain simple perverse-coherent sheaf on the nilpotent cone. Every such simple object is a quotient of a certain ``standard'' object---but this is not useful for studying supports, because all standard objects have full support in the nilpotent cone. Instead, we need to construct some ``smaller'' objects which still surject onto our simple perverse-coherent sheaves. These objects, which we call ``reduced standard objects,'' are obtained by ``modular reduction'' from the corresponding simple perverse-coherent sheaf over $\ovK$. (The idea of such a construction goes back at least to work of Cline--Parshall--Scott~\cite{cps}.) These objects are different from the corresponding simple perverse-coherent sheaf, but as a step towards Theorem~\ref{thm:intro-main} we prove that the support of the reduced standard object associated with $\lambda$ is $\BC(\scO_\lambda^\ovK)$.

\subsection{Motivation}

Our main motivation for considering the problem studied in the present paper comes from the \emph{Humphreys conjecture} on support varieties of indecomposable tilting modules for $G_\F$~\cite{humphreys}. This conjecture predicts that these support varieties are closures of nilpotent $G_\F$-orbits determined by Lusztig's bijection between two-sided cells and nilpotent orbits. This conjecture was proved in type $\mathbf{A}$ by the second author~\cite{hardesty}, and in~\cite{ahr}, adapting some constructions of Bezrukavnikov (in the analogous setting of quantum groups at a root of unity) we proved it in large characteristic (without an explicit bound); however outside of type $\mathbf{A}$ this conjecture is still open for ``reasonable'' prime characteristics.\footnote{After this paper appeared in preprint form, a proof of the relative version of the Humphreys conjecture for $p$ larger than the Coxeter number was obtained by the first two authors~\cite{ah2}.}  A priori, to a dominant weight $\lambda$ one can attach two nilpotent $G_\F$-orbits which could describe support varieties of tilting modules with a highest weight attached to $\lambda$:
either $\scO_\lambda^\F$ or $\BC(\scO_\lambda^\ovK)$. The orbit that appears ``naturally'' in this problem (at least via the methods of~\cite{ahr}) is $\scO_\lambda^\F$, but the one used by Humphreys is $\BC(\scO_\lambda^\ovK)$. This distinction did not cause any trouble in~\cite{ahr} because for other reasons we had to restrict to large characteristics, but it would play a role in any attempt to prove this conjecture for smaller characteristics. Theorem~\ref{thm:main-intro-orbits} shows that these orbits coincide, thus solving this discrepancy.

In~\cite{ahr-conj} we propose some conjectures\footnote{Some of these conjectures have now been proved in~\cite{ah2}.} which aim at putting the Humphreys conjecture in a larger picture, in a hope of clarifying its significance. These conjectures involve only the orbits $\scO_\lambda^\F$, and their compatibility with the Humphreys conjecture is guaranteed by the results of the present paper.

\subsection{Contents of the paper}

We begin in Section~\ref{sec:notation} with generalities on reductive groups, the nilpotent cone, and the Springer resolution.  This section also reviews the construction of the Lusztig--Vogan bijection~\eqref{eqn:lvbij} in terms of simple perverse-coherent sheaves.  In Section~\ref{sec:balanced}, we introduce and study an $\O$-analogue of a Slodowy slice to a nilpotent orbit.  This will serve as an important technical tool later in the paper.

In Section~\ref{sec:intexotic}, we define the exotic t-structure and record some of its basic properties. (This construction builds on general results on exceptional collections defined over complete discrete valuation rings, proved in Appendix~\ref{sec:exc-pid}.) As an application, in Section~\ref{sec:reduced}, we use exotic sheaves to define a new class of perverse-coherent sheaves over $\F$, called \emph{reduced standard objects}.   The main result of Section~\ref{sec:reduced} is a kind of first approximation to Theorem~\ref{thm:main-intro-orbits}: it asserts that the supports of reduced standard objects correspond via~\eqref{eqn:bc-intro} to nilpotent orbits arising from the Lusztig--Vogan bijection for $\ovK$. Finally, in Section~\ref{sec:proof-main} we prove Theorem~\ref{thm:intro-main}.

\subsection{Convention}

At various points in the paper we consider certain schemes and affine group schemes that could be defined over various base rings. To avoid confusion we use subscripts to specify the base ring of schemes. (However we do not use subscripts for morphisms, since the ring under consideration is always clear from context.) In order to avoid notational clutter, we will sometimes affix a single subscript $\bk$ to some constructions like products or categories of equivariant coherent sheaves, writing e.g.~$\Coh^G(X)_\bk$ instead of $\Coh^{G_\bk}(X_\bk)$.

\subsection{Acknowledgments}

P.A. is grateful to David Vogan for suggesting a thesis problem in 1998 that is still keeping him busy 20 years later!

\section{Notation and preliminaries}
\label{sec:notation}

\subsection{Reductive groups}

Let $\O$ be a complete discrete valuation ring with residue field $\F$, and denote by $\K$ the fraction field of $\O$. We assume throughout that $\F$ is algebraically closed, of characteristic $p>0$, and that $\K$ has characteristic $0$. We also fix an algebraic closure $\ovK$ of $\K$.  

Let $G$ be a split connected reductive group over $\O$, and let $T \subset G$ be a split maximal torus. This group corresponds to some root datum $(\bX, \bX^\vee, R, R^\vee)$, where $\bX$ is the character lattice of $T$, and $\bX^\vee$ is the cocharacter lattice.  Let $W$ be the Weyl group of $R$ (or of $G$), and let $w_0 \in W$ be its longest element.


Choose, once and for all, a positive subsystem $R^+ \subset R$.  Let $B \subset G$ be the subgroup generated by the maximal torus $T$ and by the root subgroups corresponding to the \emph{negative} roots.  Denote the Lie algebras of $G$, $B$, and $T$ by $\fg$, $\fb$, and $\ft$, respectively.  As an $\O$-module, $\fb$ decomposes as
\[
\fb = \ft \oplus \fn
\qquad\text{where}\qquad
\fn = \bigoplus_{\alpha \in -R^+} \fg_\alpha.
\]

We set
\[
G_\F := \Spec(\F) \times_{\Spec(\O)} G, \quad G_\K := \Spec(\K) \times_{\Spec(\O)} G, \quad G_\ovK := \Spec(\ovK) \times_{\Spec(\O)} G.
\]
Their Lie algebras  can be described as
\[
\fg_\F = \F \otimes_\O \fg, \quad \fg_\K = \K \otimes_\O \fg, \quad \fg_\ovK = \ovK \otimes_\O \fg
\]
respectively, see e.g.~\cite[\S 2.1]{riche-kostant} for details.  Similar notation is used for the groups or Lie algebras obtained from $B$ or $T$ by change of scalars.  

We assume throughout the paper that $p$ is \emph{pretty good} for $G$ in the sense of~\cite[Definition~2.11]{herpel}.  This means that $p$ is good for $G$, and in addition, the abelian groups $\bX/\Z R$ and $\bX^\vee/ \Z R^\vee$ have no $p$-torsion.  By~\cite[Theorem~5.2]{herpel}, this assumption is equivalent to requiring $G_\F$ to be \emph{standard} in the sense of~\cite[\S 4]{mt}.  The following fact is probably known to experts.

\begin{lem}\label{lem:non-degenerate}
There is a $G$-invariant perfect pairing $(-,-): \fg \times \fg \to \O$.
\end{lem}
\begin{proof}
The analogous claim for $\fg_\F$ is proved in~\cite[Proposition~12]{mt1} (see also~\cite[Remark~4.4]{mt}).  Let us briefly review how this proof goes.  (In this paragraph only, we let $G$ denote a connected reductive group over $\F$.). One first treats the easy case in which $G$ is of the form
\begin{equation}\label{eqn:nondeg-easy}
G = S \times H
\qquad\text{where}\qquad
\begin{cases}
\text{$S$ is a torus, and} \\
\text{$H$ is a semisimple, simply-connected group} \\
\text{\qquad for which $p$ is very good.}
\end{cases}
\end{equation}
Next, according to~\cite[Theorem~5.2]{herpel}, any connected reductive group over $\F$ for which $p$ is pretty good can be obtained from a group as in~\eqref{eqn:nondeg-easy} by a finite sequence of the following operations:
\begin{enumerate}
\item Replace $G$ by a Levi subgroup $G'$.
\item 
\label{it:separable-isogeny}
Replace $G$ by a group $G'$ such that there is a separable isogeny $G \to G'$ or $G' \to G$.
\item Replace $G$ by a group $G'$ such that $G$ is isomorphic to a product $S \times G'$ with $S$ a torus.
\end{enumerate}
To finish the proof, one shows that in each of these three operations, if the Lie algebra of $G$ admits an invariant nondegenerate bilinear form, then the same holds for $G'$.

We now return to the setting of the lemma, in which $G$ is a split reductive group over $\O$.  Observe that each of the three operations above also makes sense in the context of split reductive groups over $\O$. (For~\eqref{it:separable-isogeny}, we require that the base change to $\F$ is separable.) Since isomorphism classes of split reductive groups over $\F$ and $\O$ correspond via base change, we can therefore follow the same strategy as above, starting with the easy case in~\eqref{eqn:nondeg-easy}.

\textit{Step 0. Proof for $G$ of the form~\eqref{eqn:nondeg-easy}}.  In this case $G$ is the product of a torus and a number of quasisimple, simply connected groups.  It is enough to treat these factors separately.  Any perfect pairing on the Lie algebra of the torus factor $S$ is $S$- (and hence $G$-) invariant.  Now let $H_1$ be a quasisimple direct factor of $H$, and let $\fh_1$ be its Lie algebra.  If $H_1$ is of exceptional type, then it follows from~\cite[\S\S I.4.8--I.4.9]{spring-stein} that the Killing form on $\fh_1$ is a perfect pairing.  If $H_1$ is of classical type, then it admits a ``defining representation'' $H_1 \to GL_n(\O)$.  Consider the induced map
\[
\rho: \fh_1 \to \mathfrak{gl}_n(\O)
\]
and equip $\fh_1$ with the pairing $(X,Y) \mapsto \mathrm{tr}(\rho(X)\rho(Y))$.  This pairing induces nondegenerate pairings over both $\K$ and $\F$ (see~\cite[\S I.5.3]{spring-stein} and~\cite[Corollary~2.5.8 and Proposition~2.5.10]{letellier}), so it must be a perfect pairing over $\O$.

\textit{Step 1. If the lemma holds for $G$, it holds for any Levi subgroup of $G$.} Let $G'$ be a Levi subgroup of $G$ containing $T$.  Then $G'$ is the centralizer of a subtorus $S \subset T$: namely, $S$ is the subtorus whose cocharacter lattice is the sublattice of $\bX^\vee$ consisting of elements annihilated by all roots of $G'$.  Under the adjoint action of $S$ on $\fg$, the latter decomposes as
\[
\fg = \bigoplus_{\lambda \in X^*(S)} \fg_\lambda,
\]
where $X^*(S)$ is the character lattice of $S$ (a quotient of $\bX$).  Any $G$-invariant perfect pairing on $\fg$ restricts to a perfect pairing $\fg_\lambda \times \fg_{-\lambda} \to \O$.  In particular, it gives a $G'$-invariant perfect pairing $\fg_0 \times \fg_0 \to \O$.  Since $\fg_0$ is the Lie algebra of $G'$, we are done.

\textit{Step 2. If the lemma holds for $G$, it holds for any group related to $G$ by a central isogeny whose reduction modulo $p$ is separable.}  Suppose we have a central isogeny $G \to G'$ whose reduction modulo $p$ is separable.  Then the corresponding map on Lie algebras $\fg_\F \to \fg'_\F$ is an isomorphism.  Since $\fg_\K \to \fg'_\K$ is also an isomorphism, we conclude that $\fg \to \fg'$ is an isomorphism.  Any invariant perfect pairing on $\fg$ can be transferred to $\fg'$ via this isomorphism.  The same reasoning applies if we have a central isogeny in the opposite direction, $G' \to G$.

\textit{Step 3.  If the lemma holds for $G = S \times G'$, where $S$ is a split torus and $G'$ is split reductive, then then it holds for $G'$.}  Let $\fs'$ and $\fg'$ denote the Lie algebras of $S$ and $G'$, respectively.  The existence of a $G$-invariant perfect pairing on $\fg$ implies that there is a $G$-equivariant (and hence $G'$-equivariant) isomorphism between $\fg \cong \fs \oplus \fg'$ and $\fg^* \cong \fs^* \oplus (\fg')^*$.  (Here $\fg^* = \Hom_\O(\fg,\O)$.) Now, $\fs$ and $\fs^*$ are both trivial $G'$-modules that are isomorphic as $\O$-modules.  They are therefore isomorphic as $G'$-modules.

In the category of $G'$-modules that are finitely generated over $\O$, the endomorphism ring of any object is a finitely generated $\O$-module.  Since $\O$ is a complete local ring, it follows that every such endomorphism ring is semiperfect, by~\cite[Example~23.3]{lam}.  Then, by~\cite[Theorem~A.1]{cyz}, the Krull--Schmidt theorem holds in this category.  Since we have $G'$-equivariant isomorphisms $\fs \oplus \fg' \cong \fs^* \oplus (\fg')^*$ and $\fs \cong \fs^*$, we conclude that there is a $G'$-equivariant isomorphism $\fg' \cong (\fg')^*$.  This is equivalent to the existence of a $G'$-invariant perfect pairing on $\fg'$.
\end{proof}

\begin{rmk}
\begin{enumerate}
\item In Step 3 of the preceding proof, the invariant perfect pairing on $\fg'$ is \emph{not} obtained by restricting the invariant pairing on $\fg$.  Indeed, in general, the restriction to $\fg'$ of a perfect pairing on $\fg$ may fail to be perfect.

\item In many examples, one can find a \emph{symmetric} $G$-invariant perfect pairing (cf.~\cite[Remark~4]{mcn1}), but we do not know if this holds for all reductive groups.  The difficulty lies essentially in the lack of an explicit construction in Step~3 of the preceding proof.  We thank G.~McNinch for helpful comments on this point.
\end{enumerate}
\end{rmk}

\begin{lem}
\label{lem:derived-isogeny}
There exists a central isogeny $\sigma: \tG \to G$ such that the following conditions hold:
\begin{enumerate}
\item The derived subgroup of $\tG$ is simply connected.
\item $p$ is pretty good for $\tG$, i.e.~the quotient of the cocharacter lattice of a maximal torus of $\tG$ by the root lattice has no $p$-torsion.
\item The map $\sigma: \tG_\F \to G_\F$ is separable.
\end{enumerate}
Moreover, $\sigma$ identifies the Lie algebra of $\tG$ with $\fg$.
\end{lem}
\begin{proof}
Consider the torsion subgroup $(\bX^\vee/\Z R^\vee)_{\mathrm{tors}} \subset \bX^\vee/\Z R^\vee$.  Choose a complement $F$, so that $\bX^\vee/\Z R^\vee \cong (\bX^\vee/\Z R^\vee)_{\mathrm{tors}} \oplus F$.  Thus, $F$ is a maximal free summand of $\bX^\vee/\Z R^\vee$.  Let $\tbX^\vee \subset  \bX^\vee$ be the preimage of $F$ under the quotient map $\bX^\vee \to \bX^\vee/\Z R^\vee$.

Let $\tbX = \Hom_\Z(\tbX^\vee,\Z)$.  The group $\bX$ is naturally identified with a subgroup of $\tbX$.  The quadruple $(\tbX, \tbX^\vee, R, R^\vee)$ is a root datum.  Let $\tG$ be the corresponding split reductive group over $\O$.  The obvious morphism of root data gives rise to a central isogeny $\tG \to G$.  Because $\tbX^\vee / \Z R^\vee$ is torsion-free, the derived subgroup of $\tG$ is simply connected.  Since $p$ does not divide the order of the finite group $\tbX / \bX$ (which is isomorphic to $(\bX^\vee/\Z R^\vee)_{\mathrm{tors}}$), the map $\sigma: \tG_\F \to G_\F$ is separable.

For the last assertion, let $\tfg$ be the Lie algebra of $\tG$.  Since the induced maps $\tfg_\K \to \fg_\K$ and $\tfg_\F \to \fg_\F$ are isomorphisms, the map $\tfg \to \fg$ is as well.
\end{proof}

\subsection{The Springer resolution}
\label{ss:Springer-reolution}

Let
\[
\tcN = G \times^B \fn.
\]
The group $G$ acts on this scheme in the obvious way, so we may consider the bounded derived category
\[
\Db\Coh^G(\tcN)
\]
of $G$-equivariant coherent sheaves on $\tcN$.  

\begin{lem}
\label{lem:tcn-alt}
There exists a $G$-equivariant isomorphism $\tcN \cong G \times^B (\fg/\fb)^*$.
\end{lem}

\begin{proof}
Choose a $G$-equivariant perfect pairing $({-},{-}): \fg \times \fg \to \O$ as in Lemma~\ref{lem:non-degenerate}.  This induces a $B$-equivariant isomorphism $\fn \simto (\fg/\fb)^*$.
\end{proof}

Let
\[
\tilde\pi: \tcN \to \fg
\]
be the map given by $\tilde\pi(g,x) = \mathrm{Ad}(g)(x)$. This map is proper, and gives rise to a functor
\[
\tilde{\pi}_* : \Db\Coh^G(\tcN) \to \Db\Coh^G(\fg).
\]
There is also an obvious map $p: \tcN \to G/B$.  Any $B$-representation that is finitely generated over $\O$ gives rise to a $G$-equivariant coherent sheaf on $G/B$ (see, for instance,~\cite[\S\S I.5.8--I.5.9]{jantzen}).  In particular, any $\lambda \in \bX$ defines a $B$-module structure on the free rank-$1$ $\O$-module. The corresponding (invertible) sheaf on $G/B$ will be denoted $\cO_{G/B}(\lambda)$, and we set
\[
\cO_\tcN(\lambda) := p^* \cO_{G/B}(\lambda).
\]

We denote by $\tcN_\F$, $\tcN_\K$, and $\tcN_\ovK$ the schemes obtained from $\tcN$ by change of scalars from $\O$ to $\F$, $\K$, or $\ovK$, respectively. We can then consider the corresponding derived categories of coherent sheaves $\Db\Coh^G(\tcN)_\F$, $\Db\Coh^G(\tcN)_\K$ and $\Db\Coh^G(\tcN)_\ovK$, and the functors $\tilde\pi_*$.
The change-of-scalars functors for coherent sheaves on $\tcN$ will be denoted by
\begin{align*}
\F({-}) &: \Db\Coh^G(\tcN) \to \Db\Coh^G(\tcN)_\F, \\
\K({-}) &: \Db\Coh^G(\tcN) \to \Db\Coh^G(\tcN)_\K, \\
\ovK({-}) &: \Db\Coh^G(\tcN) \to \Db\Coh^G(\tcN)_\ovK.
\end{align*}
These functors commute with the functors $\tilde\pi_*$, in the sense that there exist canonical isomorphisms of functors
\begin{equation}\label{eqn:tildepi-scalar-commute}
\F \circ \tilde\pi_* \cong \tilde\pi_* \circ \F, \qquad \K \circ \tilde\pi_* \cong \tilde\pi_* \circ \K, \qquad \ovK \circ \tilde\pi_* \cong \tilde\pi_* \circ \ovK.
\end{equation}

\subsection{The nilpotent cone and perverse-coherent sheaves}
\label{ss:pcoh}

In this subsection we fix $\bk \in \{\F,\ovK\}$, and let $\cN_\bk$ denote the variety of nilpotent elements in $\fg_\bk$.  (There are subtleties involved in finding the correct definition of the nilpotent scheme over $\O$.  We will not address those here, and we will work with the nilpotent variety only over an algebraically closed field, where the results of~\cite{jantzen-nilp} are available.)  There are maps
\[
\pi:\tcN_\F \to \cN_\F,
\qquad
\pi:\tcN_\ovK \to \cN_\ovK,
\]
both given by $\pi(g,x) = \mathrm{Ad}(g)(x)$.  Note that for $\bk \in \{\F,\ovK\}$ the map $\tilde\pi$ defined in~\S\ref{ss:Springer-reolution} factors as $\tilde\pi = i \circ \pi$, where $i: \cN_\bk \hookrightarrow \fg_\bk$ is the inclusion map.

The bounded derived category of $G_\bk$-equivariant coherent sheaves on $\cN_\bk$ will be denoted by
\[
\Db\Coh^G(\cN)_\bk.
\]
Since $\pi$ is proper, it gives rise to a functor
\[
\pi_*: \Db\Coh^G(\tcN)_\bk \to \Db\Coh^G(\cN)_\bk.
\]
For $\lambda \in \bX^+$, set
\begin{equation}\label{eqn:dpc-defn}
\Dpc_\lambda(\bk) := \pi_*\cO_{\tcN}(w_0\lambda)
\qquad\text{and}\qquad
\npc_\lambda(\bk) := \pi_*\cO_\tcN(\lambda).
\end{equation}
Let $\scO \subset \cN_\bk$ be a $G_\bk$-orbit.  We define the \emph{star} of this orbit to be the open subvariety
\begin{equation}\label{eqn:star-defn}
\St(\scO) = \bigcup_{\substack{\text{$\scO' \subset \cN_\bk$ a $G_\bk$-orbit} \\ \scO \subset \overline{\scO'}}} \scO'.
\end{equation}
Let $i_\scO: \scO \hookrightarrow \St(\scO)$ be the embedding of $\scO$ as a (reduced) closed subscheme of its star.

The category $\Db\Coh^G(\cN)_\bk$ is equipped with a remarkable t-structure called the \emph{perverse-coherent t-structure}.  Its heart is denoted by
\[
\PCoh(\cN)_\bk,
\]
and objects in the heart are called \emph{perverse-coherent sheaves}.
We will not recall the definition of the perverse-coherent t-structure in detail here, but we will review some of its key properties below, following~\cite{achar-pcoh, achar, ab:pcs, bez}.

\begin{rmk}\label{rmk:prettygood-pcoh}
The results on perverse-coherent sheaves in~\cite{achar-pcoh, achar} are stated under the assumption that $G_\F$ has a simply connected derived subgroup. However, using the separable isogeny from Lemma~\ref{lem:derived-isogeny}, it is straightforward to transfer these results to arbitrary groups in pretty good characteristic.
\end{rmk}

Recall the definition of the \emph{support} $\supp(\cF)$ of a complex $\cF$ of coherent sheaves on a scheme, see e.g.~\cite[\S 4.1]{ahr} for references; this support is a closed subset of the underlying topological space of the given scheme. We will also consider the order $\leq$ on $\bX$ such that $\lambda < \mu$ iff $\mu - \lambda$ is a sum of positive roots.
Finally, as in~\S\ref{ss:LV-bijection-intro} we set
\[
\Omega_\bk := 
\left\{
(\scO,\cV) \,\Big|\,
\begin{array}{@{}c@{}}
\text{$\scO \subset \cN_\bk$ is a nilpotent orbit, and} \\
\text{$\cV$ is an irreducible $G_\bk$-equivariant vector bundle on $\scO$}
\end{array}
\right\}.
\]
Then the following properties hold:
\begin{enumerate}
\item The perverse-coherent t-structure is bounded, and every object in the heart $\PCoh(\cN)_\bk$ has finite length.\label{it:pcoh-bounded}
\item If $\scO \subset \cN$ is an orbit that is open in the support of $\cF \in \PCoh(\cN)_\bk$, then $\sH^i(\cF)|_{\St(\scO)}$ vanishes for $i \ne \frac{1}{2}\codim \scO$.\label{it:pcoh-support}
\item The objects $\Dpc_\lambda(\bk)$ and $\npc_\lambda(\bk)$ lie in $\PCoh(\cN)$.  Moreover,\label{it:pcoh-qhered}
\[
\Hom(\Dpc_\lambda(\bk), \npc_\mu(\bk)) \cong
\begin{cases}
\bk & \text{if $\lambda = \mu$,} \\
0 & \text{otherwise.}
\end{cases}
\]
\item 
\label{it:simples-pc}
Fix a nonzero map $c_\lambda: \Dpc_\lambda(\bk) \to \npc_\lambda(\bk)$, and set\label{it:irr-x}
\[
\Lpc_\lambda(\bk) = \im(c_\lambda: \Dpc_\lambda(\bk) \to \npc_\lambda(\bk)).
\]
Then $\Lpc_\lambda(\bk)$ is a simple object in $\PCoh(\cN)$.  Moreover, every simple object is isomorphic to $\Lpc_\lambda(\bk)$ for a unique $\lambda \in \bX^+$, and each composition factor of the kernel of the surjection $\Dpc_\lambda(\bk) \twoheadrightarrow \Lpc_\lambda(\bk)$, resp.~of the cokernel of the embedding $\Lpc_\lambda(\bk) \hookrightarrow \npc_\lambda(\bk)$, is of the form $\Lpc_\mu(\bk)$ with $\mu < \lambda$.
\item Let $\scO \subset \cN_\bk$ be a nilpotent orbit, and let $\cV$ be an irreducible $G_{\bk}$-equivariant vector bundle on $\scO$.  There is unique simple perverse-coherent sheaf\label{it:irr-orbit}
\[
\IC(\scO,\cV)
\]
that is characterized by the following properties: it is supported on $\overline{\scO}$, and $\sH^{\frac{1}{2}\codim\scO}(\IC(\scO,\cV))|_{\St(\scO)} \cong i_{\scO*}\cV$.  Moreover, every simple object is isomorphic to $\IC(\scO,\cV)$ for a unique pair $(\scO,\cV) \in \Omega_\bk$.
\end{enumerate}

Here, items~\eqref{it:pcoh-bounded}, \eqref{it:pcoh-support}, and~\eqref{it:irr-orbit} come from the general theory of perverse-coherent sheaves: see~\cite[\S4.2]{ab:pcs}, as well as~\cite[\S4.5]{achar-pcoh}. Items~\eqref{it:pcoh-qhered} and~\eqref{it:irr-x} come from~\cite[Proposition~6.1 and Theorem~6.2]{achar-pcoh} (see also~\cite[Corollary~3]{bez}).  In view of item~\eqref{it:pcoh-qhered}, the map $c_\lambda$ in item~\eqref{it:irr-x} is unique up to scalar.

Note that items~\eqref{it:irr-x} and~\eqref{it:irr-orbit} give two different classifications of the simple objects in $\PCoh(\cN)_\bk$: one is parametrized by $\bX^+$, and the other by $\Omega_\bk$.

\begin{defn}
\label{defn:lv-bijection}
Let $\bk \in \{\F,\ovK\}$.  The \emph{Lusztig--Vogan bijection} for $G_\bk$ is the bijection
\[
\bX^+ \overset{\sim}{\leftrightarrow} \Omega_\bk
\]
determined by the following condition: $\lambda \in \bX^+$ corresponds to $(\scO,\cV) \in \Omega_\bk$ if $\Lpc_\lambda(\bk) \cong \IC(\scO,\cV)$.
\end{defn}

One can also interpret the Lusztig--Vogan bijection from a slightly different point of view if one chooses, for any $G_\bk$-orbit $\mathscr{O} \subset \cN_\bk$, a representative $x_\scO \in \scO$. Let us denote by $Z_{G_\bk}(x_\scO)$ the centralizer of $x_\scO$, and by $Z_{G_\bk}^{\mathrm{red}}(x_\scO)$ its reductive quotient (i.e. the quotient of $Z_{G_\bk}(x_\scO)$ by its unipotent radical). Our assumptions imply that the natural morphism $G_\bk / Z_{G_\bk}(x_\scO) \simto \scO$ is an isomorphism of varieties (see e.g.~\cite[Proposition~4.2]{mt}); hence pullback along the embedding $\{x_\scO\} \hookrightarrow \scO$ defines an equivalence of categories
\[
\Coh^{G_\bk}(\scO) \simto \Rep(Z_{G_\bk}(x_\scO)),
\]
where $\Rep(Z_{G_\bk}(x_\scO))$ is the category of finite-dimensional $Z_{G_\bk}(x_\scO)$-representations.  In particular, we deduce a bijection between the sets of simple objects in these two categories.  Since every irreducible $Z_{G_\bk}(x_\scO)$-module factors through the quotient map $Z_{G_\bk}(x_\scO) \to Z_{G_\bk}^{\mathrm{red}}(x_\scO)$, we obtain a bijection between the set of isomorphism classes of irreducible $G_\bk$-equivariant vector bundles on $\scO$ and the set of isomorphism classes of simple $Z_{G_\bk}^{\mathrm{red}}(x_\scO)$-modules.  Thus, $\Omega_\bk$ gets identified with the set
\[
\Omega'_\bk := 
\left\{
(\scO,L) \,\Big|\,
\begin{array}{@{}c@{}}
\text{$\scO \subset \cN$ is a nilpotent orbit, and} \\
\text{$L$ is an irreducible $Z_{G_\bk}^{\mathrm{red}}(x_\scO)$-module}
\end{array}
\right\},
\]
and the Lusztig--Vogan bijection for $G_\bk$ can be thought of as a bijection
\begin{equation}
\label{eqn:LV-bij}
\bX^+ \overset{\sim}{\leftrightarrow} \Omega_\bk'.
\end{equation}
The image of $\lambda \in \bX^+$ under this bijection will be denoted $(\scO_\lambda^\bk, L_\lambda^\bk)$.

\subsection{Graded versions}
\label{ss:grading}

We will denote by $\Gm$ the multiplicative group over $\O$, and by $(\Gm)_\F$ and $(\Gm)_\ovK$ its base change to $\F$ and $\ovK$ respectively. Then
$\Gm$ acts on $\fg$ by $z \cdot x = z^{-2}x$.  This makes the coordinate ring $\cO(\fg)$ into a graded ring concentrated in even, nonnegative degrees.  This $\Gm$-action preserves $\fn \subset \fg$, and it induces a $\Gm$-action on $\tcN$.  The $\Gm$-actions on $\tcN$ and on $\fg$ commute with the actions of $G$.  The map $\tilde\pi: \tcN \to \fg$ is $G \times \Gm$-equivariant, so one may consider the functor
\[
\tilde\pi_*: \Db\Cohmix(\tcN) \to \Db\Cohmix(\fg).
\]
Similar remarks apply to the $\F$- and $\ovK$-versions of these spaces. Moreover, we also have change-of-scalars functors $\F : \Db\Cohmix(\fg) \to \Db\Cohmix(\fg)_\F$ and $\ovK : \Db\Cohmix(\fg) \to \Db\Cohmix(\fg)_\ovK$, and they commute with $\tilde\pi_*$ as in~\eqref{eqn:tildepi-scalar-commute}. For $\bk \in \{\F,\ovK\}$, the $(\Gm)_\bk$-action on $\fg_\bk$ preserves $\cN_\bk$, so we also have functors
\[
\pi_*: \Db\Cohmix(\tcN)_\bk \to \Db\Cohmix(\cN)_\bk.
\]

We write $\cF \mapsto \cF\la 1\ra$ for the autoequivalence of any of these categories that twists the $\Gm$-equivariant structure by the tautological character of $\Gm$.  In the $\Gm$-equivariant setting, we modify~\eqref{eqn:dpc-defn} to include a normalization of the $(\Gm)_\bk$-action, as follows: we set
\[
\Dpc_\lambda(\bk) := \pi_*\cO_{\tcN}(w_0\lambda)\la \delta_{w_0\lambda}\ra
\qquad\text{and}\qquad
\npc_\lambda(\bk) := \pi_*\cO_\tcN(\lambda)\la -\delta_{w_0\lambda}\ra,
\]
where $\delta_\lambda$ is the length of a minimal element $v$ of the Weyl group such that $v\lambda \in -\bX^+$.  (See~\cite[\S2.3]{achar}.)

The definition of the perverse-coherent t-structure carries over to the $(\Gm)_\bk$-equiva\-riant setting.  The heart of the resulting t-structure is denoted by
\[
\PCoh^{\Gm}(\cN)_\bk \subset \Db\Cohmix(\cN)_\bk.
\]
This category is stable under $\la 1\ra$.  Properties~\eqref{it:pcoh-bounded} and~\eqref{it:pcoh-support} from~\S\ref{ss:pcoh} remain true as stated for $\PCoh^{\Gm}(\cN)_\bk$, but the remaining properties must be modified as follows:
\begin{enumerate}
\setcounter{enumi}{2}
\item The objects $\Dpc_\lambda(\bk)$ and $\npc_\lambda(\bk)$ lie in $\PCoh^{\Gm}(\cN)_\bk$.  Moreover,
\[
\Hom(\Dpc_\lambda(\bk), \npc_\mu(\bk)\la k\ra) \cong
\begin{cases}
\bk & \text{if $\lambda = \mu$ and $k = 0$,} \\
0 & \text{otherwise.}
\end{cases}
\]
\item Fix a nonzero map $c_\lambda: \Dpc_\lambda(\bk) \to \npc_\lambda(\bk)$, and set\label{it:irr-x-gr}
\[
\Lpc_\lambda(\bk) = \im(c_\lambda: \Dpc_\lambda(\bk) \to \npc_\lambda(\bk)).
\]
Then $\Lpc_\lambda(\bk)$ is a simple object in $\PCoh(\cN)_\bk$.  Moreover, every simple object is isomorphic to $\Lpc_\lambda(\bk)\la k\ra$ for a unique pair $(\lambda,k) \in \bX^+ \times \Z$, and each composition factor of the kernel of the surjection $\Dpc_\lambda(\bk) \twoheadrightarrow \Lpc_\lambda(\bk)$, resp.~of the cokernel of the embedding $\Lpc_\lambda(\bk) \hookrightarrow \npc_\lambda(\bk)$, is of the form $\Lpc_\mu(\bk) \langle m \rangle$ with $\mu < \lambda$.
\item Let $\scO \subset \cN_\bk$ be a nilpotent orbit, and let $\cV$ be an irreducible $(G \times \Gm)_\bk$-equivariant vector bundle on $\scO$.  There is a unique simple perverse-coherent sheaf\label{it:irr-orbit-gr}
\[
\IC(\scO,\cV)
\]
that is characterized by the following properties: it is supported on $\overline{\scO}$, and $\sH^{\frac{1}{2}\codim\scO}(\IC(\scO,\cV))|_{\St(\scO)} \cong i_{\scO*}\cV$.  Moreover, every simple object is isomorphic to $\IC(\scO,\cV)$ for a unique pair $(\scO,\cV)$.
\end{enumerate}
See~\S\ref{ss:pcoh} for references to~\cite{achar-pcoh, ab:pcs, bez} for these statements (see also Remark~\ref{rmk:prettygood-pcoh}). Note that in part~\eqref{it:irr-orbit-gr}, the simple objects are parametrized not by $\Omega_\bk$, but instead by the larger set $\Omega_\bk^{\Gm}$ consisting of pairs $(\scO,\cV)$ where $\cV$ is a $(G \times \Gm)_\bk$-equivariant vector bundle, rather than a $G$-equivariant vector bundle.  Comparing parts~\eqref{it:irr-x-gr} and~\eqref{it:irr-orbit-gr}, we see that there is a \emph{graded Lusztig--Vogan bijection}
\[
\bX^+ \times \Z \overset{\sim}{\leftrightarrow} \Omega_\bk^{\Gm}.
\]
The extra $\Gm$-action is crucial for some applications, but for most of this paper, we will work \emph{without} this $\Gm$-action.  (One exception is the proof of Proposition~\ref{prop:support}, where the $\Gm$-action plays an important role.)  However, there is no loss in doing so: as explained in~\cite[\S3]{ah}, the graded Lusztig--Vogan bijection is completely determined by the ordinary (ungraded) Lusztig--Vogan bijection.  In particular, the main theorem of this paper implies that the graded Lusztig--Vogan bijection is also independent of $\bk$.

\section{Balanced nilpotent sections and associated Slodowy slices}
\label{sec:balanced}

\subsection{Balanced nilpotent sections and their centralizers}
\label{ss:balanced}

Elements of $\fg$ are in a canonical bijection with the $\O$-points of the $\O$-scheme $\fg$. Following~\cite{mcninch2016}, such points will be called \emph{sections}. Any section $x \in \fg$, considered as a morphism $\Spec(\O) \to \fg$, determines by base change an $\F$-point $x_\F$ of $\fg_\F$ (in other words, an element of the $\F$-vector space $\fg_\F$) and a $\K$-point $x_\K$ of $\fg_\K$ (in other words, an element of the $\K$-vector space $\fg_\K$). The image of $x_\K$ in $\fg_\ovK$ will be denoted $x_\ovK$.

We will denote by
\[
Z_G(x), \quad Z_G(x)_\F, \quad Z_G(x)_\K, \quad Z_G(x)_\ovK
\]
the scheme-theoretic centralizer of $x$ in $G$, of $x_\F$ in $G_\F$, of $x_\K$ in $G_\K$, and of $x_\ovK$ in $G_\ovK$, respectively. We then have canonical identifications
\begin{align*}
Z_G(x)_\F &= \Spec(\F) \times_{\Spec(\O)} Z_G(x), \\
Z_G(x)_\K &= \Spec(\K) \times_{\Spec(\O)} Z_G(x), \\
Z_G(x)_\ovK &= \Spec(\ovK) \times_{\Spec(\O)} Z_G(x).
\end{align*}
Note also that if we set
\begin{align*}
\fz_\fg(x)_\F &:= \{ y \in \fg_\F \mid [x_\F,y] = 0\}, \\
\fz_\fg(x)_\K &:= \{ y \in \fg_\K \mid [x_\K,y] = 0\}, \\
\fz_\fg(x)_\ovK &:= \{ y \in \fg_\ovK \mid [x_\ovK,y] = 0\}
\end{align*}
then by~\cite[Equation~I.7.2(7)]{jantzen} we have
\begin{equation}
\label{eqn:Lie-centralizer}
\mathrm{Lie}(Z_G(x)_\F) = \fz_\fg(x)_\F, \quad
\mathrm{Lie}(Z_G(x)_\K) = \fz_\fg(x)_\K, \quad
\mathrm{Lie}(Z_G(x)_\ovK) = \fz_\fg(x)_\ovK.
\end{equation}
Finally, by~\cite[Proposition~4.2]{mt}, $Z_G(x)_\F$ is a smooth group scheme. (Of course, $Z_G(x)_\K$ and $Z_G(x)_\ovK$ are also smooth.)

Following~\cite[Definition~1.4.1]{mcninch2016}, a section $x$ will be called \emph{balanced} if 
\[
\dim(Z_G(x)_\F)=\dim(Z_G(x)_\K).
\]
(The remarks above show that our terminology is indeed compatible with that in~\cite{mcninch2016}.)
On the other hand,
a section $x$ will be called \emph{nilpotent} if $x_\K$ is a nilpotent element in $\fg_\K$. By~\cite[Lemma~3.2.1]{mcninch2016}, if $x$ is nilpotent then $x_\F$ is a nilpotent element in $\fg_\F$. The sections which we will be mostly interested in are the \emph{balanced nilpotent sections}.

\begin{rmk}
By invariance of the dimension under field extensions (see e.g.~\cite[Proposition~5.38]{gw}) we have $\dim(Z_G(x)_\ovK)=\dim(Z_G(x)_\K)$. Hence the ``balanced'' condition can also be stated purely in terms of algebraically closed fields by requiring that $\dim(Z_G(x)_\F)=\dim(Z_G(x)_\ovK)$.
\end{rmk}

We refer e.g.~to~\cite[Definition~3.3.1]{mcninch2016} or~\cite[Definition~5.3]{jantzen-nilp} for the definition of a cocharacter associated with a nilpotent element. Recall also that under our assumptions there exists a canonical bijection from the set of nilpotent $G_\ovK$-orbits in $\fg_\ovK$ to the set of nilpotent $G_\F$-orbits in $\fg_\F$, see~\cite[\S 4.1]{ahr} for references. This bijection will be called the \emph{Bala--Carter bijection}, and denoted $\BC$. It satisfies $\dim(\scO)=\dim(\BC(\scO))$ for any $G_\ovK$-orbit $\scO$. By work of Spaltenstein (see again~\cite[\S 4.1]{ahr} for references), it is known also that $\BC$ is a bijection of posets, for the orders given by inclusions of closures of nilpotent orbits.

The main properties of balanced nilpotent sections from~\cite[Theorem~3.4.5]{mcninch2016} that we will need are summarized in the following theorem.

\begin{thm}[McNinch]
\label{thm:balanced-sections}
If $y \in \fg_\F$ is a nilpotent element, then there exists a balanced nilpotent section $x$ and a cocharacter $\varphi : \Gm \to G$ such that 
\begin{enumerate}
\item
\label{it:balanced-1}
$y=x_\F$;
\item
\label{it:balanced-2}
$\Spec(\F) \times_{\Spec(\O)} \varphi$ is a cocharacter associated with $x_\F$;
\item
\label{it:balanced-3}
$\Spec(\ovK) \times_{\Spec(\O)} \varphi$ is a cocharacter associated with $x_\ovK$;
\item
\label{it:balanced-4}
$\BC(G_\ovK \cdot x_\ovK) = G_\F \cdot x_\F$.
\end{enumerate}
\end{thm}


The other important property we will use is the following.

\begin{thm}
\label{thm:flatness}
Let $x \in \fg$ be a balanced nilpotent section. Then the $\O$-group scheme $Z_G(x)$ is smooth over $\O$. Moreover, the groups of connected components of the algebraic groups $\Spec(\F) \times_{\Spec(\O)} Z_G(x)$ and $\Spec(\ovK) \times_{\Spec(\O)} Z_G(x)$ have the same cardinality.
\end{thm}

\begin{proof}
These claims are proved in~\cite[Theorem~1.6, Theorem~1.8]{har:smcent}. A different argument for smoothness is also given in~\S\ref{ss:flatness} below.
\end{proof}

\subsection{Integral Slodowy slices for balanced nilpotent sections}
\label{ss:slodowy}

We continue with the setting of~\S\ref{ss:balanced}, and fix a balanced nilpotent section $x$ and a cocharacter $\varphi : \Gm \to G$ such that $\Spec(\F) \times_{\Spec(\O)} \varphi$ is a cocharacter associated with $x_\F$, and $\Spec(\ovK) \times_{\Spec(\O)} \varphi$ is a cocharacter associated with $x_\ovK$. Our goal is to define a ``Slodowy slice'' in $\fg$ attached to $x$.

\begin{lem}
\label{lem:summand}
The $\O$-submodule $[x,\fg] \subset \fg$ is a direct summand in $\fg$.
\end{lem}

\begin{proof}
Consider the right exact sequence of $\O$-modules $\fg \xrightarrow{\ad(x)} \fg \to \fg/[x,\fg] \to 0$.  After tensoring with $\F$ or $\K$, one obtains analogous right exact sequences over those fields.  In particular, we have
\begin{align*}
\dim \F \otimes_\O (\fg/[x,\fg]) &= \dim \fg_\F - \mathrm{rank}(\ad(x_\F)) = \dim \fz_\fg(x)_\F, \\
\dim \K \otimes_\O (\fg/[x,\fg]) &= \dim \fg_\K - \mathrm{rank}(\ad(x_\K)) = \dim \fz_\fg(x)_\K.
\end{align*}
Since $Z_G(x)_\F$ and $Z_G(x)_\K$ are smooth (so that their dimension coincides with that of their Lie algebra), and since $x$ is balanced, using~\eqref{eqn:Lie-centralizer} we see that these dimensions are equal.  Therefore, $\fg/[x,\fg]$ is a torsion-free $\O$-module, and the lemma follows.
\end{proof}

We now consider the $\Gm$-action on $\fg$ determined by $\varphi$ via the adjoint action. Then $x$ has weight $2$ for this action (because this is true by assumption over $\ovK$), and the submodule $[x,\fg] \subset \fg$ is $\Gm$-stable. We fix a $\Gm$-stable complement $M \subset \fg$ for this submodule (which exists by Lemma~\ref{lem:summand}).

\begin{lem}
\label{lem:weights-Slodowy}
All the $\Gm$-weights on $M$ are nonpositive.
\end{lem}

\begin{proof}
It is sufficient to prove a similar claim for the $(\Gm)_\ovK$-weights on $\ovK \otimes_\O M$. Now by assumption $\Spec(\ovK) \times_{\Spec(\O)} \varphi$ is a cocharacter of $G_\ovK$ associated with $x_\ovK$. By~\cite[Lemma~5.7 and Proposition~5.8]{jantzen-nilp}, we see that $[x_\ovK,\fg_\ovK]$ contains all the $(\Gm)_\ovK$-weight spaces of $\fg_\ovK$ of positive weight, which implies our claim.
\end{proof}

We set
\[
\mathcal{S}_x := x + M \,\, \subset \,\, \fg.
\]
(Contrary to what the notation might suggest, this scheme depends not only on $x$, but of course also on $\varphi$ and $M$.)
If we define a $\Gm$-action on $\fg$ via $z \cdot y = z^{-2} \varphi(z) \cdot y$, then $\mathcal{S}_x$ is a $\Gm$-stable closed subscheme of $\fg$, and in view of Lemma~\ref{lem:weights-Slodowy} the weights of $\Gm$ on $\mathcal{O}(\mathcal{S}_x)$ are nonnegative, the weight-$0$ subspace consisting of the constants $\O \subset \mathcal{O}(\mathcal{S}_x)$. We will also denote by
\[
a_x : G \times_{\Spec(\O)} \mathcal{S}_x \to \fg
\]
the morphism induced by the adjoint action.

\subsection{Some properties of Slodowy slices over fields}
\label{ss:slodowy-fields}

We continue with the setting of~\S\ref{ss:slodowy}, and let $\bk$ be either $\F$ or $\ovK$.
We will consider the affine subspace
\[
\mathcal{S}_x^\bk := x_\bk + (\bk \otimes_\O M) \subset \fg_\bk
\]
(which is a variant of the \emph{Slodowy slices} constructed in~\cite{slodowy}). This variety is endowed with an action of $(\Gm)_\bk$ (induced by the $\Gm$-action on $\mathcal{S}_x$ considered above) which contracts it to $\{x_\bk\}$. We will denote by $a_x^\bk : G_\bk \times \mathcal{S}_x^\bk \to \fg_\bk$ the base change of $a_x$ to $\bk$, i.e.~the morphism induced by the adjoint action.

\begin{prop}
\label{prop:Slodowy-flat}
The morphism $a_x^\bk$ is smooth (and hence, in particular, flat).
\end{prop}

\begin{proof}
We observe that the differential of $a_x^\bk$ at the point $(1,x_\bk)$ identifies with the morphism $\fg_\bk \times (\bk \otimes_\O M) \to \fg_\bk$ sending $(y,y')$ to $[x_\bk,y] + y'$. This differential is surjective since $\bk \otimes_\O M$ is a complement to $[x_\bk,\fg_\bk]$ in $\fg_\bk$, which proves that $a_x^\bk$ is smooth at $(1,x_\bk)$. Since the locus of points of $G_\bk \times \mathcal{S}_x^\bk$ where this map is smooth is open, and stable under the $(G \times \Gm)_\bk$-action defined by $(g,z) \cdot (h,y) = (gh \varphi(z)^{-1}, z^{-2} \varphi(z) \cdot y)$, this locus must then be the whole of $G_\bk \times \mathcal{S}_x^\bk$.
\end{proof}

Since $a_x^\bk$ is flat, it is open (see e.g.~\cite[Theorem~14.33]{gw}).  Let
\[
V_x^\bk := (a_x^\bk)(G_\bk \times \mathcal{S}_x^\bk)
\]
be its image (an open subset of $\fg_\bk$).

\begin{lem}
\label{lem:intersection-Slodowy-orbit}
The following square is Cartesian, where the vertical maps are the closed embeddings and the horizontal maps are induced by the adjoint action:
\[
\begin{tikzcd}
G_\bk \times \{x_\bk\} \ar[r] \ar[d, hook] & G_\bk \cdot x_\bk \ar[d, hook] \\
G_\bk \times \mathcal{S}_x^\bk \ar[r] & V_x^\bk.
\end{tikzcd}
\]
\end{lem}

\begin{proof}
By smoothness of $a_x^\bk$ (see Proposition~\ref{prop:Slodowy-flat}), the fiber product $(G_\bk \times \mathcal{S}_x^\bk) \times_{V^\bk_x} (G_\bk \cdot x_\bk)$ is smooth over $G_\bk \cdot x_\bk$, and hence a smooth variety, of dimension $\dim(G_\bk)$.
By $G_\bk$-equivariance, we have
\[
(G_\bk \times \mathcal{S}_x^\bk) \times_{V^\bk_x} (G_\bk \cdot x_\bk) = G_\bk \times \bigl( \mathcal{S}_x^\bk \cap (G_\bk \cdot x_\bk) \bigr),
\]
where on the right-hand side we consider the scheme-theoretic intersection. (This follows e.g.~from~\cite[Lemma~4 on p.~26]{slodowy} applied to the composition $(G_\bk \times \mathcal{S}_x^\bk) \times_{V^\bk_x} (G_\bk \cdot x_\bk) \to G_\bk \times \mathcal{S}_x^\bk \to G_\bk$, where the second map is the projection.) It follows in particular that the right-hand side is smooth. Since the projection $$G_\bk \times \bigl( \mathcal{S}_x^\bk \cap (G_\bk \cdot x_\bk) \bigr) \to \mathcal{S}_x^\bk \cap (G_\bk \cdot x_\bk)$$ is smooth, using~\cite[Tag 02K5]{stacks-project} we deduce that $\mathcal{S}_x^\bk \cap (G_\bk \cdot x_\bk)$ is smooth and of dimension $0$, and hence a disjoint union of points. On the other hand this variety admits a $(\Gm)_\bk$-action which contracts it to $x_\bk$; we deduce that $\mathcal{S}_x^\bk \cap (G_\bk \cdot x_\bk) = \{x_\bk\}$, which implies the desired identification.
\end{proof}

\begin{cor}
\label{cor:fiber-product-Slodowy}
In the following diagram, every square is cartesian:
\[
\begin{tikzcd}
G_\bk \times \{x_\bk\} \ar[r] \ar[d] &
  (G_\bk \times \mathcal{S}_x^\bk) \times_{\fg_\bk} \cN_\bk \ar[r] \ar[d] &
  G_\bk \times \mathcal{S}_x^\bk \ar[d] \\
G_\bk \cdot x_\bk \ar[r] &
  \St(G_\bk \cdot x_\bk) \ar[r] \ar[d] & V^\bk_x \ar[d] \\
& \cN_\bk \ar[r] & \fg_\bk.
\end{tikzcd}
\]
The vertical maps between the top two rows are smooth and surjective.  The vertical maps between the bottom two rows are open embeddings.
\end{cor}

\begin{proof}
If a nilpotent orbit $\scO'$ intersects $V_x^\bk$, then it must contain an element in $\mathcal{S}^\bk_x$. Since the $(\Gm)_\bk$-action is contacting to $x$, and preserves the nilpotent orbits (as follows from~\cite[Lemma~2.10]{jantzen-nilp}), this implies that $x_\bk \in \overline{\scO'}$, so that $\scO' \subset \St(G_\bk \cdot x_\bk)$. Conversely, if $\scO'$ contains $x_\bk$ in its closure, then it must intersect the open subset $V_x^\bk$, hence be contained in it. We have finally proved that $V^\bk_x \cap \cN_\bk = \St(G_\bk \cdot x_\bk)$, i.e.~that the bottom square is Cartesian. Since the square formed by the the second and third schemes on the first line and the bottom line is Cartesian by definition, this implies that the upper right square is Cartesian. Then, using the same argument and Lemma~\ref{lem:intersection-Slodowy-orbit}, we deduce that the upper left square is Cartesian.

The claim about smoothness is then clear from the smoothness of $a_x^\bk$ (see Proposition~\ref{prop:Slodowy-flat}), and the final claim follows from the definitions.
\end{proof}

\subsection{Smoothness of centralizers}
\label{ss:flatness}

In this subsection we sketch a different proof of the smoothness claim in Theorem~\ref{thm:flatness}. No details of this proof will be used in the rest of the paper. We first remark that the smoothness of $Z_G(x)$ follows easily once we know that this group scheme is flat over~$\O$; see~\cite[Lemma~1.5]{har:smcent} for details.

By Theorem~\ref{thm:balanced-sections}, there exists a cocharacter $\varphi : \Gm \to G$ such that $\Spec(\F) \times \varphi$ is associated with $x_\F$ and $\Spec(\ovK) \times \varphi$ is associated with $x_\ovK$.  Then we can consider an ``integral Slodowy slice'' $\mathcal{S}_x$ as constructed in~\S\ref{ss:slodowy}. By a variant of~\cite[Lemma~4.1.1]{riche-kostant} (for discrete valuation rings instead of localizations of $\Z$) one can deduce from Proposition~\ref{prop:Slodowy-flat} (in the case $\bk=\F$) that $a_x$ is flat. Hence to conclude it suffices to prove that the following diagram is Cartesian, where the horizontal maps are induced by the adjoint action and the vertical maps are the closed embeddings:
\[
\begin{tikzcd}
Z_G(x) \times \{x\} \ar[r] \ar[d,hook] & \{x\} \ar[d,hook] \\
G \times \mathcal{S}_x \ar[r] & \fg. 
\end{tikzcd}
\]
However, it follows from Lemma~\ref{lem:intersection-Slodowy-orbit} that the base change of this diagram to $\ovK$ is Cartesian. Since the $\O$-scheme $(G \times \mathcal{S}_x) \times_\fg \{x\}$ is flat by flatness of $a_x$, this implies that the composition
\[
(G \times \mathcal{S}_x) \times_\fg \{x\} \to G \times \mathcal{S}_x \to \mathcal{S}_x
\]
(where the second map is the projection) factors through the embedding $\{x\} \hookrightarrow \mathcal{S}_x$. Hence we have
\[
(G \times \mathcal{S}_x) \times_\fg \{x\} = (G \times \{x\}) \times_\fg \{x\} = Z_G(x),
\]
which finishes the proof.

\section{Integral exotic sheaves}
\label{sec:intexotic}

In this section, we will work with the category $\Db\Cohmix(\tcN)$ (cf.~\S\ref{ss:grading}), rather than $\Db\Coh^G(\tcN)$, mainly because we anticipate that this may be useful for future applications. However, the $\Gm$-action plays almost no role in any of the arguments. Appropriate analogues of the statements in this section hold for $\Db\Coh^G(\tcN)$, and we will use these versions elsewhere in the paper.

The goal of this section is to construct a t-structure on $\Db\Cohmix(\tcN)$ (as well as on the $\F$- and $\ovK$-versions), called the \emph{exotic t-structure}, using the machinery from Appendix~\ref{sec:exc-pid}.  In the case of field coefficients, this t-structure has been extensively studied in the literature~\cite{achar, arider, bezru, mr}, but over $\O$, some of the statements in this section are new.

\subsection{Passage to a group with simply connected derived subgroup}

In the arguments below, we will need some results from~\cite{br,mr} that only apply to groups with a simply connected derived subgroup.  To accommodate these results, for the remainder of this section, we fix a central isogeny
\[
\sigma: \tG \to G
\]
as in Lemma~\ref{lem:derived-isogeny}.  Let $\tT \subset \tB \subset \tG$ be the maximal torus and the Borel subgroup obtained as the preimages of $T$ and $B$ along $\sigma$.  According to Lemma~\ref{lem:derived-isogeny}, $\sigma$ lets us identify the Lie algebra of $\tG$ with $\fg$.  This yields an identification of their Springer resolutions as well, as the obvious map
\[
\tG \times^{\tB} \fn \to G \times^B \fn
\]
is an isomorphism.  We may thus speak of both $G$- and $\tG$-equivariant sheaves on $\tcN$, and there is an obvious functor
\begin{equation}\label{eqn:tg-embed}
\Db\Cohmix(\tcN) \to \Db\Cohtmix(\tcN).
\end{equation}
Any $G$-invariant perfect pairing on $\fg$ is also $\tG$-invariant, so the isomorphism in Lemma~\ref{lem:tcn-alt} is also compatible with passage from $G$ to $\tG$.  The results we need from~\cite{br,mr} are usually applicable to $\tG \times \Gm$-equivariant sheaves on $\tG \times^{\tB} (\fg/\fb)^*$, and hence to $\Db\Cohtmix(\tcN)$.  In order to extract useful information in $\Db\Cohmix(\tcN)$, we need to understand the functor~\eqref{eqn:tg-embed}.

Let $\tbX$ be the character lattice of $\tT$ (cf.~the proof of Lemma~\ref{lem:derived-isogeny}).  This group contains $\bX$ as a subgroup.  The quotient $\tbX/\bX$ is finite and of order coprime to $p$.  Let $F$ be the kernel of $\sigma: \tG \to G$.  This is a diagonalizable smooth finite group scheme whose character group is identified with $\tbX/\bX$.

Note that any coherent sheaf $\cF \in \Cohtmix(\tcN)$ comes equipped with a canonical (in particular, functorial) decomposition
\begin{equation}\label{eqn:f-decomp}
\cF = \bigoplus_{\nu \in \tbX/\bX} \cF^\nu,
\end{equation}
according to the action of $F$.  To see this decomposition more concretely, we can pass through the ``induction'' equivalence
\[
\Cohtmix(\tcN) \cong \Coh^{\tB \times \Gm}(\fn), 
\]
and identify the right-hand side with the category of finitely generated $\tB$-equiva\-riant graded modules over the ring $\cO(\fn) = \mathrm{Sym}_\O(\fn^*)$.  If we regard $\cF$ as such a module, then it is in particular a rational $F$-module, so it admits a canonical decomposition~\eqref{eqn:f-decomp} as an $F$-module.  Since $F$ is in the center of $\tB$ and acts trivially on $\cO(\fn)$, each summand on the right-hand side of~\eqref{eqn:f-decomp} is stable under the actions of $\tB$ and $\cO(\fn)$.  In other words, the decomposition~\eqref{eqn:f-decomp} takes place in $\Cohtmix(\tcN)$.

For $\nu \in \tbX/\bX$, let $\Cohtmix(\tcN)^\nu$ be the full subcategory of $\Cohtmix(\tcN)$ consisting of objects $\cF$ that satisfy $\cF = \cF^\nu$.  An immediate consequence of~\eqref{eqn:f-decomp} is that we have a categorical decomposition
\begin{equation}\label{eqn:cohtmix-decomp}
\Db\Cohtmix(\tcN) = \bigoplus_{\nu \in \tbX/\bX} \Db\Cohtmix(\tcN)^\nu.
\end{equation}
At the level of abelian categories, it is clear that the natural functor $\Cohmix(\tcN) \to \Cohtmix(\tcN)$ is fully faithful.  Its image is the subcategory $\Cohtmix(\tcN)^0$, so the decomposition above gives us an identification
\[
\Db\Cohmix(\tcN) \cong \Db\Cohtmix(\tcN)^0.
\]
For the remainder of this section, we identify $\Db\Cohmix(\tcN)$ with a subcategory of $\Db\Cohtmix(\tcN)$ in this way.

We conclude this subsection with two lemmas on generators for these derived categories.

\begin{lem}
\label{lem:line-generate}
For any $\bE \in \{\ovK,\O,\F\}$, the category $\Db\Cohmix(\tcN)_\bE$ is generated as a triangulated category by objects of the form $\cO_{\tcN_\bE}(\lambda)\la k\ra$ with $\lambda \in \bX$ and $k \in \Z$.
\end{lem}

\begin{proof}
In the case where $\bE$ is a field, this is proved in~\cite[Corollary~5.8]{achar-pcoh} or~\cite[Corollary~2.7]{mr}.  Here, therefore, we will only treat the case where $\bE = \O$.  (However, the reader can easily modify this argument to handle the field case as well.)  As in the discussion above, we can replace $\Db\Cohmix(\tcN)$ with $\Db\Coh^{B \times \Gm}(\fn)$, and we will think of objects $M \in \Coh^{B \times \Gm}(\fn)$ as $B$-equivariant graded $\cO(\fn)$-modules.  From now on, all $\cO(\fn)$-modules will implicitly be assumed to be finitely generated.

Let $M \in \Coh^{B \times \Gm}(\fn)$. By~\cite[Proposition~2]{serre-Groth}, there exists a $(B \times \Gm)$-stable $\O$-submodule $M' \subset M$ which is of finite type over $\O$ and which generates $M$ as an $\cO(\fn)$-module. Then by~\cite[Proposition~3]{serre-Groth} there exists a $(B \times \Gm)$-module $M''$ which is free over $\O$ and surjects to $M'$; in this way we see that $M$ is a quotient of a $(B\times \Gm)$-equivariant free $\cO(\fn)$-module, and then that $M$ admits a resolution
\[
\cdots \to P^2 \to P^1 \to P^0 \to M \to 0
\]
where each $P^k$ is a $(B \times \Gm)$-equivariant free $\cO(\fn)$-module. Since $\cO(\fn)$ is isomorphic to a ring of polynomials (in $\mathrm{rk}(\fn)$ many variables) with coefficients in $\O$, it has finite global dimension, say $d$.  Let $Q$ be the kernel of $P^{d-1} \to P^{d-2}$, and consider the exact sequence
\[
0 \to Q \to P^{d-1} \to \cdots \to P^0 \to M \to 0.
\]
A routine homological algebra argument shows that as a (graded) $\cO(\fn)$-module, $Q$ must be projective.  Since $\cO(\fn)$ is a graded polynomial ring over a noetherian local ring, a suitable variant of Nakayama's lemma implies that every projective graded $\cO(\fn)$-module is free.\footnote{When invoking the graded Nakayama lemma, instead of using the grading coming from the $\Gm$-action, we could instead use the grading given by a strictly dominant cocharacter of $T \subset B$.  The latter version also makes sense in the setting of $\Coh^G(\tcN) \cong \Coh^B(\fn)$.}  We have thus shown that $M$ admits a finite resolution by $B$-equivariant free graded $\cO(\fn)$-modules.

Let $M$ be a $B$-equivariant free graded $\cO(\fn)$-module.  We will show (by induction on the rank of $M$ over $\cO(\fn)$) that $M$ admits a filtration whose subquotients are $B$-equivariant free graded $\cO(\fn)$-modules of rank~$1$.  Such objects correspond to $(G \times \Gm)$-equivariant line bundles, so this claim will prove the lemma.

Consider the quotient map $M \to M/\fn M$.  Then $M/\fn M$ is a free $\O$-module of finite rank (equal to the rank of $M$ as an $\cO(\fn)$-module).  The quotient map admits a $(T \times \Gm)$-equivariant splitting $M/\fn M \to M$.  Choose such a splitting, and let $M_0$ be its image.  Then any $\O$-basis for $M_0$ is a $\cO(\fn)$-basis for $M$.  
As a $T$-representation, $M_0$ decomposes as a direct sum
\[
M_0 = \bigoplus_{\lambda \in \bX} (M_0)_\lambda,
\]
where each $(M_0)_\lambda$ is again a free $\O$-module of finite rank.  Choose some $\lambda$ such that $(M_0)_\lambda \ne 0$, and such that $\lambda$ is minimal for this property with respect to the partial order $\leq$ on $\bX$ considered in~\S\ref{ss:pcoh}.  Then choose an element $v \in (M_0)_\lambda$ that is part of some $\O$-basis for $(M_0)_\lambda$ consisting of vectors homogeneous with respect to the $\Gm$-action.  Then the $\O$-span $\O \cdot v$ is a $\mathrm{Dist}(B)$-submodule of $M_0$, and hence also a $B$-submodule by~\cite[Lemma~I.7.15]{jantzen}. (This statement is applicable here since $B$ is smooth, and hence infinitesimally flat; see~\cite[\S I.10.11]{jantzen}.) Let $M' \subset M$ be the $\cO(\fn)$-submodule generated by $v$.  This is a $(B \times \Gm)$-equivariant submodule of $M$ that is free over $\cO(\fn)$ of rank~$1$.

Since $v$ is part of an $\cO(\fn)$-basis for $M$,
the quotient $M'' := M/M'$ is again a $B$-equivariant free graded $\cO(\fn)$-module (of rank lower than that of $M$).  By induction, $M''$ admits a $B$-equivariant filtration whose subquotients are free over $\cO(\fn)$ of rank~$1$, and hence so does $M$.
\end{proof}

\begin{lem}
\label{lem:line-generate-2}
Let $\nu \in \tbX/\bX$.  For any $\bE \in \{\ovK,\O,\F\}$, the category $\Db\Cohtmix(\tcN)^\nu_\bE$ is generated as a triangulated category by objects of the form $\cO_{\tcN_\bE}(\lambda)\la k\ra$ with $\lambda \in \tbX$, $\lambda + \bX = \nu$, and $k \in \Z$.
\end{lem}
\begin{proof}
By Lemma~\ref{lem:line-generate} applied to $\tG$, the category $\Db\Cohtmix(\tcN)_\bE$ as a whole is generated by the line bundles $\cO_{\tcN_\bE}(\lambda)\la k\ra$ with $\lambda \in \tbX$ and $k \in \Z$.  Since each $\cO_{\tcN_\bE}(\lambda)\la k\ra$ lies in one summand on the right-hand side of~\eqref{eqn:cohtmix-decomp} (namely, the summand labeled by $\nu = \lambda + \bX$), we see that each $\Db\Cohtmix(\tcN)^\nu_\bE$ is generated by those line bundles $\cO_{\tcN_\bE}(\lambda)\la k\ra$ that it contains.
\end{proof}

\subsection{Exotic generators for the derived category}

We will now define a new set of generators, using the extended affine braid group action constructed in~\cite{br}.  

Let $\Wext := W \ltimes \tbX$ be the extended affine Weyl group for $\tG$. Recall that in general this group is not a Coxeter group, but it is endowed with a natural length function and with a ``Bruhat order''; see~\cite{mr} for details and references. Recall also that every element $w \in \Wext$ determines an element of the extended affine braid group $\Bext$, denoted by $T_w$.  The main result of~\cite{br} associates to $T_w$ a certain autoequivalence of $\Db\Cohmix(\tcN)$, denoted by $\cJ_{T_w}$.  Given $\lambda \in \tbX$, regard it as an element of $\Wext$, and consider its right coset for the finite Weyl group $W\lambda \subset \Wext$.  Let $w_\lambda$ be the unique element of minimal length in $W\lambda$. Following~\cite[\S3.3]{mr}, for $\bE \in \{\ovK,\O,\F\}$, we set
\[
\tnex_\lambda(\bE) := \cJ_{T_{w_\lambda}}(\cO_{\tcN_\bE})
\qquad\text{and}\qquad
\tDex_\lambda(\bE) := \cJ_{(T_{w_\lambda^{-1}})^{-1}}(\cO_{\tcN_\bE}).
\]
These are objects in $\Db\Cohtmix(\tcN)_\bE$, but it can deduced from~\cite{mr} that some of them lie in $\Db\Cohmix(\tcN)_\bE$.  To explain this, we need some more notation.  Let us temporarily regard $\tbX$ as a subset of the real vector space $\mathbb{R} \otimes_\Z \tbX$.  In the latter, it makes sense to take the convex hull of any finite set of elements.  For $\lambda \in \tbX$, we set
\begin{align*}
\conv(\lambda) &= (\lambda + \Z R) \cap (\text{convex hull of $W \cdot \lambda$ in $\mathbb{R} \otimes_\Z \tbX$}), \\
\convo(\lambda) &= \conv(\lambda) \smallsetminus W \cdot \lambda.
\end{align*}
(Recall that $R$ is the root system of $G$.)  It is well known that when $\lambda$ and $\mu$ are dominant, we have $\conv(\lambda) \subset \conv(\mu)$ if and only if $\lambda \le \mu$.  As a consequence, for arbitrary $\lambda,\mu \in \tbX$, we have
\begin{equation}\label{eqn:conv-preorder}
\mu \in \convo(\lambda)
\qquad\Longrightarrow\qquad
\lambda \notin \conv(\mu).
\end{equation}

Consider the preorder $\preceq$ on $\tbX$ given by $\mu \preceq \lambda$ if $\mu \in \conv(\lambda)$.  By~\eqref{eqn:conv-preorder}, the equivalence classes for this preorder are precisely the $W$-orbits in $\tbX$.  For $\lambda \in \tbX$, let $D_{\conv(\lambda)}$, resp.~$D_{\convo(\lambda)}$, be the full triangulated subcategory of $\Db\Cohmix(\tcN)$ generated by objects of the form $\cO_\tcN(\mu)\la k\ra$ with $k \in \Z$ and $\mu \in \conv(\lambda)$, resp.~$\mu \in \convo(\lambda)$.  The categories $D_{\conv(\lambda)}$ and $D_{\convo(\lambda)}$ are both contained in a single summand of the right-hand side of~\eqref{eqn:cohtmix-decomp} (namely, the one corresponding to $\nu = \lambda + \bX$).  It can be deduced from~\cite[Lemma~3.1]{mr} (see also the proof of~\cite[Proposition~3.7]{mr})\footnote{Although~\cite{mr} works with field coefficients, the specific statements cited here are essentially minor variations on~\cite[Lemma~1.11.3]{br}, which holds for $\bE = \O$ as well.  It is left to the reader to check that the arguments we need go through for general $\bE$.} that for all $\lambda \in \tbX$, we have
\begin{equation}
\label{eqn:tnex-quotient}
\tnex_\lambda(\bE) \cong \cO_\tcN(\lambda)\la \delta_\lambda\ra \pmod {D_{\convo(\lambda)}},
\end{equation}
where $\delta_\lambda$ is the length of a minimal element $v \in W$ such that $v\lambda$ is dominant. In particular, this implies that
\[
\tnex_\lambda(\bE) \in D_{\conv(\lambda)} \subset \Db\Cohtmix(\tcN)^\nu_\bE,
\qquad\text{where $\nu = \lambda + \bX$.}
\]
In particular, if $\lambda \in \bX$, then $\tnex_\lambda(\bE) \in \Db\Cohmix(\tcN)_\bE$.  We will see later that similar claims hold for $\tDex_\lambda(\bE)$.

\begin{lem}\label{lem:exotic-scalar}
For all $\lambda \in \tbX$, we have
\begin{align*}
\ovK(\tnex_\lambda(\O)) &\cong \tnex_\lambda(\ovK), &
\F(\tnex_\lambda(\O)) &\cong \tnex_\lambda(\F), \\
\ovK(\tDex_\lambda(\O)) &\cong \tDex_\lambda(\ovK), &
\F(\tDex_\lambda(\O)) &\cong \tDex_\lambda(\F).
\end{align*}
\end{lem}
\begin{proof}
This follows from the fact that the functors $\cJ_{T_w}$ commute with change of scalars (see~\cite[\S1.2]{br}).
\end{proof}

\begin{lem}\label{lem:tnex-generate}
For any $\bE \in \{\ovK,\O,\F\}$, the category $\Db\Cohmix(\tcN)_\bE$ is generated as a triangulated category by objects of the form $\tnex_\lambda(\bE)\la k\ra$ with $\lambda \in \bX$ and $k \in \Z$.\end{lem}
\begin{proof}
By Lemma~\ref{lem:line-generate}, $\Db\Cohmix(\tcN)$ is the union of the $D_{\conv(\lambda)}$ with $\lambda \in \bX$. 
It is therefore enough to prove that each $D_{\conv(\lambda)}$ is generated by the objects $\tnex_\mu(\bE)\la k\ra$ with $\mu \in \conv(\lambda)$.  We will prove this by induction with respect to the preorder $\preceq$ on $\lambda$.  The base case is that in which $\convo(\lambda) = \varnothing$.  In this case, the claim follows from~\eqref{eqn:tnex-quotient}.

In general, the property~\eqref{eqn:tnex-quotient} means that there exists an object $\cF$ and a diagram
\[
\tnex_\lambda(\bE) \xleftarrow{f} \cF \xrightarrow{g} \cO_\tcN(\lambda)\la \delta_\lambda\ra
\]
such that the cones of both $f$ and $g$ lie in $D_{\convo(\lambda)}$.  Next,~\eqref{eqn:conv-preorder} implies that $\convo(\lambda) = \bigcup_{\nu \in \convo(\lambda)} \conv(\nu)$.  By induction, $D_{\convo(\lambda)}$ is therefore generated by the objects $\tnex_\nu(\bE)\la k\ra$ with $\nu \in \convo(\lambda)$.  The diagram above then shows that the subcategory generated by the $\tnex_\mu(\bE)\la k\ra$ with $\mu \in \conv(\lambda)$ contains $D_{\convo(\lambda)}$ and all the $\cO_\tcN(v\lambda)\la k\ra$ for $v \in W$.  The result follows.
\end{proof}

(The preceding proof shows more generally that for any $\nu \in \tbX/\bX$, the category $\Db\Cohtmix(\tcN)^\nu_\bE$ is generated by objects of the form $\tnex_\lambda(\bE)\la k\ra$ with $\lambda + \bX = \nu$ and $k \in \Z$.)

\subsection{Exceptional sequences}

Let $\le_\Bruhat$ be the \emph{Bruhat order} on $\tbX$, i.e., the order defined so that $\lambda \le_\Bruhat \mu$ iff $w_\lambda$ is smaller than $w_\mu$ is the Bruhat order of $\Wext$.  Let $\le'$ be any refinement of $\le_\Bruhat$ to a total order such that $(\tbX,\le')$ is isomorphic to $(\Z_{\ge 0},\le)$, and such that
\[
\lambda \in \convo(\mu)
\qquad\Longrightarrow\qquad
\lambda <' \mu.
\]
(This last condition makes sense by~\eqref{eqn:conv-preorder}.)  Then the subset $(\bX,\le')$ is also isomorphic to $(\Z_{\ge 0}, \le)$, 

\begin{prop}
\label{prop:exotic-exc}
For $\bE \in \{\ovK,\O,\F\}$, the category $\Db\Cohmix(\tcN)_\bE$ is graded $\Hom$-finite, the collection $\{\tnex_\lambda(\bE)\}_{\lambda \in \bX}$ is a graded exceptional sequence (with respect to the order $\le'$ above), and the collection $\{\tDex_\lambda(\bE)\}_{\lambda \in \bX}$ is a dual sequence.
\end{prop}

In particular, this proposition says that $\tDex_\lambda(\bE) \in \Db\Cohmix(\tcN)_\bE$ for any $\lambda \in \bX$.

\begin{proof}
Using our identification of $\Db\Cohmix(\tcN)_\bE$ with a direct summand of $\Db\Cohtmix(\tcN)_\bE$, it is easy to see that this proposition would follow from the analogous claim in which $G$ and $\bX$ are replaced by $\tG$ and $\tbX$. (Here the fact that $\tDex_\lambda(\bE)$ belongs to $\Db\Cohmix(\tcN)_\bE$ when $\lambda \in \bX$ will follow from the fact that this object is indecomposable and admits a nonzero morphism to $\tnex_\lambda(\bE)$.) For the remainder of the proof, we work in the latter setting.

In the case where $\bE$ is a field, the claim follows from the discussion in~\cite[\S2.5]{mr}. Let us now consider the case where $\bE = \O$. First, the objects $\tnex_\lambda(\bE)$ generate $\Db\Cohtmix(\tcN)_\bE$ by Lemma~\ref{lem:tnex-generate}. It is proved in~\cite[Proposition~5.4]{mr2} that we have
\[
\Hom(\tDex_\lambda(\O), \tnex_\mu(\O)[n]\la k\ra)
\cong \begin{cases}
\O & \text{if $\lambda = \mu$ and $n = k=  0$,} \\
0 & \text{otherwise.}
\end{cases}
\]
(More precisely, in~\cite{mr2} the coefficients considered are a localization of $\Z$; the present setting is completely analogous.) Similar arguments show that the collection $\{\tnex_\lambda(\bE)\}_{\lambda \in \tbX}$ satisfies the conditions on $\Hom$-groups that define graded exceptional sequences. Then Remark~\ref{rmk:Hom-finite} ensures that $\Db\Cohtmix(\tcN)_\bE$ is graded $\Hom$-finite, and Lemma~\ref{lem:dual-crit} tells us that the sequence $\{\tDex_\lambda(\bE)\}_{\lambda \in \tbX}$ is dual to $\{\tnex_\lambda(\bE)\}_{\lambda \in \tbX}$.
\end{proof}

\subsection{Exotic t-structures}

In view of Proposition~\ref{prop:exotic-exc}, using Theo\-rem~\ref{thm:exc-tstruc} we obtain a t-structure on $\Db\Cohmix(\tcN)_\bE$, called the \emph{exotic t-structure}. The heart of this t-structure will be denoted by
\[
\ExCoh^{\Gm}(\tcN)_\bE,
\qquad \bE \in \{\ovK,\O,\F \}.
\]
In accordance with the conventions in the rest of the paper, when $\bE = \O$, we usually omit the subscript and denote the category simply by $\ExCoh^{\Gm}(\tcN)$.
 (Note that although we need to choose a total order $\le'$ that refines $\le_\Bruhat$ in order to invoke Theorem~\ref{thm:exc-tstruc}, the resulting t-structure is independent of that choice; see Remark~\ref{rmk:thm-t-structure}\eqref{it:independence}.)

\begin{lem}
\label{lem:Hom-DbCoh}
Let $\cF,\cG$ in $\Db\Cohmix(\tcN)_\O$.
\begin{enumerate}
\item
\label{it:Hom-DbCoh-K}
The functor $\K(-)$ induces an isomorphism
\[
\K \otimes_\O \Hom(\cF,\cG) \simto \Hom(\K(\cF), \K(\cG)).
\]
\item
\label{it:Hom-DbCoh-F}
There exists a natural short exact sequence
\[
\F \otimes_\O \Hom(\cF,\cG) \hookrightarrow \Hom(\F(\cF), \F(\cG)) \twoheadrightarrow \Tor^\O_1(\F,\Hom(\cF,\cG[1]))
\]
where the first map induced by the functor $\F(-)$.
\end{enumerate}
\end{lem}

\begin{proof}
We explain the proof of~\eqref{it:Hom-DbCoh-F}; the proof of~\eqref{it:Hom-DbCoh-K} is similar and easier. As in~\cite[Proof of Proposition~5.4]{mr2}, the functor $\F$ induces an isomorphism
\[
\F \lotimes_\O R\Hom(\cF,\cG) \simto R\Hom(\F(\cF), \F(\cG)).
\]
Applying the functor $\F \lotimes_\O (-)$ to the truncation triangle $\tau_{\leq 0} R\Hom(\cF,\cG) \to R\Hom(\cF,\cG) \to \tau_{\geq 1} R\Hom(\cF,\cG) \xrightarrow{[1]}$ and using this isomorphism we deduce a distinguished triangle
\[
\F \lotimes_\O \tau_{\leq 0} R\Hom(\cF,\cG) \to R\Hom(\F(\cF),\F(\cG)) \to \F \lotimes_\O \tau_{\geq 1} R\Hom(\cF,\cG) \xrightarrow{[1]},
\]
which induces an exact sequence
\begin{multline*}
\mathsf{H}^{-1} (\F \lotimes_\O \tau_{\geq 1} R\Hom(\cF,\cG)) \to \mathsf{H}^0(\F \lotimes_\O \tau_{\leq 0} R\Hom(\cF,\cG)) \to \Hom(\F(\cF),\F(\cG)) \\
\to \mathsf{H}^0(\F \lotimes_\O \tau_{\geq 1} R\Hom(\cF,\cG)) \to \mathsf{H}^1(\F \lotimes_\O \tau_{\leq 0} R\Hom(\cF,\cG)).
\end{multline*}
in cohomology. Since the functor $\F \lotimes_\O (-)$ is right exact the fifth term in this sequence vanishes, and the second one identifies with $\F \otimes_\O \Hom(\cF,\cG)$. And since $\mathsf{H}^j(\F \lotimes_\O M)=0$ for $j \leq -2$ and any $\O$-module $M$, the first term vanishes and the fourth one identifies with $\Tor^\O_1(\F,\Hom(\cF,\cG[1]))$. We therefore obtain the desired short exact sequence.
\end{proof}

Thanks to Lemma~\ref{lem:exotic-scalar} and Lemma~\ref{lem:Hom-DbCoh}, we are in the setting of~\S\ref{ss:exc-scalar}. We will invoke some results from that section below.

In the lemma below we mention the notion of highest weight categories. For the definition of this notion (due, in slightly different terms, to Cline--Pashall--Scott and Be{\u\i}linson--Ginzburg--Soergel), we refer e.g.~to~\cite[\S 3.5]{ahr-disconn}.

\begin{lem}
\phantomsection
\begin{enumerate}
\item For any $\bE \in \{\ovK,\O,\F\}$, the category $\ExCoh^{\Gm}(\tcN)_\bE$ is noetherian.  If $\bE$ is a field, it is also artinian.\label{it:excoh-noeth}
\item For any $\bE \in \{\ovK,\O,\F\}$, the objects $\tnex_\lambda(\bE)$ and $\tDex_\lambda(\bE)$ lie in $\ExCoh^{\Gm}(\tcN)_\bE$.\label{it:dn-heart}
\item When $\bE$ is a field, $\ExCoh^{\Gm}(\tcN)_\bE$ is a highest-weight category.\label{it:field-hwc}
\item In $\ExCoh^{\Gm}(\tcN)$, the objects $\tnex_\lambda(\O)$ and $\tDex_\lambda(\O)$ are torsion-free.\label{it:dn-torfree}
\end{enumerate}
\end{lem}
\begin{proof}
\eqref{it:excoh-noeth}~This is immediate from Theorem~\ref{thm:exc-tstruc}.

\eqref{it:dn-heart}~When $\bE$ is a field, this is proved in~\cite[\S3.4]{mr} (see also~\cite[Proposition~8.5]{arider}).  Suppose now that $\bE = \O$.  To prove that $\tnex_\lambda(\O)$ lies in the heart, it is enough to show that
\[
\Hom(\tnex_\lambda(\O), \tnex_\mu(\O)[n]\la k\ra) = 0
\]
for all $n < 0$ and all $\mu \in \bX$.
If this were nonzero, the same considerations as in~\cite[Proof of Proposition~5.4]{mr2} would tell us that $\Hom(\tnex_\lambda(\F), \tnex_\mu(\F)[n]\la k\ra)$ is also nonzero for some $n < 0$, contradicting the fact that $\tnex_\lambda(\F) \in \ExCoh^{\Gm}(\tcN)_\F$.

Consider now the object $\mathsf{H}^0(\tDex_\lambda(\O))$.  Since $\F({-})$ is right t-exact, we have
\[
\mathsf{H}^0(\F(H^0(\tDex_\lambda(\O)))) \cong \mathsf{H}^0(\F(\tDex_\lambda(\O))) \cong \mathsf{H}^0(\tDex_\lambda(\F)) \cong \tDex_\lambda(\F).
\]
In other words, after applying $\F({-})$ to the distinguished triangle
\[
\tau^{\le -1}\tDex_\lambda(\O) \to \tDex_\lambda(\O) \to \mathsf{H}^0(\tDex_\lambda(\O)) \xrightarrow{[1]},
\]
the second and third terms become isomorphic, so we must have $\F(\tau^{\le -1}\tDex_\lambda(\O)) = 0$.  But it is easily checked that $\F({-})$ kills no nonzero object, so $\tau^{\le -1}\tDex_\lambda(\O) = 0$, and $\tDex_\lambda(\O) \cong \mathsf{H}^0(\tDex_\lambda(\O))$ belongs to $\ExCoh^{\Gm}(\tcN)$, as desired.

\eqref{it:field-hwc}~See~\cite[\S3.5]{mr} or~\cite[Proposition~8.5]{arider}.

\eqref{it:dn-torfree}~For $\tnex_\lambda(\O)$, this follows from part~\eqref{it:dn-heart} and Lemma~\ref{lem:li-tf}.  For $\tDex_\lambda(\O)$, this follows from Lemma~\ref{lem:exc-scalar}\eqref{it:exc-scalar-f} and the fact that $\F(\tDex_\lambda(\O)) \cong \tDex_\lambda(\F)$ lies in $\ExCoh^{\Gm}(\tcN)_\F$.
\end{proof}

\subsection{Simple objects and their \texorpdfstring{$\O$}{O}-versions}

For each $\lambda \in \bX$, fix some map
\[
c_\lambda = c_\lambda(\O): \tDex_\lambda(\O) \to \tnex_\lambda(\O)
\]
that is a generator of the free $\O$-module $\Hom(\tDex_\lambda(\O), \tnex_\lambda(\O))$.  Such a map becomes an isomorphism after passage to the quotient category
\[
\Db\Cohmix(\tcN)_{\O, \le'\lambda}/\Db\Cohmix(\tcN)_{\O, <'\lambda}.
\]
We denote the base change of this map to $\ovK$ or $\F$ by $c_\lambda(\ovK)$ or $c_\lambda(\F)$, respectively.  As in~\S\ref{ss:special-heart}, we set
\[
\tirrex_\lambda(\bE) := \im(c_\lambda(\bE): \tDex_\lambda(\bE) \to \tnex_\lambda(\bE))
\qquad\text{for $\bE \in \{\ovK,\O,\F\}$}
\]
and
\[
\tirrex^+_\lambda(\O) =
\begin{array}{c}
\text{the unique maximal subobject of $\tnex_\lambda(\O)$ containing $\tirrex_\lambda(\O)$} \\
\text{and such that $\tirrex^+_\lambda(\O)/\tirrex_\lambda(\O)$ is a torsion object.}
\end{array}
\]

As explained in~\S\ref{ss:special-heart}, if $\bE \in \{\ovK,\F\}$ the objects $\tirrex_\lambda(\bE)$ are simple, and the assignment $(\lambda,n) \mapsto \tirrex_\lambda(\bE) \langle n \rangle$ induces a bijection between $\bX \times \Z$ and the set of isomorphism classes of simple objects in $\ExCoh^{\Gm}(\tcN)_\bE$.

Concerning the case $\bE=\O$, Lemma~\ref{lem:li-change} tells us that
\[
\K(\tirrex_\lambda(\O)) \cong \K(\tirrex^+_\lambda(\O))\cong \tirrex_\lambda(\K). 
\]
That lemma also tells us that $\tirrex_\lambda(\O)$ and $\tirrex^+_\lambda(\O)$ are torsion-free objects. We define the \emph{reduced standard and reduced costandard objects for $\tcN_{\F}$} by
\[
\tDred_\lambda(\F) = \F(\tirrex_\lambda(\O))
\qquad\text{and}\qquad
\tnred_\lambda(\F) = \F(\tirrex^+_\lambda(\O)),
\]
respectively. These objects belong to $\ExCoh^{\Gm}(\tcN)_\F$, and there is a sequence of canonical maps
\begin{equation}\label{eqn:canonical-sequence}
\tDex_\lambda(\F) \twoheadrightarrow \tDred_\lambda(\F) \twoheadrightarrow
\tirrex_\lambda(\F)
\hookrightarrow \tnred_\lambda(\F) \hookrightarrow \tnex_\lambda(\F).
\end{equation}
Here the first (resp.~fourth) morphism is obtained from the morphism $\tDex_\lambda(\O) \to \tirrex_\lambda(\O)$, resp.~$\tirrex^+_\lambda(\O) \to \tnex_\lambda(\O)$ by application of the functor $\F$. The surjectivity, resp.~injectivity, of this morphism is checked in the proof of Lemma~\ref{lem:li-change}. The second, resp.~third, morphism in~\eqref{eqn:canonical-sequence} is the projection to the top, resp.~embedding of the socle; see again Lemma~\ref{lem:li-change}. Moreover, $\tirrex_\lambda(\F)$ is also the top of $\tDex_\lambda(\F)$, resp.~the socle of $\tnex_\lambda(\F)$.

\section{Supports of reduced standard objects}
\label{sec:reduced}

\subsection{Reduced standard and costandard perverse-coherent sheaves}

Below we will need the following fact, proved in~\cite[Proposition~2.6]{achar}.  In this statement, $\dom(\lambda)$ denotes the unique dominant weight in the $W$-orbit of a weight $\lambda \in \bX$.

\begin{lem}\label{lem:pi-texact}
Let $\bk \in \{\F,\ovK\}$.  The functor $\pi_*: \Db\Coh^G(\tcN)_\bk \to \Db\Coh^G(\cN)_\bk$ is t-exact for the exotic and perverse-coherent t-structures. For $\lambda \in \bX$, we have
\[
\begin{aligned}
\pi_*\tDex_\lambda(\bk) &\cong \Dpc_{\dom(\lambda)}(\bk), \\
\pi_*\tnex_\lambda(\bk) &\cong \npc_{\dom(\lambda)}(\bk),
\end{aligned}
\qquad
\pi_* \tirrex_\lambda(\bk) = \begin{cases}
\Lpc_{w_0(\lambda)}(\bk) & \text{if $\lambda \in -\bX^+$;} \\
0 & \text{otherwise.}
\end{cases}
\]
\end{lem}

We define the \emph{reduced standard and costandard objects} in $\PCoh(\cN)_\bk$ by setting
\begin{equation}
\label{eqn:red-std}
\Dred_\lambda(\F) = \pi_*\tDred_{w_0\lambda}(\F)
\qquad\text{and}\qquad
\nred_\lambda(\F) = \pi_*\tnred_{w_0\lambda}(\F)
\end{equation}
for $\lambda \in \bX^+$. These objects are perverse-coherent since $\pi_*$ is t-exact.  Moreover, applying $\pi_*$ to the sequence \eqref{eqn:canonical-sequence} (for the weight $w_0\lambda$) provides a sequence of surjections and injections
\begin{equation}
\label{eqn:canonical-sequence2}
\Dpc_\lambda(\F) \twoheadrightarrow \Dred_\lambda(\F) \twoheadrightarrow
\Lpc_\lambda(\F)
\hookrightarrow \nred_\lambda(\F) \hookrightarrow \npc_\lambda(\F).
\end{equation}

\subsection{Statement}
\label{ss:statement-support-Dred}

In view of Theorem~\ref{thm:balanced-sections}, there exists a collection $(x_j : j \in J)$ of balanced nilpotent sections of $\fg$ and a collection $(\varphi_j : j \in J)$ of cocharacters $\Gm \to G$ such that
\begin{enumerate}
\item
for any $j \in J$, the cocharacter $\Spec(\F) \times_{\Spec(\O)} \varphi_j$, resp.~$\Spec(\ovK) \times_{\Spec(\O)} \varphi_j$, is associated with $x_{j,\F}$, resp.~$x_{j,\ovK}$;
\item
for any $j \in J$ we have $\BC(G_\ovK \cdot x_{j,\ovK}) = G_\F \cdot x_{j,\F}$;
\item
the set $(x_{j,\F} : j \in J)$, resp.~$(x_{j,\ovK} : j \in J)$, is a set of representatives for the nilpotent orbits of $G_\F$, resp.~$G_\ovK$.
\end{enumerate}
For $j \in J$, we denote by $\imath_j : \Spec(\O) \to \fg$ the inclusion of the point associated with $x_j$, and by $\imath_j^\F$ its base change to $\F$. 

We will use the representatives $(x_{j,\F} : j \in J)$, resp.~$(x_{j,\ovK} : j \in J)$, to define the Lusztig--Vogan bijections over $\F$ and $\ovK$ as in~\eqref{eqn:LV-bij}.
The goal of this section is to prove the following.

\begin{prop}
\label{prop:support}
For any $\lambda \in \bX^+$, we have
\[
\supp(\Dred_\lambda(\F)) = \supp(\nred_\lambda(\F)) = \overline{\BC(\scO_\lambda^{\ovK})}.
\]
Moreover, if $j \in J$ is such that $x_{j,\ovK} \in \scO_\lambda^{\ovK}$, then
\begin{enumerate}
\item
the complexes of $\O$-modules $(\imath_j)^* \tilde\pi_* \tirrex_\lambda(\O)$ and $(\imath_j)^* \tilde\pi_* \tirrex_\lambda^+(\O)$ are concentrated in degrees $\leq \frac{1}{2} \codim(\scO_\lambda^{\ovK})$;
\item
the $Z_G(x_j)$-modules
\[
\sH^{\frac{1}{2} \codim(\scO_\lambda^{\ovK})}((\imath_j)^* \tilde\pi_*\tirrex_\lambda(\O)) \quad \text{and} \quad \sH^{\frac{1}{2} \codim(\scO_\lambda^{\ovK})}((\imath_j)^* \tilde\pi_* \tirrex_\lambda^+(\O))
\]
are free over $\O$, and we have isomorphisms of $Z_G(x_j)_\ovK$-modules
\begin{align*}
\ovK \otimes_\O \sH^{\frac{1}{2} \codim(\scO_\lambda^{\ovK})}((\imath_j)^* \tilde\pi_* \tirrex_\lambda(\O)) &\cong L_\lambda(\ovK), \\
\ovK \otimes_\O \sH^{\frac{1}{2} \codim(\scO_\lambda^{\ovK})}((\imath_j)^* \tilde\pi_* \tirrex^+_\lambda(\O)) &\cong L_\lambda(\ovK)
\end{align*}
and isomorphisms of $Z_G(x_j)_\F$-modules
\begin{align*}
\F \otimes_\O \sH^{\frac{1}{2} \codim(\scO_\lambda^{\ovK})}((\imath_j)^* \tilde\pi_* \tirrex_\lambda(\O)) &\cong \sH^{\frac{1}{2} \codim(\scO_\lambda^{\ovK})}((\imath^\F_j)^* \Dred_\lambda(\F)), \\
\F \otimes_\O \sH^{\frac{1}{2} \codim(\scO_\lambda^{\ovK})}((\imath_j)^* \tilde\pi_* \tirrex_\lambda^+(\O)) &\cong \sH^{\frac{1}{2} \codim(\scO_\lambda^{\ovK})}((\imath^\F_j)^* \nred_\lambda(\F)).
\end{align*}
\end{enumerate}
\end{prop}

\subsection{Generalities on support}

We start with the following general remark. Let $X$ be a flat noetherian $\O$-scheme, and denote by $X_\F$ and $X_\ovK$ its base-change to $\F$ and $\ovK$ respectively. Then we have natural ``change of scalars'' functors
\[
\F : D^- \Coh(X) \to D^- \Coh(X_\F), \quad \ovK : D^- \Coh(X) \to D^- \Coh(X_\ovK),
\]
which can be described either as the derived functors of the functors $\Coh(X) \to \Coh(X_\F)$ and $\Coh(X) \to\Coh(X_\ovK)$ sending $\cF$ to $\F \otimes_\O \cF$ and $\ovK \otimes_\O \cF$ respectively, or as the derived pullbacks under the projection morphisms $X_\F \to X$ and $X_\ovK \to X$.

Let now $x$ be an $\O$-point of $X$, and denote by $x_\F$ and $x_\ovK$ the closed points in $X_\F$ and $X_\ovK$ obtained by base change.

\begin{lem}
\label{lem:support-K-F}
For any $\cF$ in $D^- \Coh(X)$, we have
\[
x_\ovK \in \supp(\ovK(\cF)) \implies x_{\F} \in \supp(\F(\cF)).
\]
\end{lem}

\begin{proof}
This claim follows from the arguments in~\cite[\S 4.4]{ahr}. (In the end, this proof boils down to the obvious fact that if $M$ is a finitely generated $\O$-module such that $\ovK \otimes_\O M \neq 0$, then $\F \otimes_\O M \neq 0$.)
\end{proof}

\begin{lem}
\label{lem:support-pcoh}
Let $\bk \in \{\F,\ovK\}$, and let $\cF \in \PCoh(\cN)_\bk$.  Its support $\supp(\cF)$ is the union of the orbit closures $\overline{\scO}$ where $\scO$ runs over the nilpotent $G_\bk$-orbits such that some $\IC(\scO,\cV)$ occurs as a composition factor of $\cF$.
\end{lem}
\begin{proof}
We will use the various properties recalled in~\S\ref{ss:pcoh}.  Let $Z$ be the union of orbit closures described above. It is immediate from property~\eqref{it:irr-orbit} that $\supp(\cF) \subset Z$.  To prove equality, we proceed by induction on the length of a composition series of $\cF$.  If $\cF$ is simple, equality holds by property~\eqref{it:irr-orbit} again.  Otherwise, choose a short exact sequence
\[
0 \to \cF' \to \cF \to \cF'' \to 0
\]
with $\cF''$ simple, say $\cF'' = \IC(\scO'',\cV'')$.  Let $Z'$ be the set defined analogously to $Z$ for $\cF'$, so that $Z = Z' \cup \overline{\scO''}$.  By induction $\supp(\cF') = Z'$. Property~\eqref{it:pcoh-support} of~\S\ref{ss:pcoh} implies that for every orbit $\scO$ that is open in $Z$, at least one of $\sH^{\frac{1}{2}\codim \scO}(\cF')|_{\St(\scO)}$ or $\sH^{\frac{1}{2}\codim \scO}(\cF'')|_{\St(\scO)}$ is nonzero, and that the cohomology sheaves vanish on $\St(\scO)$ in all degrees other than $\frac{1}{2}\codim \scO$.  We deduce that $\sH^{\frac{1}{2}\codim \scO}(\cF)|_{\St(\scO)}\ne 0$, and hence that $\supp(\cF) = Z$.
\end{proof}

\subsection{Proof of Proposition~\ref{prop:support}}

We are now in a situation to prove Proposition~\ref{prop:support}. Note that the objects $\nred_\lambda(\F)$ and $\Dred_\lambda(\F)$ make sense in the $(G \times \Gm)_\F$-equivariant setting (cf.~\S\ref{ss:grading}).  Since they are also perverse-coherent, the claims in Proposition~\ref{prop:support} will follow from the following more general result. Here we denote by $Z_{G \times \Gm}(x_j)$ the centralizer of $x_j$ in $G \times \Gm$, where $\Gm$ acts on $\fg$ as in~\S\ref{ss:grading}. We also denote by $Z_{G \times \Gm}(x_j)_\F$, resp.~$Z_{G \times \Gm}(x_j)_\ovK$, the base change of this $\O$-group scheme to $\F$, resp.~$\ovK$. (The following proof is the one point in the paper where it is crucial to work with the extra $\Gm$-action.)

\begin{prop}
\label{prop:support-general}
Let $\cF \in \Db\Coh^{G \times \Gm}(\tcN)$ be an object such that both $\pi_*\ovK(\cF)$ and $\pi_*\F(\cF)$ are perverse-coherent sheaves, and such that there is a nilpotent $G_\ovK$-orbit $\scO$ with $\supp(\pi_*\ovK(\cF)) = \overline{\scO}$.  Then
\[
\supp(\pi_*\F(\cF)) = \overline{\BC(\scO)}.
\]
If moreover $\ovK(\cF) \cong \IC(\scO,\cV)$ for some simple $(G \times \Gm)_\ovK$-equivariant vector bundle $\cV$ on $\scO$, and if $j \in J$ is such that $x_{j,\ovK} \in \scO$, then the complex $(\imath_j)^* \tilde\pi_* \cF$ is concentrated in degrees $\leq \frac{1}{2} \codim(\scO)$, its cohomology in degree $\frac{1}{2} \codim(\scO)$ is free over $\O$, and we have an isomorphism of $Z_{G \times \Gm}(x_j)_\ovK$-modules
\begin{equation}
\label{eqn:isom-stalk-ovK}
\ovK \otimes_\O \sH^{\frac{1}{2} \codim(\scO)}((\imath_j)^* \tilde\pi_*\cF) \cong 
\sH^{0}(u^*\cV),
\end{equation}
where $u: \{x_{j,\ovK}\} \hookrightarrow \scO$ is the embedding, and an isomorphism of $Z_{G \times \Gm}(x_j)_\F$-modules
\begin{equation}
\label{eqn:isom-stalk-F}
\F \otimes_\O \sH^{\frac{1}{2} \codim(\scO)}((\imath_j)^* \tilde\pi_*\cF) \cong \sH^{\frac{1}{2} \codim(\scO)}((\imath^\F_j)^* \tilde\pi_*\F(\cF)).
\end{equation}
\end{prop}

\begin{proof}
Let us first show that $\supp(\pi_*\F(\cF)) \supset \overline{\BC(\scO)}$.  It is enough to show that $\supp(\tilde\pi_*\F(\cF)) \supset \overline{\BC(\scO)}$.  This follows from Lemma~\ref{lem:support-K-F} using $X = \fg$.

For the opposite containment, let $\scO'$ be a nilpotent $G_\F$-orbit that is open in $\supp(\pi_*\F(\cF))$.  For brevity, let $d = \frac{1}{2}\codim \scO'$.  Let $k \in J$ be such that $x_{k,\F} \in \scO'$, and consider an integral Slodowy slice $\mathcal{S}$ as constructed in~\S\ref{ss:slodowy} (for the balanced nilpotent section $x_k$ and the cocharacter $\varphi_k$).  Let $a: G \times \mathcal{S} \to \fg$ be the action map, and let $a^{\ovK} : G_\ovK \times \mathcal{S}^{\ovK} \to \fg_\ovK$, resp.~$a^\F : G_\F \times \mathcal{S}^\F \to \fg_\F$, be its base change to $\ovK$, resp.~to $\F$. We will make use of the natural induction equivalence 
\begin{equation}
\label{eqn:equivalence-equivariant-S}
\theta: \Db \Cohmix(G \times \mathcal{S}) \cong \Db\Coh^{\Gm}(\mathcal{S}),
\end{equation}
as well as the corresponding equivalences $\theta_\F$ and $\theta_\ovK$. (Here, $\Gm$ acts on $G \times \mathcal{S}$ and $\mathcal{S}$ as in Section~\ref{sec:balanced}.)

By property~\eqref{it:pcoh-support} from~\S\ref{ss:pcoh}, we know that $(\pi_*\F(\cF))|_{\St(\scO')}[d]$ is a coherent sheaf (i.e., a complex whose cohomology is concentrated in degree $0$) supported on $\scO'$.  From the diagram in Corollary~\ref{cor:fiber-product-Slodowy} and the flat base change theorem (which is applicable thanks to Proposition~\ref{prop:Slodowy-flat}), we conclude that $(a^\F)^*\tilde\pi_*\F(\cF)[d]$ is a coherent sheaf on $G_\F \times \mathcal{S}^\F$ supported (set-theoretically) on $G_\F \times \{x_{k,\F}\}$, and hence, after passing through~\eqref{eqn:equivalence-equivariant-S}, that the object
\[
\theta_\F(a^\F)^*\tilde\pi_*\F(\cF)[d] \cong \F \lotimes_\O (\theta a^* \tilde\pi_*\cF)[d]
\]
is a coherent sheaf supported (set-theoretically) on $x_{k,\F}$.

Let $\cG = \theta a^* \tilde\pi_*\cF[d] \in \Db \Coh^{\Gm}(\mathcal{S})$.  This can be thought of as a complex of finitely generated graded $\cO(\mathcal{S})$-modules. In particular, each graded component of this complex is bounded and finitely generated over $\O$.  Recall a bounded complex of finitely generated $\O$-modules $M$ is isomorphic to a free $\O$-module (considered as a complex concentrated in degree $0$) iff we have $\sH^{\neq 0}(\F \lotimes_\O M) = 0$.  The previous paragraph tells us that $\sH^{\neq 0}(\F \lotimes_\O \cG) = 0$, so $\cG$ belongs to $\Coh^{\Gm}(\mathcal{S})$, and is free as an $\O$-module.  It follows that $\ovK \otimes_\O \cG$ is nonzero.  Retracing the steps in the previous paragraph, we find that $(\pi_*\ovK(\cF))|_{\St(\scO'')} \ne 0$, where $\scO'':=G_\ovK \cdot x_{k,\ovK}$ is the orbit such that $\BC(\scO'')=\scO'$.  It follows that $\scO''=\scO$, i.e.~that $\scO' = \BC(\scO)$, and then that $\supp(\F(\cF)) = \overline{\BC(\scO)}$.

Suppose now that $\pi_*\ovK(\cF) = \IC(\scO,\cV)$ for some $\cV$ as in the statement, and that $k=j$.  Let $V$ be the (simple) $Z_{G \times \Gm}(x_j)_\ovK$-representation corresponding to $\cV$.  Recall that there is an isomorphism
\begin{equation}\label{eqn:gm-centralizer}
Z_{G \times \Gm}(x_j) \cong \Gm \ltimes Z_G(x_j),
\end{equation}
where $\Gm$ acts on $Z_G(x_j)$ by conjugation via the cocharacter $\varphi_j$ (see, for instance,~\cite[Eq.~(2.6)]{ah}).  Moreover, the induced action of $(\Gm)_\ovK$ on the reductive quotient of $Z_G(x_j)_\ovK$ is trivial, so the copy of $(\Gm)_\ovK$ on the base change to $\ovK$ of the right-hand side of~\eqref{eqn:gm-centralizer} acts on the irreducible representation $V$ by a single character.

Let $t: \{x_j\} \hookrightarrow \mathcal{S}$ be the inclusion map, and let $t_\ovK$, resp.~$t_\F$, be its base change to $\ovK$, resp.~$\F$.  In the notation of the statement of the proposition, we have
\begin{equation}\label{eqn:rep-local}
V = \sH^{0}(u^*\cV).
\end{equation}
By property~\eqref{it:irr-orbit} from~\S\ref{ss:pcoh}, our assumption implies that
\[
\ovK \otimes_\O \cG = 
\theta_\ovK (a^\ovK)^*\tilde\pi_*\ovK(\cF)[d] \cong t_{\ovK*}V.
\]
It is easily checked that the $(\Gm)_\ovK$-action on the left-hand side coming from~\eqref{eqn:equivalence-equivariant-S} is identified with the $(\Gm)_\ovK$-action on right-hand side coming from~\eqref{eqn:gm-centralizer}.  In particular, the graded $\cO(\mathcal{S}^\ovK)$-module $\ovK \otimes_\O \cG$ is concentrated in a single grading degree. Therefore, $\cG$ is also concentrated in a single grading degree, so it is supported scheme-theoretically on $x_j$: there is a free $\O$-module $V_\O$ such that $\cG \cong t_*V_\O$, and such that
\begin{equation}\label{eqn:rep-lattice}
V \cong \ovK \otimes_\O V_\O.
\end{equation}
Since the inclusion map $\imath_j$ factors as
\[
\{x_j\} \xrightarrow{t} \mathcal{S} \hookrightarrow G \times \mathcal{S} \xrightarrow{a} \fg,
\]
we see that
\[
(\imath_j)^*\tilde\pi_*\cF \cong t^*\cG \cong t^*t_*V_\O[-d].
\]
This object has cohomology only in degrees${}\le d$, and its cohomology in degree $d$ is identified with $V_\O$.  In view of~\eqref{eqn:rep-local} and~\eqref{eqn:rep-lattice}, we deduce~\eqref{eqn:isom-stalk-ovK}. Similarly we have
\[
(\imath_j^\F)^*\tilde\pi_*\F(\cF) \cong t_\F^*\F(\cG) \cong t_\F^*t_{\F*} (\F \otimes_\O V_\O)[-d],
\]
which implies~\eqref{eqn:isom-stalk-F}. We leave it to the reader to check that the isomorphisms constructed in this way are $Z_{G \times \Gm}(x_j)_\ovK$-equivariant and $Z_{G \times \Gm}(x_j)_\F$-equivariant respectively.
\end{proof}

\begin{rmk}\phantomsection
\label{rmk:support}
\begin{enumerate}
\item
\label{it:supp-orbit}
In the setting of Proposition~\ref{prop:support-general}, since $\pi_*\F(\cF)$ is a perverse-coherent sheaf whose support is $\overline{\BC(\scO)}$, its restriction to $\St(\BC(\scO))$ is a coherent sheaf placed in degree $\frac{1}{2} \codim(\scO)$, and supported on an infinitesimal neighborhood of $\scO':=\BC(\scO)$. We claim that this coherent sheaf is in fact supported scheme-theoretically on $\scO'$ (so that it coincides with the vector bundle associated with the module $\sH^{\frac{1}{2} \codim(\scO)}((\imath^\F_j)^* \tilde\pi_*\F(\cF))$). Indeed, if $\mathcal{U}$ is the image of $a^\F$, we have $\mathcal{U} \cap \cN_\F = \St(\scO')$ (see Corollary~\ref{cor:fiber-product-Slodowy}). Therefore, it suffices to show that the coherent sheaf $\tilde{\pi}_*\F(\cF)|_{\mathcal{U}}[\frac{1}{2} \codim(\scO')]$ is supported scheme-theoretically on $\scO'$. Now the map $a^\F : G_\F \times \mathcal{S}^\F \to \mathcal{U}$ is flat and surjective, and hence faithfully flat. Our proof shows that $(a^\F)^* \tilde{\pi}_*\F(\cF)|_{\mathcal{U}}[\frac{1}{2} \codim(\scO')]$ is supported on $G_\F \times \{x_{j,\F}\}$, which implies our claim.
\item
Under the identification provided by~\eqref{eqn:gm-centralizer}, the $(\Gm)_\F$- or $(\Gm)_\ovK$-action on the unipotent radical of $Z_G(x_j)_\F$ or $Z_G(x_j)_\ovK$ is contracting (see~\cite[Proposition~5.8]{jantzen-nilp} and~\cite[\S2]{ah}).  Since $\Gm$ acts on the module $V_\O$ constructed later in the proof with a single weight, the unipotent radical must act trivially on the resulting $(\Gm \ltimes Z_G(x_j))_\F$- or $(\Gm \ltimes Z_G(x_j))_\ovK$-representation.  In particular, in Proposition~\ref{prop:support}, the $Z_G(x_j)_\F$-modules\label{it:supp-factor}
\[
\sH^{\frac{1}{2} \codim(\scO_\lambda^{\ovK})}((\imath^\F_j)^* \Dred_\lambda(\F))
\qquad\text{and}\qquad
\sH^{\frac{1}{2} \codim(\scO_\lambda^{\ovK})}((\imath^\F_j)^* \nred_\lambda(\F))
\]
factor through modules for the reductive quotient of $Z_G(x_j)_\F$. (See~\cite[\S 4.2]{har:smcent} for similar arguments in a more general context.)
\end{enumerate}
\end{rmk}

\section{Agreement of the Lusztig--Vogan bijections}
\label{sec:proof-main}

\subsection{Overview}
\label{ss:overview}

Our goal in this section is to compare the Lusztig--Vogan bijections for $G_\ovK$ and $G_\F$. 
To make sense of such a comparison, one first needs to construct a bijection between the sets $\Omega_\ovK$ and $\Omega_\F$ introduced in~\S\ref{ss:pcoh}, which will occupy the first half of the section.

Recall that after choosing a representative for each nilpotent orbit over $\ovK$, resp.~over $\F$, the set $\Omega_\ovK$, resp.~$\Omega_\F$, gets identifies with a set denoted $\Omega'_\ovK$, resp.~$\Omega'_\F$ (see again~\S\ref{ss:pcoh}). To construct our bijection $\Omega_\ovK \overset{\sim}{\leftrightarrow} \Omega_\F$, we will make a coherent choice for these representatives, then construct a bijection between the associated sets $\Omega'_\ovK$ and $\Omega'_\F$.   

More specifically,
recall that we have fixed in~\S\ref{ss:statement-support-Dred} balanced nilpotent sections $(x_j : j \in J)$, which provide in particular (by base change to $\F$ and $\ovK$) representatives $(x_{j,\F} : j \in J)$ and $(x_{j,\ovK} : j \in J)$ for the nilpotent orbits of $G_\F$ and $G_\ovK$ respectively. Using these choices of representatives we obtain sets $\Omega'_\F$ and $\Omega'_\ovK$. As explained above, we want to construct a bijection
\begin{equation}
\label{eqn:bijection-Omega'}
\Omega'_\ovK \simto \Omega'_\F
\end{equation}
which ``extends'' the Bala--Carter bijection from~\S\ref{ss:balanced}, in the sense that the first component of the image of a pair $(\scO,V)$ will be $\BC(\scO)$. Since $\BC(G_\ovK \cdot x_{j,\ovK}) = G_\F \cdot x_{j,\F}$ (see Theorem~\ref{thm:balanced-sections}), constructing such a bijection is equivalent to constructing, for any $j \in J$, a bijection
\begin{equation}
\label{eqn:bijection-simples}
\Irr(Z_G(x_j)_\ovK) \simto \Irr(Z_G(x_j)_\F)
\end{equation}
between the sets of isomorphism classes of simple modules for $Z_G(x_j)_\ovK$ and $Z_G(x_j)_\F$.
The construction of this bijection is explained (after some preliminaries) in~\S\ref{ss:bijection-xj} below.

\subsection{Representation theory of centralizers}
\label{ss:labeling-simples}

We now fix some $j \in J$.
The construction of~\eqref{eqn:bijection-simples} will involve replacing $\K$ by a finite extension $\K' \subset \ovK$.
Note that in this setting, if we
let $\O'$ be
the integral closure of $\O$ in $\K'$, then $\O'$ is a complete discrete valuation ring, which is finite (and free) as an $\O$-module, and has $\K'$ as fraction field (see~\cite[Chap.~II, Proposition~3]{serre}). In particular, since we assume that $\F$ is algebraically closed, the residue field of $\O'$ must still be $\F$. Of course, replacing $\O$ by $\O'$ does not change the field $\ovK$ either. For such datum, we will denote by $Z_G(x_j)_{\O'}$ and $Z_G(x_j)_{\K'}$ the base changes of $Z_G(x_j)$ to $\O'$ and $\K'$ respectively.  

Let $Z_G(x_j)_\ovK^{\mathrm{red}}$ and $Z_G(x_j)_\F^{\mathrm{red}}$ be the reductive quotients of $Z_G(x_j)_\ovK$ and $Z_G(x_j)_\F$ (i.e.~the quotients of these groups by their unipotent radical).  Because $p$ is assumed to be pretty good, it does not divide the order of the finite group $Z(G_\F)/Z(G_\F)^\circ$ (which is equal to the order of the torsion part of $\bX/ \Z R$).  Therefore, by~\cite[Lemma~6.2]{ahjr3}, we have
\begin{equation}\label{eqn:pnotdiv}
p \nmid |Z_G(x_j)_\F/Z_G(x_j)_\F^\circ| = |Z_G(x_j)^{\mathrm{red}}_\F/Z_G(x_j)^{\mathrm{red},\circ}_\F|,
\end{equation}
Next, since $\K'$ is a perfect field, the unipotent radical of $Z_G(x_j)_{\ovK}$ is defined over $\K'$, so that $Z_G(x_j)^{\mathrm{red}}_\ovK$ has a natural $\K'$-form $Z_G(x_j)^{\mathrm{red}}_{\K'}$, which is a connected reductive $\K'$-group. Moreover, for $\bk \in \{\F,\K',\ovK\}$ the pullback functor induces a bijection
\begin{equation}
\label{eqn:Irr-red}
\Irr(Z_G(x_j)^{\mathrm{red}}_\bk) \simto \Irr(Z_G(x_j)_\bk)
\end{equation}
between the corresponding sets of isomorphism classes of simple modules. Below we will simply identify these two sets via this bijection.

Let $\bk$ be either $\F$ or $\ovK$. Then $Z_G(x_j)^{\mathrm{red}}_\bk$ is a possibly disconnected reductive group over $\bk$. The representation theory of such groups is studied in the companion paper~\cite{ahr-disconn}; in particular, these results provide a combinatorial description of $\Irr(Z_G(x_j)^{\mathrm{red}}_\bk)$. We will not need to go into the details of this description. Indeed, we essentially just need one key fact from that paper: according to~\cite[Theorem~3.7]{ahr-disconn}, thanks to~\eqref{eqn:pnotdiv},  the category $\Rep(Z_G(x_j)^{\mathrm{red}}_\F)$ admits a natural structure of highest weight category.

In practice this means that the set $\Irr(Z_G(x_j)^{\mathrm{red}}_\F)$ carries a partial order $\preceq_j$ (defined explicitly in~\cite[\S 3.1]{ahr-disconn}), and that for each simple object $L$ we have a ``standard object'' $\Delta(L)$ and a ``costandard object'' $\nabla(L)$ with maps $\Delta(L) \twoheadrightarrow L \hookrightarrow \nabla(L)$ such that all the composition factors of the kernel, resp.~cokernel, of the first, resp.~second, map are strictly smaller than $L$ for the order $\preceq_j$.

\subsection{Identification of simple modules}
\label{ss:bijection-xj}

For $\bk \in \{\F,\K',\overline{\K}\}$ we will denote by $\mathsf{K}(Z_G(x_j)_\bk)$ the Grothendieck group of the category of finite-dimensional algebraic representations of $Z_G(x_j)_\bk$. This group is a free $\Z$-module, with a basis consisting of the classes of simple representations. Recall that, since $Z_G(x_j)_{\O'}$ is flat (see Theorem~\ref{thm:flatness}), one has a canonical ``decomposition map''
\[
d_{Z_G(x_j)_{\O'}} : \mathsf{K}(Z_G(x_j)_{\K'}) \to \mathsf{K}(Z_G(x_j)_\F),
\]
see~\cite[Theorem~2]{serre-Groth}. This map is compatible with field extensions (for $\K'$) in the obvious sense.

The construction of~\eqref{eqn:bijection-simples} will be given by the following result, which will be proved in~\S\ref{ss:comparison-labelings} below.

\begin{prop}
\label{prop:bijection-simples}
There exists a finite extension $\K_0 \subset \ovK$ of $\K$ such that if $\K'$ contains $\K_0$ then
\begin{enumerate}
\item
\label{it:bijection-simples-1}
the functor sending a $Z_G(x_j)_{\K'}$-module $V$ to $\ovK \otimes_{\K'} V$ (with its natural $Z_G(x_j)_\ovK$-module structure) induces a bijection
\[
\Irr(Z_G(x_j)_{\K'}) \simto \Irr(Z_G(x_j)_\ovK);
\]
\item
\label{it:bijection-simples-2}
$d_{Z_G(x_j)_{\O'}}$ is an isomorphism;
\item
\label{it:bijection-simples-3}
for any simple $Z_G(x_j)_\ovK$-module $V$, there exists a unique simple $Z_G(x_j)_\F$-module $V'$ such that the image under $d_{Z_G(x_j)_{\O'}}$ of the simple $Z_G(x_j)_{\K'}$-module corresponding to $V$ (under the bijection of~\eqref{it:bijection-simples-1}) is the class of $\Delta(V')$; in particular, $V'$ is the unique simple module whose class appears with nonzero coefficient in this image and is maximal for this property (with respect to the order $\preceq_j$);
\item
\label{it:bijection-simples-4}
the map $V \mapsto V'$ (with $V$, $V'$ as in~\eqref{it:bijection-simples-3}) is a bijection
\[
\Irr(Z_G(x_j)_\ovK) \simto \Irr(Z_G(x_j)_\F).
\]
\end{enumerate}
\end{prop}

Gluing the bijections~\eqref{eqn:bijection-simples} over all $j \in J$, we deduce the sought-after bijection~\eqref{eqn:bijection-Omega'}, and hence finally a bijection
\begin{equation}
\label{eqn:bijection-Omega}
\Omega_\ovK \simto \Omega_\F.
\end{equation}
The construction of this bijection involves a choice of balanced nilpotent sections.  We will see in Theorem~\ref{thm:main} that the bijection is in fact independent of this choice.

\subsection{A Levi factor}
\label{ss:Levi-factor}

We continue with the setting of~\S\S\ref{ss:labeling-simples}--\ref{ss:bijection-xj}. Proposition~\ref{prop:bijection-simples} will be deduced from the results of~\cite[Section~4]{ahr-disconn}. But, since that paper considers \emph{reductive} groups over $\O'$, we will need to consider some ``nice'' $\O'$-group scheme which specializes over $\F$ to $Z_G(x_j)^{\mathrm{red}}_\F$, and over $\ovK$ to $Z_G(x_j)^{\mathrm{red}}_\ovK$. (Note that there exists no notion of ``reductive quotient'' over $\O$.) This group scheme will be constructed as a kind of ``Levi factor'' in $Z_G(x_j)$.

More precisely, recall that we have also chosen some cocharacter $\varphi_j$.
We will denote by $Z_G^{\mathrm{Levi}}(x_j)$ the centralizer in $Z_G(x_j)$ of $\varphi_j$, and by $Z^{\mathrm{Levi}}_G(x_j)_\bk$ its base-change to $\bk$ (for $\bk \in \{\F,\O',\K',\overline{\K}\}$).
For $\bk \in \{\F,\K',\overline{\K}\}$, it is well known that $Z_G^{\mathrm{Levi}}(x_j)_\bk$ is a Levi factor of $Z_G(x_j)_\bk$; in other words the restriction to this subgroup of the quotient morphism $Z_G(x_j)_\bk \to Z_G(x_j)^{\mathrm{red}}_\bk$ is an isomorphism (see e.g.~\cite[Proposition~3.2.2]{mcninch2016}).

\begin{lem}
\label{lem:Levi-smooth}
The $\O$-group scheme $Z_G^{\mathrm{Levi}}(x_j)$ is smooth.
\end{lem}

\begin{proof}
Consider the semidirect product $Z_G(x_j) \rtimes \Gm$, where $\Gm$ acts on $Z_G(x_j)$ by conjugation via $\varphi_j$. Then the centralizer in $Z_G(x_j) \rtimes \Gm$ of the subgroup $\{1\} \rtimes \Gm$ is smooth by~\cite[Exp.~11, Corollaire~5.3]{sga3}. On the other hand, it is easy to check that this centralizer coincides with $Z_G^{\mathrm{Levi}}(x_j) \times \Gm$. Hence the latter group scheme is smooth. Using~\cite[Tag~02K5]{stacks-project}, we deduce that $Z_G^{\mathrm{Levi}}(x_j)$ is smooth.
\end{proof}

For $\bk \in \{\F,\overline{\K}\}$, the projection $Z_G(x_j)_\bk \to Z_G(x_j)^{\mathrm{red}}_\bk$ induces a bijection between the groups of connected components of $Z_G(x_j)_\bk$ and $Z_G(x_j)^{\mathrm{red}}_\bk$. Hence the same is true for the embedding $Z^{\mathrm{Levi}}_G(x_j)_\bk \to Z_G(x_j)_\bk$. In view of Theorem~\ref{thm:flatness}, it follows that the groups of connected components of $Z_G^{\mathrm{Levi}}(x_j)_\F$ and $Z_G^{\mathrm{Levi}}(x_j)_\ovK$ have the same cardinality, which we will denote by $m$.

The decomposition of $Z^{\mathrm{Levi}}_G(x_j)_\ovK$ into connected components defines a decomposition of the coordinate ring
\[
\mathcal{O}(Z_G^{\mathrm{Levi}}(x_j)_\ovK) = \bigoplus_{i=1}^m \mathcal{O}(Z_G^{\mathrm{Levi}}(x_j)_\ovK) \epsilon_i
\]
where $(\epsilon_1, \ldots, \epsilon_m)$ are mutually orthogonal idempotents. Let $\K_0 \subset \ovK$ be a finite extension of $\K$ such that all of these elements belong to $\K_0 \otimes_\O \cO(Z_G^{\mathrm{Levi}}(x_j))$. From now on we will assume that $\K'$ contains $\K_0$, and let as above $\O'$ be
the integral closure of $\O$ in $\K'$. 

Denote by $Z_G^{\mathrm{Levi}}(x_j)_{\O'}$ the base change of $Z_G^{\mathrm{Levi}}(x_j)$ to $\O'$. The arguments in~\cite[\S 3.3]{har:smcent} show that each $\epsilon_i$ belongs to $\cO(Z_G^{\mathrm{Levi}}(x_j)_{\O'})$. Then it is not difficult to check that the $\O'$-submodule $\bigoplus_{i=1}^m \O' \cdot \epsilon_i \subset \cO(Z_G^{\mathrm{Levi}}(x_j)_{\O'})$ is a Hopf subalgebra, and that it defines a constant finite $\O'$-group scheme $A_j$, endowed with a morphism $\varpi_j : Z_G(x_j)^{\mathrm{Levi}}_{\O'} \to A_j$. Moreover, the base change of this morphism to $\F$, resp.~$\ovK$, identifies with the projection from $Z_G^{\mathrm{Levi}}(x_j)_\F$, resp.~$Z_G^{\mathrm{Levi}}(x_j)_\ovK$, to its group of connected components. (See~\cite[\S 3.4]{har:smcent} for details.)

\begin{lem}
\label{lem:Levi-centralizer}
The kernel $Z_G^{\mathrm{Levi}}(x_j)_{\O'}^\circ$ of $\varpi$ is a reductive group scheme over $\O'$, and the morphism $\varpi$ identifies $A_j$ with the quotient $Z_G^{\mathrm{Levi}}(x_j)_{\O'} / Z_G^{\mathrm{Levi}}(x_j)_{\O'}^\circ$.
\end{lem}

\begin{proof}
First we note that $\cO(Z_G^{\mathrm{Levi}}(x_j)^\circ_{\O'})$ is a direct summand of $\cO(Z_G^{\mathrm{Levi}}(x_j)_{\O'})$ as an $\O'$-module, and hence is flat over $\O'$. Thus our group scheme is flat.

Let now $\mathbf{A}_j$ be the finite group associated with $A_j$. Then the arguments in the proof of~\cite[Lemma~4.2]{ahr-disconn} show that the morphism $Z_G^{\mathrm{Levi}}(x_j)_{\O'}(\O') \to \mathbf{A}_j$ defined by $\varpi$ is surjective. 
We choose a section $\mathbf{A}_j \to Z_G^{\mathrm{Levi}}(x_j)_{\O'}(\O')$ of this morphism, and denote by $\iota : A_j \to Z_G^{\mathrm{Levi}}(x_j)_{\O'}$ the associated scheme morphism.

The same arguments as in~\cite[Lemma~4.3]{ahr-disconn} show that the natural morphism of $\O'$-schemes $A_j \times Z_G^{\mathrm{Levi}}(x_j)_{\O'}^\circ \to Z_G^{\mathrm{Levi}}(x_j)_{\O'}$ is an isomorphism. It follows that $Z_G^{\mathrm{Levi}}(x_j)_{\O'}^\circ$ is a smooth $\O'$-group scheme whose base changes to $\F$ and $\ovK$ are connected reductive algebraic groups; in other words $Z_G^{\mathrm{Levi}}(x_j)_{\O'}^\circ$ is a reductive group scheme over $\O'$.

The final claim in the statement is also clear from these arguments.
\end{proof}

\subsection{Proof of Proposition~\ref{prop:bijection-simples}}
\label{ss:comparison-labelings}

We consider the setting of~\S\ref{ss:Levi-factor}, and define $\K_0$ as in~\S\ref{ss:Levi-factor}. We will prove that if $\K'$ contains $\K_0$, then Properties~\eqref{it:bijection-simples-1}--\eqref{it:bijection-simples-4} hold.

Property~\eqref{it:bijection-simples-1} is in fact true for any $\K'$ (regardless of whether it contains $\K_0$ or not): arguing as in~\cite[\S 3.6]{serre-Groth}, this property simply follows from the fact that any simple $Z_G(x_j)_\ovK$-modules admits a $\K'$-form, which follows from Proposition~\ref{prop:support}. (This property can also be deduced from the general results of~\cite{ahr-disconn} under the assumption that $\K'$ contains $\K_0$; see~\cite[\S 4.5]{ahr-disconn}.)

In order to prove~\eqref{it:bijection-simples-2}--\eqref{it:bijection-simples-4}, we consider the Grothendieck groups $\mathsf{K}(Z_G^{\mathrm{Levi}}(x_j)_{\K'})$ and $\mathsf{K}(Z_G^{\mathrm{Levi}}(x_j)_{\F})$ of the categories of finite-dimensional algebraic $Z_G^{\mathrm{Levi}}(x_j)_{\K'}$-modules and $Z_G^{\mathrm{Levi}}(x_j)_{\F}$-modules respectively. Since these subgroups are Levi factors in $Z_G(x_j)_{\K'}$ and $Z_G(x_j)_\F$ respectively, pullback induces isomorphisms
\begin{equation}
\label{eqn:K-Levi}
\mathsf{K}(Z_G^{\mathrm{Levi}}(x_j)_{\K'}) \simto \mathsf{K}(Z_G(x_j)_{\K'}), \quad \mathsf{K}(Z_G^{\mathrm{Levi}}(x_j)_{\F}) \simto \mathsf{K}(Z_G(x_j)_{\F}).
\end{equation}
Moreover, since $Z_G^{\mathrm{Levi}}(x_j)_{\O'}$ is smooth (see Lemma~\ref{lem:Levi-smooth}), we have a corresponding decomposition map $d_{Z_G^{\mathrm{Levi}}(x_j)_{\O'}}$, and going back to the definition of this map we see that the diagram
\[
\begin{tikzcd}[column sep=2cm]
\mathsf{K}(Z_G^{\mathrm{Levi}}(x_j)_{\K'}) \ar[r, "d_{Z_G^{\mathrm{Levi}}(x_j)_{\O'}}"] \ar[d, "\eqref{eqn:K-Levi}"] & \mathsf{K}(Z_G^{\mathrm{Levi}}(x_j)_{\F})  \ar[d, "\eqref{eqn:K-Levi}"] \\
\mathsf{K}(Z_G(x_j)_{\K'}) \ar[r, "d_{Z_G(x_j)_{\O'}}"] & \mathsf{K}(Z_G(x_j)_{\K'})
\end{tikzcd}
\]
commutes.

Now the results of~\cite[Section~4]{ahr-disconn} can be applied to the $\O'$-group scheme $Z_G^{\mathrm{Levi}}(x_j)_{\O'}$ thanks to Lemma~\ref{lem:Levi-centralizer}. With this in mind, Property~\eqref{it:bijection-simples-2} is an application of~\cite[Theorem~4.4]{ahr-disconn}, and Properties~\eqref{it:bijection-simples-3}--\eqref{it:bijection-simples-4} follow from~\cite[Lemma~4.7]{ahr-disconn}.

\subsection{Statement}
\label{ss:statement-orbits}

Having explained the construction of~\eqref{eqn:bijection-Omega'}, we can at last state the main result of this paper.

\begin{thm}
\label{thm:main}
The following diagram commutes, where the diagonal arrows are the Lusztig--Vogan bijections for $G_\ovK$ and $G_\F$:
\[
\begin{tikzcd}
&\bX^+ \ar[ld] \ar[rd] & \\
\Omega_{\ovK} \ar[rr, leftrightarrow, "\sim", "\eqref{eqn:bijection-Omega}"'] & & \Omega_\F.
\end{tikzcd}
\]
\end{thm}

Recall that the construction of~\eqref{eqn:bijection-Omega} involves the choice of a set $(x_j: j \in J)$ of balanced nilpotent sections.  Theorem~\ref{thm:main} evidently implies that~\eqref{eqn:bijection-Omega} is in fact independent of these choices.

With respect to these choices, 
Theorem~\ref{thm:main} is equivalent to the assertion that for any $\lambda \in \bX^+$ we have
\begin{enumerate}
\item
\label{it:coincidence-orbits}
$\mathscr{O}^\F_\lambda = \BC(\mathscr{O}^{\ovK}_\lambda)$;
\item
if $j \in J$ is such that $\mathscr{O}^\F_\lambda=G_\F \cdot x_{j,\F}$ (or equivalently such that $\scO_\lambda^\ovK = G_\ovK \cdot x_{j,\ovK}$), then
the simple $Z_G(x_j)_\F$-module $L_\lambda^\F$ corresponds to the simple $Z_G(x_j)_\ovK$-module $L_\lambda^\ovK$ under the bijection~\eqref{eqn:bijection-simples}.
\end{enumerate}

Concerning~\eqref{it:coincidence-orbits}, we note that
by definition we have
\[
\supp(\Lpc_\lambda(\F)) = \overline{\mathscr{O}^\F_\lambda}.
\]
On the other hand, in Proposition~\ref{prop:support} we have proved that
\[
\supp(\Dred_\lambda(\F)) = \overline{\BC(\mathscr{O}^\ovK_\lambda)}.
\]
Now we have a surjection $\Dred_\lambda(\F) \twoheadrightarrow \Lpc_\lambda(\F)$ (see~\eqref{eqn:canonical-sequence2}), so using Lemma~\ref{lem:support-pcoh} we deduce that
\begin{equation}
\label{eqn:inclusion-orbits}
\overline{\mathscr{O}^\F_\lambda} \subset \overline{\BC(\mathscr{O}^\ovK_\lambda)}.
\end{equation}
Hence all that remains to be proved to obtain~\eqref{it:coincidence-orbits} is the opposite containment.

\subsection{Simple modules for centralizers}
\label{ss:simples-centralizers}

In this subsection again we fix some $j \in J$, and set
\[
\bX^+_j := \{\lambda \in \bX^+ \mid \scO^{\ovK}_\lambda = G_\ovK \cdot x_{j,\ovK}\}.
\]
Then the Lusztig--Vogan bijection for $G_\ovK$ induces a bijection between $\bX^+_j$ and the set of isomorphism classes of simple modules for the centralizer $Z_G(x_j)_\ovK$, sending $\lambda$ to $L_\lambda^\ovK$.

\begin{prop}
\label{prop:simples-centralizer}
Assume that for any $\lambda \in \bX^+_j$ we have $\scO^{\F}_\lambda = G_\F \cdot x_{j,\F}$. Then the assignment $\lambda \mapsto L_\lambda^\F$ induces a bijection between $\bX^+_j$ and the set of isomorphism classes of simple modules for $Z_G(x_j)_\F$. Moreover the following commutes:
\[
\begin{tikzcd}
& \bX^+_j \ar[ld, "\lambda \mapsto L_\lambda^\ovK" description] \ar[rd, "\lambda \mapsto L_\lambda^\F" description] & \\
\Irr(Z_G(x_j)_\ovK) \ar[rr, "\sim", "\eqref{eqn:bijection-simples}"'] && \Irr(Z_G(x_j)_\F).
\end{tikzcd}
\]
\end{prop}

\begin{proof}
Under our assumption, each $L_\lambda^\F$ is indeed a simple $Z_G(x_j)_\F$-module, and our assignment is injective because it is obtained by restricting the Lusztig--Vogan bijection for $G_\F$. What remains to be proved is surjectivity (and the commutativity of the diagram).

Let us choose a finite extension $\K'$ of $\K$ as in Proposition~\ref{prop:bijection-simples}.
If we denote by $M_\lambda$ the $Z_G(x_j)_\F$-module $\sH^{\frac{1}{2} \codim(\scO_\lambda^\ovK)}((\imath^\F_j)^* \Dred_\lambda(\F))$, then Proposition~\ref{prop:support} implies that the image under the decomposition map $d_{Z_G(x_j)_{\O'}}$ of the class of the simple $Z_G(x_j)_{\K'}$-module corresponding to $L_\lambda^{\ovK}$ under the bijection of Proposition~\ref{prop:bijection-simples}\eqref{it:bijection-simples-1} is $[M_\lambda]$. Using Proposition~\ref{prop:bijection-simples}\eqref{it:bijection-simples-2}--\eqref{it:bijection-simples-3}, we deduce that the classes $([M_\lambda] : \lambda \in \bX_j^+)$ form a $\Z$-basis of $\mathsf{K}(Z_G(x_j)_\F)$, and moreover that for any $\lambda \in \bX_j^+$ the class $[M_\lambda]$ coincides with the class of the standard $Z_G(x_j)^{\mathrm{red}}_\F$-module whose top is the image of $L_\lambda^\ovK$ under~\eqref{eqn:bijection-simples}.

On the other hand, the isomorphism classes of simple $Z_G(x_j)_\F$-modules also form a basis of $\mathsf{K}(Z_G(x_j)_\F)$, and the classes $([L_\lambda^\F] : \lambda \in \bX^+_j)$ form a subfamily of this basis. Moreover, since $\Dred_\lambda(\F)$ is a quotient of $\Dpc_\lambda(\F)$ (see~\eqref{eqn:canonical-sequence2}), the composition factors of the kernel $\mathcal{K}$ of the surjection $\Dred_\lambda(\F) \twoheadrightarrow \Lpc_\lambda(\F)$ are of the form $\Lpc_\mu(\F)$ with $\mu < \lambda$ (see Property~\eqref{it:simples-pc} in~\S\ref{ss:pcoh}). The support of these composition factors is contained in $\overline{\scO_\lambda^\F}$ (by Proposition~\ref{prop:support} and our assumption), so that $\sH^{\frac{1}{2} \codim(\scO_\lambda^{\ovK}) + 1}((\imath^\F_j)^* \mathcal{K})=0$. Hence we have an exact sequence 
\[
\sH^{\frac{1}{2} \codim(\scO_\lambda^{\ovK})}((\imath^\F_j)^* \mathcal{K}) \to M_\lambda \to L_\lambda^\F \to 0.
\]
In particular, $[M_\lambda]$ has coefficient $1$ on $[L_\lambda^\F]$ and, for $\mu \in \bX^+_j \smallsetminus \{\lambda\}$, if the coefficient of $[M_\lambda]$ on $[L_\mu^\F]$ is nonzero, then $\mu < \lambda$.

If now $V$ is a $Z_G(x_j)_\F$-module, there exist coefficients $(a_\lambda : \lambda \in \bX^+_j)$ in $\Z$ (almost all zero) such that
\[
[V] = \sum_{\lambda \in \bX^+_j} a_\lambda \cdot [M_\lambda].
\]
If $\lambda$ is maximal among the elements such that $a_\lambda \neq 0$, then the remarks above show that the coefficient of $[V]$ on $[L_\lambda^\F]$ (in the basis consisting of classes of simples modules) is $a_\lambda$, and thus nonzero. If we assume that $V$ is simple, this implies that $V \cong L_\lambda^\F$, which concludes the proof of surjectivity.

Finally, since $L_\lambda^\F$ is a composition factor of $M_\lambda$, the remarks above show that it is smaller than the image of $L_\lambda^\ovK$ under~\eqref{eqn:bijection-simples} (with respect to the order $\preceq_j$).
Then a straightforward induction argument (with respect to this order) implies that these modules are in fact isomorphic.
\end{proof}

\subsection{Proof of Theorem~\ref{thm:main}}

In view of Proposition~\ref{prop:simples-centralizer}, all that remains to be proved is that the inclusion~\eqref{eqn:inclusion-orbits} is an equality.
First we observe that if $\scO_\lambda^{\ovK}$ is the zero orbit, then $\BC(\scO_\lambda^{\ovK})$ is also the zero orbit, so that the inclusion~\eqref{eqn:inclusion-orbits} must be an equality.

Let now $\scO \subset \cN_\ovK$ be an orbit, and assume the claim is known for any $\mu \in \bX^+$ such that $\scO_\mu^\ovK \subset \overline{\scO} \smallsetminus \scO$. Then if $\lambda \in \bX^+$ is such that $\scO_\lambda^\ovK = \scO$ and if the embedding $\supp(\Lpc_\lambda(\F)) \subset \supp(\Dred_\lambda(\F))$ is strict, then there exists some orbit $\scO' \subset \overline{\scO} \smallsetminus \scO$ such that $\scO_\lambda^\F = \BC(\scO')$. If $j \in J$ is such that $x_{j,\ovK} \in \scO'$, then $L_\lambda^\F$ is a simple $Z_G(x_j)_\F$-module, which cannot be isomorphic to any $L_\mu^\F$ with $\mu \in \bX_j^+$ (because the Lusztig--Vogan bijection \emph{is} a bijection). But there exists no such module by
Proposition~\ref{prop:simples-centralizer} applied to this choice of $j$. (This proposition is applicable thanks to our induction hypothesis.)

\subsection{Complement: identification of \texorpdfstring{$M_\lambda$}{Mlambda}}

Let $\lambda \in \bX^+_j$, and
recall the $Z_G(x_j)_\F$-module $M_\lambda$ introduced in the proof of Proposition~\ref{prop:simples-centralizer}. In the course of this proof we observed that 
the class of $M_\lambda$ is the class of a standard $Z_G(x_j)^{\mathrm{red}}_\F$-module, which can now be identified with $\Delta(L_\lambda^\F)$ thanks to Theorem~\ref{thm:main}.

In this subsection we note that one can say more about this module.

\begin{prop}\label{prop:compl-std}
For any $j \in J$ and $\lambda \in \bX^+_j$,
there exists an isomorphism of $Z_G(x_j)_\F$-modules
\[
\Delta(L_\lambda^\F) \simto M_\lambda.
\]
\end{prop}
\begin{rmk}
Recall that the module $M_\lambda$ in Proposition~\ref{prop:compl-std} was defined in terms of the perverse-coherent sheaf $\Dred_\lambda(\F)$.  If one starts with $\nred_\lambda(\F)$ instead, the reasoning below can be used to show that the resulting $Z_G(x_j)_\F$-module is isomorphic to $\nabla(L_\lambda^\F)$.
\end{rmk}

\begin{proof}
To fix notation, we set $\scO=\scO^\F_\lambda$.

As seen in the course of the proof of Proposition~\ref{prop:simples-centralizer}, there exists a surjection
$M_\lambda \twoheadrightarrow L_\lambda^\F$.
Our first observation is that in fact $L_\lambda^\bk$ is the top of $M_\lambda$.

For this, let $j: \St(\scO) \hookrightarrow \cN_\F$ be the inclusion, and recall from~\cite[\S4]{ab:pcs} that there is a fully faithful functor $j_{!*}: \PCoh(\St(\scO)) \to \PCoh(\cN_\F)$ whose image is the full subcategory consisting of objects with no nontrivial subobject or quotient supported on $\cN_\F \smallsetminus \St(\scO)$.  Let $i_\scO: \scO \hookrightarrow \St(\scO)$ be the embedding of $\scO$ as a reduced closed subscheme of $\St(\scO)$.  For any vector bundle $\cV$ on $\scO$, we have $\IC(\scO,\cV) = j_{!*} i_{\scO*}\cV$.

Since $\Lpc_\lambda(\F)$ is the top of $\Dred_\lambda(\F)$, it is clear that $\Dred_\lambda(\F)$ has no nonzero quotient supported on $\cN_\F \smallsetminus \St(\scO)$.  Let $\cF$ be the unique maximal subobject of $\Dred_\lambda(\F)$ supported on $\cN_\F \smallsetminus \St(\scO)$.  Then the cokernel of $\cF \hookrightarrow \Dred_\lambda(\F)$ must lie in the essential image of $j_{!*}$; in fact, it is identified with $j_{!*}(\Dred_\lambda(\F)|_{\St(\scO)})$.  According to Remark~\ref{rmk:support}\eqref{it:supp-orbit}, $\Dred_\lambda(\F)|_{\St(\scO)}$ is supported scheme-theoretically on $\scO$.  We therefore have a short exact sequence
\[
0 \to \cF \to \Dred_\lambda(\F) \to \IC(\scO,\mathcal{M}_\lambda) \to 0,
\]
where $\mathcal{M}_\lambda$ is the vector bundle on $\scO$ corresponding to $M_\lambda$.    Now, let $V$ be the top of $M_\lambda$, and let $\cV$ be the corresponding vector bundle.  The quotient map $\mathcal{M}_\lambda \to \cV$ gives rise to map $\IC(\scO,\mathcal{M}_\lambda) \to \IC(\scO,\cV)$.  Here, $\IC(\scO,\cV)$ is a semisimple perverse-coherent sheaf.  The map is nonzero on every summand, so it is surjective.  Composing with $\Dred_\lambda(\F) \to \IC(\scO,\mathcal{M}_\lambda)$, we find that $\IC(\scO,\cV)$ is a semisimple quotient of $\Dred_\lambda(\F)$.  But since the latter has a simple top, we must have $\IC(\scO,\cV) \cong \Lpc_\lambda(\F)$, hence $V \cong L_\lambda^\F$.

Let $\mathscr{C}$ be the Serre subcategory of the category of finite-dimensional algebraic $Z_G(x_j)_\F^{\mathrm{red}}$-modules generated by the simple objects which are smaller than $L_\lambda^\F$ (with respect to $\preceq_j$). Since $[M_\lambda] = [\Delta(L_\lambda^\F)]$, every composition factor of $M_\lambda$ satisfies this condition, so in view of Remark~\ref{rmk:support}\eqref{it:supp-factor} we have $M_\lambda \in \mathscr{C}$.  On the other hand, by the general theory of highest-weight categories, the standard object $\Delta(L_\lambda^\F)$ also belongs to $\mathscr{C}$, and is the projective cover of $L_\lambda^\F$ in this subcategory.  Therefore, there exists a map $\Delta(L_\lambda^\F) \to M_\lambda$ whose composition with the surjection $M_\lambda \twoheadrightarrow L_\lambda^\F$ is surjective. It follows that
this map is surjective. Since $M_\lambda$ and $\Delta(L_\lambda^\F)$ have the same number of composition factors (because they have the same class in $\mathsf{K}$-theory), it must be an isomorphism.
\end{proof}

\appendix
\section{Exceptional sequences with coefficients in a complete local principal ideal domain}
\label{sec:exc-pid}

\newcommand{\sT}{\mathcal{T}}
\newcommand{\il}{\iota^{\mathrm{L}}}
\newcommand{\ir}{\iota^{\mathrm{R}}}
\newcommand{\pil}{\Pi^{\mathrm{L}}}
\newcommand{\pir}{\Pi^{\mathrm{R}}}

Let $\bE$ be a complete local principal ideal domain (i.e.~either a field or a complete discrete valuation ring), and let $\varpi$ be a generator of its unique maximal ideal.  Let $\sT$ be an $\bE$-linear triangulated category.  Throughout this section, we impose the following assumptions on $\sT$:
\begin{itemize}
\item $\sT$ is equipped with a \emph{Tate twist}, i.e., an autoequivalence of triangulated categories $\la 1\ra: \sT \to \sT$.
\item $\sT$ is \emph{graded $\Hom$-finite}, i.e., for any two objects $X, Y \in \sT$, the $\bE$-module
\[
\bigoplus_{k \in \Z} \Hom(X, Y\la k\ra)
\]
is a finitely generated $\bE$-module.
\end{itemize}
Note that the second assumption implies that no nonzero power of $\la 1\ra$ is the identity functor, and that no nonzero object is isomorphic to a Tate twist of itself.

In the case when $\bE$ is a field, there exists a well-known theory of \emph{(graded) exceptional sequences} in such triangulated categories, exposed e.g.~in~\cite{bez,bezru} building in particular on constructions of Bondal--Kapranov~\cite{bk}. One important feature of this construction is that it allows the construction of a t-structure on $\sT$  using the \emph{recollement} formalism of~\cite{bbd}. Our aim in this appendix is to extend this theory to the setting when $\bE$ is a general ring as above. This extension does not require new ideas, but only some care in dealing with new technical difficulties.


\subsection{Noetherian t-structures and recollement}

We begin by studying triangulated categories generated by a single object.  For field coefficients, the following statement can be found in~\cite[Corollary~1]{bez}.  Recall that an abelian category is said to be \emph{noetherian} if every object in it is noetherian, i.e., if every object satisfies the ascending chain condition on subobjects.

\begin{prop}\label{prop:stratum-tstruc}
Let $\sT$ be a $\bE$-linear triangulated category equipped with a Tate twist $\la 1\ra: \sT \to \sT$.  Assume that $\sT$ is graded $\Hom$-finite, and that there exists an object $N \in \sT$ with the following properties:
\begin{enumerate}
\item $\sT$ is generated (as a triangulated category) by objects of the form $N\la k\ra$.
\item We have
\begin{equation}\label{eqn:hom-n-hyp}
\Hom(N, N[n]\la k\ra) =
\begin{cases}
0 & \text{if $n < 0$, or if $n = 0$ and $k \ne 0$,} \\
\bE & \text{if $n = k = 0$,} \\
\text{a free $\bE$-module} & \text{if $n = 1$ (for any $k \in \Z$).}
\end{cases}
\end{equation}
\end{enumerate}
Define an object $\bar N$ as follows:
\[
\bar N = 
\begin{cases}
\mathrm{cone}(N \xrightarrow{\varpi \cdot \id} N) &
\text{if $\bE$ is not a field,} \\
0 & \text{otherwise.}
\end{cases}
\]
Finally, let $\sA \subset \sT$ be the smallest full subcategory that is closed under extensions (in particular, under finite direct sums) and contains the objects
\begin{equation}\label{eqn:heart-gen}
0,
\qquad N\la k\ra,
\qquad \bar N\la k\ra
\qquad\text{for all $k \in \Z$.}
\end{equation}
Then $\sA$ is the heart of a bounded t-structure $(\sT^{\le 0}, \sT^{\ge 0})$ on $\sT$, given by
\begin{align*}
\sT^{\le 0} &=
\begin{array}{@{}l@{}}
\text{the subcategory generated under extensions by objects of the form}\\
\text{$N[n]\la k\ra$ with $n \ge 0$ and $k \in \Z$,}
\end{array},
\\
\sT^{\ge 0} &=
\begin{array}{@{}l@{}}
\text{the subcategory generated under extensions by objects of the form}\\
\text{$N[n]\la k\ra$ and $\bar N[n]\la k\ra$ with $n \le 0$ and $k \in \Z$.}
\end{array}.
\end{align*}
Moreover, $\sA$ is a noetherian category.  If $\bE$ is a field, it is also artinian.
\end{prop}
\begin{proof}
We will prove this in the case where $\bE$ is not a field.  The field case is considerably easier; the appropriate modifications are left to the reader.

We will make extensive use of the ``$*$'' operation from~\cite[\S1.3.9]{bbd}. Recall that this operation is associative, see~\cite[Lemme~1.3.10]{bbd}. Let $\sA^1$ be the full subcategory of $\sT$ consisting of the objects listed in~\eqref{eqn:heart-gen}.  We also set
\[
\sA^k = \underbrace{\sA^1 * \cdots * \sA^1}_{\text{$k$ factors}}.
\]
By definition, we have $\sA = \bigcup_{k \ge 1} \sA^k$. 

\textit{Step 1. The cone of any nonzero morphism $N \to N$ lies in $\sA$.}  Any such morphism is  a scalar multiple of the identity by assumption.  After composing with an automorphism of $N$ (multiplication by a suitable unit in $\bE$), we may assume that the morphism is multiplication by $\varpi^k$ for some $k \ge 0$.  We will now prove the claim by induction on $k$.  For $k = 0$, it is trivial.  For $k \ge 1$, we claim more precisely that 
\begin{equation}\label{eqn:cone-pik}
\mathrm{cone}(\varpi^k) \in \underbrace{\bar N * \cdots * \bar N}_{\text{$k$ factors}} \subset \sA.
\end{equation}
For $k = 1$, this holds by definition.  For $k \ge 2$, factor the map as 
\[
N \xrightarrow{\varpi^{k-1}} N \xrightarrow{\varpi} N.
\]
The octahedral axiom shows that $\mathrm{cone}(\varpi^k) \in \mathrm{cone}(\varpi^{k-1}) * \mathrm{cone}(\varpi)$, and then~\eqref{eqn:cone-pik} follows by induction.

\textit{Step 2. Calculations of $\Hom$-groups among objects of $\sA^1$.}  We will compute various $\Hom$-spaces involving the objects in~\eqref{eqn:heart-gen}.  Note that $\Hom(N,N\la k\ra)$ has been described in the assumptions of the proposition.  

We begin with $\Hom(N, \bar N\la k\ra)$.  We have an exact sequence
\begin{multline*}
\cdots \to \Hom(N ,N\la k\ra) \xrightarrow{\varpi}  \Hom(N, N\la k\ra) \to  \Hom(N, \bar N\la k\ra)\\
 \to  \Hom(N, N[1]\la k\ra) \xrightarrow{\varpi} \Hom(N, N[1]\la k\ra) \to \cdots.
\end{multline*}
Since $\Hom(N, N[1]\la k\ra)$ is a free $\bE$-module, the map between the fourth and fifth terms is injective.  It follows that
\begin{equation}\label{eqn:hom-n-barn}
\Hom(N, \bar N\la k\ra) \cong
\begin{cases}
\bE/\varpi & \text{if $k = 0$,} \\
0 & \text{if $k \ne 0$.}
\end{cases}
\end{equation}

Next, we determine $\Hom(\bar N, N[n]\la k\ra)$ for $n \in \{0,1\}$.  We have an exact sequence
\begin{multline}\label{eqn:hom-barn-n-les}
\cdots \to \Hom(N, N[n-1]\la k\ra) \xrightarrow{\varpi} \Hom(N, N[n-1]\la k\ra) \\
\to  \Hom(\bar N, N[n]\la k\ra)
 \to  \Hom(N, N[n]\la k\ra) \xrightarrow{\varpi} \Hom(N, N[n]\la k\ra) \to \cdots.
\end{multline}
When $n = 0$, the first two terms vanish by assumption, so $\Hom(\bar N, N\la k\ra)$ is the kernel of multiplication by $\varpi$ on $\Hom(N, N\la k\ra)$.  We conclude that
\begin{equation}\label{eqn:hom-barn-n}
\Hom(\bar N, N\la k\ra) = 0 \qquad\text{for all $k \in \Z$.}
\end{equation}
On the other hand, when $n = 1$, both maps labelled $\varpi$ in~\eqref{eqn:hom-barn-n-les} are injective, so $\Hom(\bar N, N[1]\la k\ra)$ is the cokernel of the first such map.  We conclude that 
\begin{equation}\label{eqn:ext-barn-n}
\Hom(\bar N, N[1]\la k\ra) \cong
\begin{cases}
\bE/\varpi & \text{if $k = 0$,} \\
0 & \text{if $k \ne 0$.}
\end{cases}
\end{equation}

Finally, let us compute $\Hom(\bar N, \bar N\la k\ra)$.  Consider the sequence
\begin{multline*}
\cdots \to \Hom(\bar N, N\la k\ra) \to  \Hom(\bar N, \bar N\la k\ra)\\
 \to  \Hom(\bar N, N[1]\la k\ra) \xrightarrow{\varpi} \Hom(\bar N, N[1]\la k\ra) \to \cdots.
\end{multline*}
Using~\eqref{eqn:hom-barn-n} and~\eqref{eqn:ext-barn-n}, we see that the first term always vanishes, and that the map on the second line is zero.  Therefore, $\Hom(\bar N, \bar N\la k\ra) \cong \Hom(\bar N, N[1]\la k\ra)$, so
\begin{equation}\label{eqn:hom-barn-barn}
\Hom(\bar N, \bar N\la k\ra) \cong
\begin{cases}
\bE/\varpi & \text{if $k = 0$,} \\
0 & \text{if $k \ne 0$.}
\end{cases}
\end{equation}

\textit{Step 3. We have $\sA^1 * (\sA^1[1]) \subset (\sA^1[1]) * \sA$.} An object $C$ belongs to $\sA^1 * (\sA^1[1])$ if and only if it occurs in a distinguished triangle $X \to C \to Y[1] \to $ with $X, Y \in \sA^1$.  In other words, $C$ is the cone of some map $f: Y \to X$ in $\sA^1$.  Let us consider all the possibilities for $X$ and $Y$, and show that in each case, $C$ lies in $(\sA^1[1]) * \sA$:
\begin{enumerate}
\item If $f = 0$, then $C \cong X \oplus Y[1]$, so the claim is clear.  In particular, this applies if either $X$ or $Y$ is $0$.
\item Suppose $X= N\la m\ra$ and $Y = N\la k\ra$.  If $m \ne k$, then $f = 0$, and we are done.  If $m = k$, the claim follows from Step~1.
\item Suppose $X=\bar N\la m\ra$ and $Y = N\la k\ra$.  If $m \ne k$, then by~\eqref{eqn:hom-n-barn} we have $f = 0$, and we are done.  Suppose now that $m = k$.  In this case, $\Hom(Y,X) \cong \bE/\varpi$ is a field, so if $f$ is nonzero, then it must be the composition of the canonical map $N\la m\ra \to \bar N\la m\ra$ with an automorphism of $\bar N\la m\ra$.  By the definition of $\bar N$, the cone of the canonical map $N\la m\ra \to \bar N\la m\ra$ is $N[1]\la m\ra$.
\item Suppose $X = N\la m\ra$ and $Y = \bar N\la k\ra$.  By~\eqref{eqn:hom-barn-n}, $f = 0$.
\item Suppose $X = \bar N\la m\ra$ and $Y = \bar N\la k\ra$.  If $m \ne k$, then by~\eqref{eqn:hom-barn-barn} we have $f = 0$, and we are done.  Suppose now that $m = k$. Since $\Hom(Y,X) \cong \bE/\varpi$ is a field, any nonzero morphism $Y \to X$ is an isomorphism.  The claim follows.
\end{enumerate}

\textit{Step 4. We have $\sA * (\sA^1[1]) \subset (\sA^1[1]) * \sA$.}  It is enough to show that $\sA^k * (\sA^1[1]) \subset (\sA^1[1]) * \sA$ for all $k \ge 1$. We proceed by induction on $k$.  The case where $k = 1$ has been done in Step~3.  For $k > 1$, we have
\begin{multline*}
\sA^k * (\sA^1[1]) = \sA^1 * (\sA^{k-1} *(\sA^1[1])) 
\subset (\sA^1 * (\sA^1[1])) * \sA  \\
\subset (\sA^1[1]) * \sA * \sA
\subset (\sA^1[1]) * \sA.
\end{multline*}
Here the first inclusion uses the induction hypothesis, and the second one the result of Step~3.

\textit{Step 5. We have $\sA * (\sA[1]) \subset (\sA[1]) *\sA$.}  Again, it is enough to show that $\sA * (\sA^k[1]) \subset (\sA^k[1]) * \sA$ for all $k \ge 1$.  For $k = 1$, this has been done in Step~4.  For $k > 1$, we have
\begin{multline*}
\sA * (\sA^k[1]) = (\sA * (\sA^{k-1}[1])) * \sA^1[1] 
\subset (\sA^{k-1}[1]) * (\sA * (\sA^1[1])) \\
\subset \sA^{k-1}[1] * (\sA^1[1] )* \sA = (\sA^k[1]) * \sA,
\end{multline*}
as desired.

\textit{Step 6. The category $\sA$ is the heart of a bounded t-structure on $\sT$ as claimed in the statement of the proposition.}
In~\cite[\S 1.2.3]{bbd}, the authors define the notion of \emph{admissible} morphisms with respect to a full subcategory. The precise definition of this notion will not be important for us, since
according to~\cite[\S1.3.11(ii)]{bbd} the statement proved in Step~5 is equivalent to the assertion that every morphism in $\sA$ is admissible (with respect to $\sA$). According to~\cite[Proposition~1.2.4]{bbd}, this implies that $\sA$ is an admissible abelian subcategory of $\sT$ in the sense of~\cite[D\'efinition~1.2.5]{bbd}. Finally, since $N$ generates $\sT$, applying~\cite[Proposition~1.3.13]{bbd} we obtain that $\sA$ is the heart of a (unique) $t$-structure on $\sT$.

An explicit description of this t-structure appears in the paragraph preceding~\cite[Proposition~1.3.13]{bbd}: $\sT^{\le 0}$ and $\sT^{\ge 0}$ are the categories generated under extensions by $\sA[n]$ with $n \ge 0$ and $n \le 0$, respectively.  Of course, we may replace $\sA$ by $\sA^1$.  For $\sT^{\ge 0}$, the resulting description is as in the statement of the present proposition.  For $\sT^{\le 0}$, we may further omit $\bar N$ from the description, since $\bar N \in N * N[1]$.

\textit{Step 7. Every object $M \in \sA$ admits a filtration $0 = M_0 \subset M_1 \subset \cdots \subset M_n = M$ such that each subquotient $M_i/M_{i-1}$ is isomorphic to either $N\la k\ra$ or $\bar N\la k\ra$ for some $k \in \Z$.} This is just a restatement of the fact that $\sA$ is generated under extensions by the objects $N\la k\ra$ and $\bar N\la k\ra$, translated into the language of abelian categories.

\textit{Step 8. Let $M$ be a nonzero subobject of $N$.  Then $M$ contains a subobject isomorphic to $N$.}  Choose a filtration of $M$ as in Step~7.  The first step in this filtration, $M_1$, is a subobject of $M$ and of $N$ that is isomorphic to some $N\la k\ra$ or $\bar N\la k\ra$.  But by~\eqref{eqn:hom-n-hyp} and~\eqref{eqn:hom-barn-n}, we must have $M_1 \cong N$.

\textit{Step 9. The category $\sA$ is noetherian.}  In view of Step~7, it is enough to prove that the objects $N$ and $\bar N$ are noetherian.  We actually claim that $\bar N$ is a simple object.  To prove this, it is enough to show that any nonzero map $Y \to \bar N$ in $\sA$ is surjective.  Suppose first that $Y = N\la k\ra$ or $Y = \bar N\la k\ra$.  If $k \ne 0$, there is no nonzero map $Y \to \bar N$; if $k = 0$, we saw in Step~3 that the cone of any nonzero map $Y \to \bar N$ lies in $\sA[1]$, so the map is surjective.  For general $Y \in \sA$, the claim then follows by induction on the length of the filtration from Step~7.

It remains to show that $N$ is noetherian.  Suppose we have an ascending chain of subobjects $M_1 \subset M_2 \subset \cdots $ in $N$.  By Step~8, $M_1$ contains a subobject $Q$ that is isomorphic to $N$.  The composition of the inclusion maps $Q \hookrightarrow M_1 \hookrightarrow N$ may be identified with $\varpi^k: N \to N$ for some $k \ge 0$.  To show that our ascending chain is eventually constant, it is enough to show that the chain of subobjects $M_1/Q \subset M_2/Q \subset \cdots $ in $\mathrm{cok}(\varpi^k: N \to N)$ is eventually constant.  The cokernel of $\varpi^k: N \to N$ is described in~\eqref{eqn:cone-pik}: it is a finite extension of simple objects, so it is noetherian.
\end{proof}

\begin{rmk}
In the setting of Proposition~\ref{prop:stratum-tstruc}, suppose we assume in addition that $\Hom(N,N[1]\la k\ra) = 0$ and that $\Hom(N,N[2]\la k\ra)$ is a free $\bE$-module (for all $k \in \Z$).  One can then show that $N$ is a projective object in $\sA$, and that the functor
\[
\bigoplus_{k \in \Z} \Hom(N\la -k\ra, {-}): \sA \to \bE\lgmod
\]
is an equivalence of categories, where $\bE\lgmod$ is the category of finitely generated graded $\bE$-modules.
\end{rmk}

The following fact is probably well-known, but we could not find a reference, so we include a proof.

\begin{lem}\label{lem:glue-noeth}
Let $\sT_F$, $\sT$, and $\sT_U$ be triangulated categories, and suppose we have a recollement diagram
\[
\begin{tikzcd}
\sT_F \ar[r, "\iota" description] &
\sT \ar[r, "\Pi" description] \ar[l, bend left, "\il"] \ar[l, bend right, "\ir"'] &
\sT_U \ar[l, bend left, "\pil"] \ar[l, bend right, "\pir"']
\end{tikzcd}
.
\]
Suppose $\sT_F$ and $\sT_U$ are equipped with t-structures, and let $\sA_F$ and $\sA_U$ be their hearts, respectively.  Let $\sA$ be the heart of the t-structure on $\sT$ obtained by recollement.  If $\sA_F$ and $\sA_U$ are noetherian categories, then $\sA$ is as well.
\end{lem}
\begin{proof}
As explained in~\cite[\S1.4.17.1]{bbd}, the functor $\iota$ identifies $\sA_F$ with a Serre subcategory of $\sA$.  In particular, any object of $\sA$ that is in the image of $\sA_F$ is noetherian.

Let $X \in \sA$.  We will show that $X$ is noetherian.  By~\cite[Proposition~1.4.17(ii)]{bbd}, we have a right exact sequence
\[
H^0(\pil\Pi(X)) \to X \to H^0(\iota\il(X)) \to 0.
\]
As explained above, the last object is noetherian, so it is enough to prove that $H^0(\pil\Pi(X))$ is noetherian.  Apply~\cite[Proposition~1.4.17(ii)]{bbd} again to obtain a left exact sequence
\[
0 \to H^0(\iota\ir H^0(\pil\Pi(X))) \to H^0(\pil\Pi(X)) \to H^0(\pir\Pi H^0(\pil\Pi(X))).
\]
Here, the first term is noetherian.  Since $\Pi$ is t-exact, and $\Pi \circ \pil \cong \id$, the last term can be identified with $H^0(\pir\Pi(X))$.  We have reduced the problem to showing that the image of $H^0(\pil\Pi(X)) \to H^0(\pir\Pi(X))$ is noetherian.  More generally, we will show that for any $Y \in \sA_U$, the image of the natural map
\[
H^0(\pil Y) \to H^0(\pir Y)
\]
is noetherian.  Following~\cite[D\'efinition~1.4.22]{bbd}, we denote this image by $\Pi_{!*}(Y)$.

Let $Z_1 \subset Z_2 \subset \cdots$ be an ascending chain of subobjects of $\Pi_{!*}(Y)$.  Then $\Pi(Z_1) \subset \Pi(Z_2) \subset \cdots $ is an ascending chain of subobjects of $Y \in \sA_U$.  Since $Y$ is noetherian, this chain is eventually constant: there is a subobject $Y' \subset Y$ such that $\Pi(Z_k) = Y'$ for all $k \gg 0$.  By discarding finitely many terms from the beginning of our sequence, we may assume that $\Pi(Z_k) = Y'$ for all $k \ge 1$.

By adjunction, for each $k$, we have a map
\begin{equation}\label{eqn:noeth-im}
H^0(\pil Y') \to Z_k.
\end{equation}
According to~\cite[Proposition~1.4.17(i)]{bbd}, the image of this map has no nonzero quotient in $\sA_F$.  On the other hand, $Z_k$, as a subobject of $\Pi_{!*}(Y)$, has no nonzero subobject in $\sA_F$, and hence neither does the image of~\eqref{eqn:noeth-im}.  By~\cite[Corollaire~1.4.25]{bbd}, we conclude that the image of~\eqref{eqn:noeth-im} is canonically identified with $\Pi_{!*}(Y')$.

Let $Z'_k = Z_k/\Pi_{!*}(Y')$.  To prove that $Z_1 \subset Z_2 \subset \cdots$ is eventually constant, it is enough to show that
\[
Z'_1 \subset Z'_2 \subset \cdots \subset \Pi_{!*}(Y)/\Pi_{!*}(Y')
\]
is eventually constant.  By construction, we have $\Pi(Z'_k) \cong Y'/Y' = 0$, so each $Z'_k$ lies in (the essential image of) $\sA_F$.  By adjunction, the inclusion map $Z'_k \to \Pi_{!*}(Y)/\Pi_{!*}(Y')$ factors through $H^0(\iota \ir (\Pi_{!*}(Y)/\Pi_{!*}(Y')))$.  Denote the latter object by $Y''$, and rewrite the chain of subobjects as
\[
Z'_1 \subset Z'_2 \subset \cdots \subset Y''.
\]
Since $Y'' \in \sA_F$, it is noetherian, and this chain of subobjects is eventually constant.
\end{proof}

\subsection{Exceptional sequences and their duals}

The following notion is the main focus of this appendix.  We continue to assume that $\sT$ is $\bE$-linear, equipped with a Tate twist, and that it is graded $\Hom$-finite.

\begin{defn}
\label{defn:exc-sequence}
Let $(I, \le)$ be an ordered set that is isomorphic to a subset of $(\Z_{\ge 0},\le)$.  An \emph{$\bE$-linear graded exceptional sequence} in $\sT$ is a collection of objects $\{\nabla_i\}_{i \in I}$ such that the following conditions hold:
\begin{enumerate}
\item If $i < j$, then $\Hom(\nabla_i,\nabla_j[n]\la k\ra) = 0$ for all $n, k \in \Z$.
\item We have $\Hom(\nabla_i,\nabla_i[n]\la k\ra) = 0$ unless $n = k = 0$, and $\End(\nabla_i) \cong \bE$.
\item The collection of objects $\{\nabla_i\la k\ra\}_{i \in I, k \in \Z}$ generates $\sT$ as a triangulated category.
\end{enumerate}
\end{defn}

There is an ungraded variant of this notion as well (applicable to categories without a Tate twist), obtained by simply omitting all mentions of $\la k\ra$ from the three axioms.  All the results in this section are stated in the graded case, but the corresponding statements in the ungraded case also hold (with the same proofs).

Given a graded exceptional sequence $\{\nabla_i\}_{i \in I}$ in $\sT$ and an element $i \in I$, we let
\[
\sT_{< i}, \qquad\text{resp.}\qquad
\sT_{\le i}
\]
denote the full triangulated subcategory of $\sT$ generated by the objects of the form $\nabla_j\la k\ra$ with $k \in \Z$ and $j < i$, resp.~$j \le i$.  Let
\[
\Pi_i: \sT_{\le i} \to \sT_{\le i}/\sT_{<i}
\]
be the Verdier quotient functor.  It is clear that the quotient category $\sT_{\le i}/\sT_{<i}$ is generated by the objects of the form $\Pi_i(\nabla_i)\la k\ra$.

\begin{defn}
\label{defn:exc-dual}
Let $\{\nabla_i\}_{i \in I}$ be a graded exceptional sequence in $\sT$, and let $\{\Delta_i\}_{i \in I}$ be another collection of objects indexed by $I$.  The set $\{\Delta_i\}_{i \in I}$ is said to be a \emph{dual sequence} to $\{\nabla_i\}_{i \in I}$ if for each $i \in I$, we have
\begin{enumerate}
\item If $i < j$, then $\Hom(\Delta_j, \nabla_i[n]\la k\ra) = 0$ for all $n, k \in \Z$.
\item For each $i \in I$, we have $\Delta_i \in \sT_{\le i}$ and $\Pi_i(\Delta_i) \cong \Pi_i(\nabla_i)$.
\end{enumerate}
\end{defn}

The exceptional sequence $\{\nabla_i\}_{i \in I}$ is said to be \emph{dualizable} if there exists some dual sequence to it. (It is easily seen using Lemma~\ref{lem:exc-basic} below that a dual sequence is unique if it exists, which justifies the terminology.)

\begin{lem}
\label{lem:exc-basic}
Let $\{\nabla_i\}_{i \in I}$ be a graded exceptional sequence, and let $\{\Delta_i\}_{i \in I}$ be a dual sequence.  
\begin{enumerate}
\item If $X \in \sT_{< i}$ then $\Hom(X, \nabla_i[n]\la k\ra) = 0$ and $\Hom(\Delta_i[n]\la k\ra, X) = 0$ for all $n,k \in \Z$.\label{it:exc-van}
\item For all $X \in \sT_{\le i}$, the natural maps
\begin{align*}
\Hom(X,\nabla_i[n]\la k\ra) &\to \Hom(\Pi_i(X), \Pi_i(\nabla_i)[n]\la k\ra), \\
\Hom(\Delta_i[n]\la k\ra,X) &\to \Hom(\Pi_i(\Delta_i)[n]\la k\ra, \Pi_i(X))
\end{align*}
are isomorphisms for all $n, k \in \Z$.\label{it:exc-pi}
\item If $i \ne j$, we have $\Hom(\Delta_i, \nabla_j[n]\la k\ra) = 0$ for all $n,k \in \Z$.\label{it:exc-orth}
\item For all $i \in \Z$, there are natural isomorphisms\label{it:exc-quot}
\begin{multline*}
\Hom(\nabla_i, \nabla_i[n]\la k\ra)
\cong \Hom(\Delta_i, \Delta_i[n]\la k\ra)
\cong \Hom(\Delta_i, \nabla_i[n]\la k\ra) \\
\cong \Hom(\Pi_i(\nabla_i), \Pi_i(\nabla_i)[n]\la k\ra)
\cong
\begin{cases}
\bE & \text{if $n = k = 0$,} \\
0 & \text{otherwise.}
\end{cases}
\end{multline*}
\end{enumerate}
\end{lem}
\begin{proof}
\eqref{it:exc-van}~It is enough to check this when $X$ belongs to some class of objects that generate $\sT_{<i}$.  For instance, it is enough to prove it in the case where $X = \nabla_j\la m\ra$ for some $j < i$.  In this case, the claim holds by definition.

\eqref{it:exc-pi}~This follows from part~\eqref{it:exc-van} by~\cite[Proposition~2.3.3(a), parts (iii) and~(v)]{verdier}.

\eqref{it:exc-orth}~If $i > j$, this holds by definition.  If $i < j$, then $\Delta_i \in \sT_{< j}$, so this follows from part~\eqref{it:exc-van}.

\eqref{it:exc-quot}~Identify $\Pi_i(\Delta_i)$ with $\Pi_i(\nabla_i)$. Part~\eqref{it:exc-pi} tells us that each of the first three $\Hom$-spaces is naturally isomorphic to the fourth one.  The space $\Hom(\nabla_i, \nabla_i[n]\la k\ra)$ is as described by definition.
\end{proof}

\begin{rmk}
\label{rmk:Hom-finite}
In this appendix we assume throughout that the category $\sT$ is graded $\Hom$-finite. However, if we are given a sequence $\{\nabla_i\}_{i \in I}$ of objects in a triangulated category $\sT$ (assumed only $\bE$-linear and equipped with a Tate twist) satisfying the properties in Definition~\ref{defn:exc-sequence} and a sequence $\{\Delta_i\}_{i \in I}$ of objects satisfying the conditions of Definition~\ref{defn:exc-dual}, then $\sT$ automatically satisfies a stronger finiteness property; namely, for any objects $X,Y$ the $\bE$-module
\[
\bigoplus_{n,m \in \Z} \Hom_{\sT}(X,Y \langle m \rangle [n])
\]
is finitely generated. In fact, Lemma~\ref{lem:exc-basic}\eqref{it:exc-orth}--\eqref{it:exc-quot} (whose proof does not involve the ``graded $\Hom$-finite'' condition) shows that this condition holds when $X=\Delta_i$ and $Y=\nabla_j$; the general case follows since the collections $\{\nabla_i \langle k \rangle\}_{i \in I, k \in \Z}$ and $\{\Delta_i \langle k \rangle\}_{i \in I, k \in \Z}$ both generate $\sT$ as a triangulated category.
\end{rmk}

Note that the proof of the first isomorphism in Lemma~\ref{lem:exc-basic}\eqref{it:exc-pi} does not involve the dual sequence in any way; it holds even if the exceptional sequence is not assumed to be dualizable.

\begin{lem}\label{lem:dual-crit}
Let $\{\nabla_i\}_{i \in I}$ be a graded exceptional sequence in $\sT$, and let $\{\Delta_i\}_{i \in I}$ be another sequence of objects in $\sT$.  This sequence is a dual sequence to $\{\nabla_i\}_{i \in I}$ if and only if we have
\[
\Hom(\Delta_i, \nabla_j[n]\la k\ra) \cong
\begin{cases}
\bE & \text{if $i = j$ and $n = k = 0$,} \\
0 & \text{otherwise.}
\end{cases}
\]
\end{lem}
\begin{proof}
If $\{\Delta_i\}_{i \in I}$ is a dual sequence, the $\Hom$-groups are as described by parts~\eqref{it:exc-orth} and~\eqref{it:exc-quot} of Lemma~\ref{lem:exc-basic}.

For the opposite implication, the first condition in Definition~\ref{defn:exc-dual} holds by assumption; we need only prove the second condition.  For each $i \in I$, there exists some $j$ such that $\Delta_i \in \sT_{\le j}$.  Assume that $j$ is minimal with respect to this property, i.e., that $\Delta_i \notin \sT_{< j}$.  If $j < i$, our assumptions would imply that $\Hom(\Delta_i,X) = 0$ for all $X \in \sT_{\le j}$, which is absurd.  We therefore have $j \ge i$.  Since $\Delta_i \notin \sT_{<j}$, we must have $\Pi_j(\Delta_i) \ne 0$.  Since the quotient category $\sT_{\le j}/\sT_{<j}$ is generated by the objects $\Pi_j(\nabla_j)\la k\ra$, we must have
\[
\Hom(\Pi_j(\Delta_i), \Pi_j(\nabla_j)[n]\la k\ra) \ne 0
\]
for some integers $n, k \in \Z$.  As noted above, we may use the first isomorphism in Lemma~\ref{lem:exc-basic}\eqref{it:exc-pi} even without the assumption that $\{\nabla_i\}_{i \in I}$ is dualizable.  That isomorphism tells us that
\[
\Hom(\Delta_i, \nabla_j[n]\la k\ra) \ne 0.
\]
We therefore have $j = i$, i.e., $\Delta_i \in \sT_{\le i}$.  

Next, choose a map $c: \Delta_i \to \nabla_i$ corresponding to a generator of the free $\bE$-module $\Hom(\Delta_i,\nabla_i)$.  Let $K$ be the cone of this map, and consider the long exact sequence
\begin{multline*}
\cdots \to \Hom(\nabla_i,\nabla_i[n-1]\la k\ra) \to \Hom(\Delta_i,\nabla_i[n-1]\la k\ra) \to \\
\Hom(K, \nabla_i[n]\la k\ra) \to \Hom(\nabla_i, \nabla_i[n]\la k\ra) \to \Hom(\Delta_i,\nabla_i[n]\la k\ra) \to \cdots.
\end{multline*}
If $k \ne 0$, or if $n \ne 0,1$, then the first, second, fourth, and fifth terms vanish, so $\Hom(K,\nabla_i[n]\la k\ra) = 0$ as well.  If $k = 0$ and $n = 0$, the first two terms vanish, and the last two terms are isomorphic (the map between them sends $\id \in \Hom(\nabla_i,\nabla_i)$ to the generator $c \in \Hom(\Delta_i,\nabla_i)$), so $\Hom(K, \nabla_i) = 0$.  If $k = 0$ and $n = 1$, similar reasoning with the first two terms yields $\Hom(K, \nabla_i[1]) = 0$.

We have shown that $\Hom(K,\nabla_i[n]\la k\ra) = 0$ for all $n,k \in \Z$.  By construction, $K \in \sT_{\le i}$.  Apply Lemma~\ref{lem:exc-basic}\eqref{it:exc-pi} again to conclude that $\Hom(\Pi_i(K), \Pi_i(\nabla_i)[n]\la k\ra) = 0$ for all $n,k \in \Z$.  It follows that $\Pi_i(K) = 0$, and hence that $c: \Delta_i \to \nabla_i$ becomes an isomorphism in $\sT_{\le i}/\sT_{< i}$, as desired.
\end{proof}

\begin{rmk}
Lemma~\ref{lem:dual-crit} implies that the property of being dualizable, and the dual sequence, do not depend on the order on $I$; i.e.~if a collection of objects parametrized by a set $I$ is exceptional for two different orders $\leq$ and $\preceq$ on $I$, then it is dualizable as a sequence parametrized by $(I,\leq)$ iff it is dualizable as a sequence parametrized by $(I,\preceq)$, and in this case the dual sequences agree.
\end{rmk}

\subsection{The t-structure associated with an exceptional collection}

\begin{prop}\label{prop:exc-recolle}
Let $\{\nabla_i\}_{i \in I}$ be a dualizable graded exceptional sequence in $\sT$.  For each $i \in I$, the quotient functor $\Pi_i: \sT_{\le i} \to \sT_{\le i}/\sT_{< i}$ and the inclusion functor $\iota_i: \sT_{<i} \to \sT_{\le i}$ both admit left and right adjoints.  Together, these functors give a recollement diagram
\[
\begin{tikzcd}
\sT_{<i} \ar[r, "\iota_i" description] &
\sT_{\le i} \ar[r, "\Pi_i" description] \ar[l, bend left, "\il_i"] \ar[l, bend right, "\ir_i"'] &
\sT_{\le i}/\sT_{< i} \ar[l, bend left, "\pil_i"] \ar[l, bend right, "\pir_i"']
\end{tikzcd}
.
\]
\end{prop}
\begin{proof}
For brevity, in the proof we will omit the subscript ``$i$'' from the names of the various functors in the diagram above.

\textit{Step 1. The functor $\Pi$ admits a right adjoint $\pir$.}
Let $\sT^\nabla_i \subset \sT_{\le i}$ be the full triangulated subcategory generated by the objects of the form $\nabla_i\la k\ra$ with $k \in \Z$.  We claim that the functor
\[
\Pi|_{\sT^\nabla_i}: \sT^\nabla_i \to \sT_{\le i}/\sT_{< i}
\]
is an equivalence of categories.  Indeed, lemma~\ref{lem:exc-basic}\eqref{it:exc-pi} implies that this functor is fully faithful, and since $\sT_{\le i}/\sT_{< i}$ is generated by the objects $\Pi(\nabla_i)\la k\ra$, it is also essentially surjective.

Let $\pir$ denote the composition
\[
\sT_{\le i}/\sT_{< i} \xrightarrow{(\Pi|_{\sT^\nabla_i})^{-1}} \sT^\nabla_i
\xrightarrow{\text{inclusion}}
\sT_{\le i}.
\]
Lemma~\ref{lem:exc-basic}\eqref{it:exc-pi} again implies that for any $X \in \sT_{\le i}$ and $Y \in \sT^\nabla_i$, the map
\begin{equation}\label{eqn:recolle2}
\Hom(X, Y) \to \Hom(\Pi(X), \Pi(Y))
\end{equation}
is an isomorphism.  Now let $Y' = \Pi(Y)$.  Then~\eqref{eqn:recolle2} can be rewritten as a natural isomorphism
\[
\Hom(X,\pir(Y')) \cong \Hom(\Pi(X), Y'),
\]
so $\pir$ is right adjoint to $\Pi$.

\textit{Step 2. The functor $\Pi$ admits a left adjoint $\pil$.}
This is very similar to Step~1. Let $\sT^\Delta_i \subset \sT_{\le i}$ be the full subcategory generated by objects of the form $\Delta_i\la k\ra$ with $k \in \Z$, and then define $\pil$ to be the composition
\[
\sT_{\le i}/\sT_{< i} \xrightarrow{(\Pi|_{\sT^\Delta_i})^{-1}} \sT^\Delta_i
\xrightarrow{\text{inclusion}}
\sT_{\le i}.
\]
We omit further details.

\textit{Step 3. For $X \in \sT_{\le i}/\sT_{< i}$, the adjunction maps $\Pi(\pir(X)) \to X$ and $X \to \Pi(\pil(X))$ are isomorphisms.}
This is immediate from the construction of $\pir$ and $\pil$.

\textit{Step 4. The functor $\iota$ admits a right adjoint $\ir$.  Moreover, for any $X \in \sT_{\le i}$, there is a functorial distinguished triangle
\[
\iota\ir(X) \to X \to \pir\Pi(X) \xrightarrow{[1]},
\]
where the first two maps are adjunction maps.}
Complete the adjunction map $X \to \pir\Pi(X)$ to a distinguished triangle $X' \to X \to \pir\Pi(X) \to$, and then apply $\Pi$:
\[
\Pi(X') \to \Pi(X) \to \Pi(\pir(\Pi(X))) \xrightarrow{[1]}.
\]
Step~3 implies that $\Pi(X) \to \Pi(\pir(\Pi(X)))$ is an isomorphism, so $\Pi(X') = 0$.  We conclude that $X'$ lies in $\sT_{<i}$.  We may rewrite it as $X' = \iota(X')$.  By adjunction, we have
\[
\Hom(\iota(X'), \pir\Pi(X)[-1]) = 0.
\]
Then~\cite[Proposition~1.1.9]{bbd} (see also~\cite[Corollaire~1.1.10]{bbd}) implies that the triangle $\iota(X') \to X \to \pir\Pi(X) \xrightarrow{[1]}$ is functorial in $X$.  In particular, there is a functor $\ir: \sT_{\le i} \to \sT_{<i}$ such that $X' = \ir(X)$.

Now let $Y \in \sT_{<i}$, and apply $\Hom(\iota Y,{-})$ to our distinguished triangle $\iota\ir(X) \to X \to \pir\Pi(X) \to$.  We obtain the long exact sequence
\begin{multline*}
\cdots \to \Hom(\iota Y, \pir\Pi(X)[-1]) \to \Hom(\iota Y, \iota\ir(X)) \to \Hom(\iota Y,X) \\
\to \Hom (\iota Y, \pir\Pi(X)) \to \cdots.
\end{multline*}
The first and last terms vanish, so the middle two are naturally isomorphic.  This shows that $\ir$ is right adjoint to $\iota$.

\textit{Step 5. The functor $\iota$ admits a right adjoint $\il$.  Moreover, for any $X \in \sT_{\le i}$, there is a functorial distinguished triangle
\[
\pil\Pi(X) \to X \to \iota\il(X) \xrightarrow{[1]},
\]
where the first two maps are adjunction maps.}
This is very similar to Step~4 and is left to the reader.

\textit{Step 6. For $X \in \sT_{< i}$, the adjunction maps $X \to \ir\iota(X)$ and $\il\iota(X) \to X$ are isomorphisms.}  For the first claim, it is enough to prove that $\iota(X) \to \iota \ir\iota(X)$ is an isomorphism.  Since the composition $\iota(X) \to \iota\ir\iota(X) \to \iota(X)$ is the identity map, we may instead show that $\iota\ir\iota(X) \to \iota(X)$ is an isomorphism.  For this, we apply the distinguished triangle from Step~4 and use the observation that $\Pi\iota(X) = 0$.  The proof of the second claim is similar.

We have now checked all the conditions in~\cite[\S1.4.3]{bbd}, so the proof is complete.
\end{proof}

\begin{thm}\label{thm:exc-tstruc}
Let $\sT$ be an $\bE$-linear triangulated category with a Tate twist.  Assume that $\sT$ is graded $\Hom$-finite, and that it is equipped with a dualizable graded exceptional sequence $\{\nabla_i\}_{i \in I}$.  For each $i \in I$, let
\[
\bar\nabla_i =
\begin{cases}
\mathrm{cone}(\nabla_i \xrightarrow{\varpi\cdot id} \nabla_i) & \text{if $\bE$ is not a field,} \\
0 & \text{if $\bE$ is a field.}
\end{cases}
\]
Then the categories $(\sT^{\le 0}, \sT^{\ge 0})$ given by
\begin{align*}
\sT^{\le 0} &=
\begin{array}{@{}l@{}}
\text{the subcategory generated under extensions by objects of the form}\\
\text{$\Delta_j[n]\la k\ra$ with $j \in I$, $n \ge 0$, and $k \in \Z$,}
\end{array}
\\
\sT^{\ge 0} &=
\begin{array}{@{}l@{}}
\text{the subcategory generated under extensions by objects of the form}\\
\text{$\nabla_j[n]\la k\ra$ and $\bar\nabla_j[n]\la k\ra$ with $j \in I$, $n \ge 0$ and $k \in \Z$.}
\end{array}
\end{align*}
form a t-structure on $\sT$.  The heart $\sA = \sT^{\le 0} \cap \sT^{\ge 0}$ of this t-structure is noetherian.  If $\bE$ is a field, then $\sA$ is both noetherian and artinian.
\end{thm}
\begin{proof}
Given $i \in I$, let $\sT^{\le 0}_{\le i}$ and $\sT^{\ge 0}_{\le i}$ be defined as above, but allowing only $\Delta_j$, $\nabla_j$, and $\bar\nabla_j$ with $j \le i$.

We will first show that $(\sT^{\le 0}_{\le i}, \sT^{\ge 0}_{\le i})$ is a t-structure on $\sT_{\le i}$ whose heart is noetherian.  We proceed by induction on $i$.  If $i$ is the minimal element of $I$, the claim holds by Proposition~\ref{prop:stratum-tstruc}.  Suppose now that $i$ is not minimal.  The claim holds for $(\sT^{\le 0}_{<i}, \sT^{\ge 0}_{< i})$ by induction.  We can also equip the quotient category $\sT_{\le i}/\sT_{<i}$ with a t-structure by Proposition~\ref{prop:stratum-tstruc}, using $N = \Pi_i(\Delta_i) \cong \Pi_i(\nabla_i)$.  That proposition tells us that the heart is noetherian; if $\bE$ is a field, it is also artinian.

By recollement, the following categories give a t-structure on $\sT_{\le i}$:
\begin{align*}
'\sT^{\le 0}_{\le i} &= \{ X \in \sT_{\le i} \mid \text{$\il_i(X) \in \sT^{\le 0}_{< i}$ and $\Pi(X) \in (\sT_{\le i}/\sT_{<i})^{\le 0}$} \}, \\
'\sT^{\ge 0}_{\le i} &= \{ X \in \sT_{\le i} \mid \text{$\ir_i(X) \in \sT^{\ge 0}_{< i}$ and $\Pi(X) \in (\sT_{\le i}/\sT_{<i})^{\ge 0}$} \}.
\end{align*}
By Lemma~\ref{lem:glue-noeth}, the heart of $('\sT^{\le 0}_{\le i}, '\sT^{\ge 0}_{\le i})$ is noetherian (and artinian if $\bE$ is a field). It remains to prove that $'\sT^{\le 0}_{\le i} = \sT^{\le 0}_{\le i}$ and $'\sT^{\ge 0}_{\le i} = \sT^{\ge 0}_{\le i}$.  If $X \in{} '\sT^{\le 0}_{\le i}$, consider the distinguished triangle
\[
\iota_i\il_i(X) \to X \to \pil_i\Pi(X) \to.
\]
The first term clearly lies in $\sT^{\le 0}_{\le i}$.  The explicit construction of $\pil_i$ in Proposition~\ref{prop:exc-recolle} shows that the last term does as well.  We conclude that $'\sT^{\le 0}_{\le i} \subset \sT^{\le 0}_{\le i}$.  For the opposite containment, it is enough to check that $\Delta_j \in{} '\sT^{\le 0}_{\le i}$ for all $j \le i$.  This is clear if $j < i$, and it again follows from the construction in Proposition~\ref{prop:exc-recolle} for $j = i$.  The proof that $'\sT^{\ge 0}_{\le i} = \sT^{\ge 0}_{\le i}$ is similar and will be omitted.

By construction, we have
\[
\sT^{\le 0} = \bigcup_{i \in I} \sT^{\le 0}_{\le i}
\qquad\text{and}\qquad
\sT^{\ge 0} = \bigcup_{i \in I} \sT^{\ge 0}_{\le i}.
\]
Since every object of $\sT$ belongs to some $\sT_{\le i}$, it is easy to see that $(\sT^{\le 0}, \sT^{\ge 0})$ is indeed a t-structure.  Its heart is a union of noetherian abelian categories, so it is noetherian (and, similarly, also artinian if $\bE$ is a field).
\end{proof}

\begin{rmk}\phantomsection
\label{rmk:thm-t-structure}
\begin{enumerate}
\item
For the applications in the present paper, the t-structures arising from Theorem~\ref{thm:exc-tstruc} have the following important additional property: the $\Delta_i$ and $\nabla_i$ lie in the heart.  In the case where $\bE$ is a field, it is well known that this implies that the heart is a highest weight category.
\item
\label{it:independence}
It is clear from construction that the t-structure considered in Theorem~\ref{thm:exc-tstruc} does not depend on the order $\leq$ on $I$.
\end{enumerate}
\end{rmk}

\subsection{Special objects in the heart}
\label{ss:special-heart}

By adjunction, the isomorphism $\Pi_i(\Delta_i) \cong \Pi_i(\nabla_i)$ gives rise to a canonical map $\Delta_i \to \nabla_i$.  Let
\begin{equation}\label{eqn:li-defn}
L_i = \im(H^0(\Delta_i) \to H^0(\nabla_i)).
\end{equation}
In the case where $\bE$ is a field,~\cite[Proposition~1.4.26]{bbd} tells us that up to Tate twist, the $L_i$ are precisely the simple objects of $\sA$. 

In the case where $\bE$ is a complete discrete valuation ring, recall that an object $X$ in an $\bE$-linear abelian category is said to be \emph{torsion} if $\varpi^n \cdot \id_X = 0$ for some $n \ge 1$, and \emph{torsion-free} if $\varpi \cdot \id_X$ is injective.  Note that if $X$ is torsion-free, then for any other object $Y$, $\Hom(Y,X)$ is a torsion-free $\bE$-module.

\begin{lem}\label{lem:head-socle}
Assume that $\bE$ is a field.  Then $H^0(\Delta_i)$ has a simple head, and $H^0(\nabla_i)$ has a simple socle (both isomorphic to $L_i$).
\end{lem}
This is a standard fact in the theory of recollement. For a proof, see~\cite[Proposition~2.28]{juteau}.

\begin{lem}\label{lem:li-tf}
Assume that $\bE$ is a complete discrete valuation ring. For each $i \in I$, the objects $H^0(\nabla_i)$ and $L_i$ are torsion-free.
\end{lem}

\begin{proof}
Both $\nabla_i$ and $\bar\nabla_i$ lie in $\sT^{\ge 0}$, so the long exact sequence in cohomology associated with the triangle $\nabla_i \xrightarrow{\varpi} \nabla_i \to \bar\nabla_i \xrightarrow{[1]}$ shows that $\varpi\cdot\id: H^0(\nabla_i) \to H^0(\nabla_i)$ is injective.  Since $L_i$ is a subobject of a torsion-free object, it is torsion-free as well.
\end{proof}

\begin{lem}\label{lem:liplus}
Assume that $\bE$ is a complete discrete valuation ring. For each $i$, there is a unique maximal subobject $L_i^+ \subset H^0(\nabla_i)$ that contains $L_i$, and such that $L_i^+/L_i$ is torsion.  Moreover, $L_i^+$ and $H^0(\nabla_i)/L_i^+$ are both torsion-free.
\end{lem}
\begin{proof}
Since $\sA$ is noetherian, the existence of $L_i^+$ is a consequence of the following observation: if $M, M' \subset H^0(\nabla_i)$ are two subobjects that both contain $L_i$ and such that $M/L_i$ and $M'/L_i$ are both torsion, then $(M+M')/L_i$ is again torsion.

Since $L_i^+$ is a subobject of a torsion-free object, it is torsion-free.  If $H^0(\nabla_i)/L_i^+$ were not torsion-free, it would have a nonzero torsion subobject $M$.  The preimage of $M$ in $H^0(\nabla_i)$ would enjoy the defining properties of $L_i^+$, contradicting the maximality of $L_i^+$.
\end{proof}

\subsection{Change of scalars}
\label{ss:exc-scalar}

Let $\O$ be a complete discrete valuation ring, with fraction field $\K$ and residue field $\F$.  Assume that we are given the following:
\begin{enumerate}
\item a graded $\Hom$-finite $\O$-linear triangulated category $\sT_\O$ with a graded exceptional sequence $\{\nabla^\O_i\}_{i \in I}$ and with a dual sequence $\{\Delta^\O_i\}_{i \in I}$;
\item a graded $\Hom$-finite $\K$-linear triangulated category $\sT_\K$, and a triangulated functor $\K({-}): \sT_\O \to \sT_\K$ which induces an isomorphism
\[
\K \otimes_\O \Hom(X,Y) \cong \Hom(\K(X),\K(Y))
\]
for all $X, Y \in \sT_\O$;
\item a graded $\Hom$-finite $\F$-linear triangulated category $\sT_\F$, and a triangulated functor $\F({-}): \sT_\O \to \sT_\F$ such that for all $X, Y \in \sT_\O$, there is a natural short exact sequence
\begin{equation}\label{eqn:hom-o-f-ses}
\F \otimes_\O \Hom(X,Y) \hookrightarrow \Hom(\F(X), \F(Y)) \twoheadrightarrow \Tor^\O_1(\F,\Hom(X,Y[1]))
\end{equation}
where the first map induced by the functor $\F(-)$.
\end{enumerate}
To this, we add the following assumption:
\begin{enumerate}
\setcounter{enumi}{3}
\item For $\bk \in \{\K,\F\}$, the sequence $\{ \bk(\nabla^\O_i)\}_{i \in I}$ is a graded exceptional sequence in $\sT_\bk$.
\end{enumerate}

\begin{lem}
\label{lem:exc-sequence-bk}
For $\bk \in \{\K,\F\}$, let
\[
\nabla_i^\bk := \bk(\nabla_i^\O)
\qquad\text{and}\qquad
\Delta_i^\bk := \bk(\Delta_i^\O).
\]
Then $\{\Delta^\bk_i\}_{i \in I}$ is a dual sequence to $\{\nabla^\bk_i\}$.
\end{lem}

\begin{proof}
The fact that $\Delta^\bk_i \in \sT_{\bk,\le i}$ and that $\Hom(\Delta^\bk_i, \nabla^\bk_j[n]\la k\ra) = 0$ for $i > j$ follow from the corresponding facts over $\O$.  Next, note that $\bk({-})$ induces a functor of quotient categories $\sT_{\O,\le i}/\sT_{\O,< i} \to \sT_{\bk,\le i}/\sT_{\bk,< i}$. We then deduce the fact that $\Pi_i(\Delta^\bk_i) \cong \Pi_i(\nabla^\bk_i)$ from the commutativity of the following diagram:
\[
\begin{tikzcd}
\sT_{\O,\le i} \ar[r] \ar[d] & \sT_{\O,\le i}/\sT_{\O,< i} \ar[d] \\
\sT_{\bk,\le i} \ar[r] & \sT_{\bk,\le i}/\sT_{\bk,< i}
\end{tikzcd}
\qedhere
\]
\end{proof}

Thanks to Lemma~\ref{lem:exc-sequence-bk}, each of $\sT_\O$, $\sT_\K$, and $\sT_\F$ is equipped with a t-structure provided by Theorem~\ref{thm:exc-tstruc}.  Denote their hearts by $\sA_\O$, $\sA_\K$, and $\sA_\F$, respectively. 

\begin{lem}\phantomsection
\label{lem:exc-scalar}
\begin{enumerate}
\item 
\label{it:exc-scalar-K}
The functor $\K({-}): \sT_\O \to \sT_\K$ is t-exact.  For $X \in \sA_\O$, we have $\K(X) = 0$ if and only if $X$ is torsion.
\item 
\label{it:exc-scalar-f}
The functor $\F({-}): \sT_\O \to \sT_\F$ is right t-exact.  For $X \in \sA_\O$, we have $H^i(\F(X)) = 0$ for all $i \le -2$.  Moreover, $\F(X) \in \sA_\F$ if and only if $X$ is torsion-free.
\end{enumerate}
\end{lem}

\begin{proof}
The t-exactness properties of $\K({-})$ and $\F({-})$ follow immediately from their behavior on the exceptional sequence and its dual, combined with the description of the t-structures from Theorem~\ref{thm:exc-tstruc}.

Next, let $X \in \sA_\O$, and consider the map $\O\cdot\id \to \End(X)$.  If $X$ is torsion, then after we tensor with $\K$, we get the zero map $\K\cdot\id \to \K\otimes\End(X) \cong \End(\K(X))$.  That is, the identity map of $\K(X)$ is zero, so $\K(X) = 0$.  Conversely, if $X$ is not torsion, the map $\O\cdot\id \to \End(X)$ is injective, and hence so is $\K\cdot\id \to \End(\K(X))$.  Since $\End(\K(X)) \ne 0$, we have $\K(X) \ne 0$.

To show that $H^i(\F(X)) = 0$ for $i \le -2$, or equivalently that $\F(X) \in \sT^{\ge -1}$, it is enough to show that $\Hom(\Delta_i^\F[n]\la k\ra, \F(X)) =0$ for $n \ge 2$.  This follows from~\eqref{eqn:hom-o-f-ses} and the fact that $\Hom(\Delta_i^\O[n]\la k\ra, X) = 0$ for $n \ge 1$.

Finally, if $X$ is torsion-free, then $\Hom(\Delta_i^\O\la k\ra,X) \cong \Hom(H^0(\Delta_i^\O)\la k\ra,X)$ is a torsion-free $\O$-module, so $\Tor_1^\O(\F, \Hom(\Delta_i^\O[1]\la k\ra, X[1])) = 0$.  We see from~\eqref{eqn:hom-o-f-ses} then that $\Hom(\Delta_i^\F[1]\la k\ra, \F(X)) = 0$, so $\F(X) \in \sA_\F$.  Conversely, if $X$ is not torsion-free, then it has a nonzero torsion subobject $X' \subset X$.  Moreover, $\Hom(X',X)$ is a torsion $\O$-module, so $\Tor_1^\O(\F, \Hom(X'[1],X[1])) \ne 0$. In this case,~\eqref{eqn:hom-o-f-ses} shows that $\Hom(\F(X')[1], \F(X)) \ne 0$, which implies that $H^{-1}(\F(X)) \ne 0$.
\end{proof}

To distinguish the various versions of~\eqref{eqn:li-defn}, we now include the coefficient ring in the notation, as follows:
\[
L_i(\K) \in \sA_\K,
\qquad
L_i(\O), L_i^+(\O) \in \sA_\O,
\qquad
L_i(\F) \in \sA_\F.
\]

\begin{lem}\phantomsection
\label{lem:li-change}
\begin{enumerate}
\item For all $i \in I$, we have\label{it:li-k}
\[
\K(L_i(\O)) \cong \K(L_i^+(\O)) \cong L_i(\K).
\]
\item The objects $\F(L_i(\O))$ and $\F(L_i^+(\O))$ lie in $\sA_\F$.  Moreover, $\F(L_i(\O))$ has a simple head, and $\F(L_i^+(\O))$ has a simple socle, both isomorphic to $L_i(\F)$.\label{it:li-f}
\end{enumerate}
\end{lem}
\begin{proof}
\eqref{it:li-k}~Since $\K({-})$ is t-exact, it commutes with $H^0$, and it takes the image of a morphism in $\sA_\O$ to the image of the corresponding morphism in $\sA_\K$.  It follows immediately that $\K(L_i(\O)) \cong L_i(\K)$.  Next, we have a short exact sequence $0 \to L_i(\O) \to L_i^+(\O) \to T \to 0$, where $T$ is a torsion object.  Since $\K(T) = 0$ by Lemma~\ref{lem:exc-scalar}\eqref{it:exc-scalar-K}, we conclude that $\K(L_i(\O)) \cong \K(L_i^+(\O))$.

\eqref{it:li-f}~The first assertion follows from the fact that $L_i(\O)$ and $L_i^+(\O)$ are both torsion-free (see Lemma~\ref{lem:li-tf} and Lemma~\ref{lem:liplus}) and Lemma~\ref{lem:exc-scalar}\eqref{it:exc-scalar-f}.  By definition, $L_i(\O)$ is a quotient of $H^0(\Delta_i^\O)$.  Since $\F({-})$ is right t-exact, we have an induced surjective map
\[
H^0(\F(H^0(\Delta_i^\O))) \to \F(L_i(\O)).
\]
The right t-exactness of $\F({-})$ also implies that $H^0(\F(H^0(\Delta_i^\O))) \cong H^0(\F(\Delta_i^\O)) \cong H^0(\Delta_i^\F)$.  That is, $\F(L_i(\O))$ is a quotient of $H^0(\Delta_i^\F)$.  Since the latter has a simple head (isomorphic to $L_i(\F)$), so does the former.

Next, we claim that $\F(H^0(\nabla_i^\O))$ is a subobject of $H^0(\nabla_i^\F)$.  (Note that this lies in $\sA_\F$ because $H^0(\nabla_i^\O)$ is torsion-free by Lemma~\ref{lem:li-tf}, see Lemma~\ref{lem:exc-scalar}\eqref{it:exc-scalar-f}.) Indeed, consider the truncation distinguished triangle $H^0(\nabla_i^\O) \to \nabla_i^\O \to \tau^{\ge 1}\nabla_i^\O \to$.  Apply $\F({-})$ to obtain the triangle
\[
\F(H^0(\nabla_i^\O)) \to \nabla_i^\F \to \F(\tau^{\ge 1}\nabla_i^\O) \to.
\]
Lemma~\ref{lem:exc-scalar}\eqref{it:exc-scalar-f} implies that the third term lies in $\sT_\F^{\ge 0}$.  Therefore, the long exact sequence in cohomology shows that we have an injective map $\F(H^0(\nabla_i^\O)) \to H^0(\nabla_i^\F)$.

Finally, consider the short exact sequence
\[
0 \to L_i^+(\O) \to H^0(\nabla_i^\O) \to H^0(\nabla_i^\O)/L_i^+(\O) \to 0.
\]
We have seen in Lemma~\ref{lem:liplus} that all three terms are torsion-free, so applying $\F$ yields a short exact sequence in $\sA_\F$.  In particular, $\F(L_i^+(\O))$ is a subobject of $\F(H^0(\nabla_i^\O))$, and hence (by the previous paragraph) of $H^0(\nabla_i^\F)$.  Since the latter has a simple socle (isomorphic to $L_i(\F)$), so does $\F(L_i^+(\O))$.
\end{proof}


\end{document}